\documentclass[10pt,reqno]{amsart}
\usepackage{hyperref}
\usepackage{latexsym}
\usepackage{amsmath}
\usepackage{enumerate}
\usepackage{amsfonts}
\usepackage{amssymb}
\usepackage{latexsym}
\usepackage{fixmath}
\usepackage[usenames,dvipsnames]{color}
\usepackage{mathdots}
\usepackage{multicol}
\usepackage{cleveref}
\usepackage{graphicx}
\usepackage{xcolor}
\usepackage{subcaption} 
\usepackage[boxed, linesnumbered, noend, noline]{algorithm2e}
\usepackage{comment}
\usepackage{float}

\usepackage[T1]{fontenc}
\usepackage{fourier}
\usepackage{bbm}
\usepackage{color}
\usepackage{fullpage}

\newcommand\lk[1]{\textcolor{magenta}{#1}}
\newcommand\mk[1]{\textcolor{cyan}{#1}}

\renewcommand\lk[1]{#1}
\renewcommand\mk[1]{#1}

\numberwithin{equation}{section}

\newcommand{\ism}{\cong}

	\newcommand{\pastim}{\fF_{t_i-1}}
	\newcommand{\pasthm}{\fF_{t_h-1}}
\newcommand{\tox}{\cC}

\newcommand{\toxl}{\tox}

\newcommand{\everrti}{\cR_{t_i}}
\newcommand{\everrth}{\cR_{t_h}}
\newcommand{\RTS}{Ricci-Tersenghi and Semerjian}
\renewcommand{\ln}{\log}
\renewcommand{\vec}[1]{\boldsymbol{#1}}

\renewcommand{\subset}{\subseteq}

\newcommand\dmin{d_{\mathrm{min}}}
\newcommand\dcore{d_{\mathrm{core}}}
\newcommand\dsat{d_{\mathrm{sat}}}

\newcommand{\FBP}[1]{\PHI_{\mathrm{BP},{#1}}}
\newcommand{\FUC}[1]{\PHI_{\mathrm{UC},{#1}}}
\newcommand{\FDC}[1]{\PHI_{\mathrm{DC},{#1}}}
\newcommand{\ADC}[1]{\vA_{#1}}
\newcommand{\amax}{\alpha_{\max}}
\newcommand{\lcond}{\lambda_{\mathrm{cond}}}
\newcommand{\tcond}{\theta_{\mathrm{cond}}}
\newcommand{\frz}{V_0}

\newcommand{\frozen}{\mathtt{f}}
\newcommand{\unfrozen}{\mathtt{u}}
\newcommand{\nll}{\mathtt{n}}
\newcommand{\fzn}{\frozen}
\newcommand{\uzn}{\unfrozen}

\newcommand\disteq{\stacksign{dist}{=}}

\newcommand{\BPGD}{\ensuremath{\mathtt{BPGD}}}
\newcommand{\UCP}{\ensuremath{\mathtt{UCP}}}

\newcommand{\TT}{\mathbb T}

\newcommand\G{\vec G}

\newcommand\PHI{\vec F}
\newcommand\nix{\,\cdot\,}
\newcommand\dd{\partial}

\newcommand\SIGMA{\vec\sigma}
\newcommand\SIGBP{\vec\sigma_{\mathrm{BP}}}
\newcommand\SIGUC{\vec\sigma_{\mathrm{UC}}}
\newcommand\SIGDEC{\vec\sigma_{\mathrm{DC}}}
\newcommand\SIGDC{\SIGDEC}
\newcommand\TAU{\vec\tau}
\newcommand\cA{\mathcal A}

\newcommand\cC{\mathcal C}
\newcommand\cD{\mathcal D}
\newcommand\cE{\mathcal E}

\newcommand\cG{\mathcal G}

\newcommand\cQ{\mathcal Q}
\newcommand\cR{\mathcal R}

\newcommand\fA{\mathfrak A}

\newcommand\fC{\mathfrak C}
\newcommand\fD{\mathfrak D}

\newcommand\fF{\mathfrak F}
\newcommand\fG{\mathfrak G}

\newcommand\fK{\mathfrak K}

\newcommand\fS{\mathfrak S}

\newcommand\fV{\mathfrak V}

\newcommand\fc{\mathfrak c}
\newcommand\fd{\mathfrak d}

\newcommand\fk{\mathfrak k}
\newcommand\fl{\mathfrak l}
\newcommand\fm{\mathfrak m}
\newcommand\fn{\mathfrak n}

\newcommand\fx{\mathfrak x}

\newcommand\vA{\vec A}

\newcommand\vC{\vec C}

\newcommand\vG{\vec G}

\newcommand\vN{\vec N}

\newcommand\vR{\vec R}

\newcommand\vT{\vec T}
\newcommand\vU{\vec U}
\newcommand\vV{\vec V}

\newcommand\vX{\vec X}
\newcommand\vY{\vec Y}

\newcommand\vd{\vec d}

\newcommand\vh{\vec h}

\newcommand\vj{\vec j}

\newcommand\vl{\vec l}
\newcommand\vm{\vec m}
\newcommand\vn{\vec n}

\newcommand\vu{\vec u}

\newcommand\vx{\vec x}

\newcommand\eps{\varepsilon}

\newcommand\FF{\mathbb{F}}

\newcommand\Erw{\mathbb{E}}
\newcommand\ex{\Erw}
\newcommand{\vecone}{\mathbb{1}}

\newcommand{\Po}{{\rm Po}}
\newcommand{\Bin}{{\rm Bin}}

\newcommand\bc[1]{\left({#1}\right)}
\newcommand\cbc[1]{\left\{{#1}\right\}}
\newcommand\bcfr[2]{\bc{\frac{#1}{#2}}}

\newcommand\brk[1]{\left\lbrack{#1}\right\rbrack}

\newcommand\abs[1]{\left|{#1}\right|}

\newcommand\RR{\mathbb{R}}

\newcommand{\Whp}{W.h.p.}
\newcommand{\whp}{w.h.p.}

\newcommand{\stacksign}[2]{{\stackrel{\mbox{\scriptsize #1}}{#2}}}

\newcommand\pr{\mathbb{P}} 
\renewcommand\Pr{\pr} 
\newcommand\Lem{Lemma}
\newcommand\Prop{Proposition}
\newcommand\Thm{Theorem}

\newcommand\Cor{Corollary}
\newcommand\Sec{Section}

\newtheorem{definition}{Definition}[section]

\newtheorem{theorem}[definition]{Theorem}
\newtheorem{lemma}[definition]{Lemma}
\newtheorem{proposition}[definition]{Proposition}
\newtheorem{corollary}[definition]{Corollary}

\newtheorem{fact}[definition]{Fact}

\DeclareMathOperator{\nul}{nul}

\def\pr{{\mathbb P}}

\newcommand{\la}{\lambda}
\newcommand{\ph}{\phi_{d,k,\lambda}}
\newcommand{\Ph}{\Phi_{d,k,\lambda}}
\newcommand{\zt}{\zeta_\lambda}

\newcommand{\expld}{\exp(-\lambda-d z^{k-1})}

\makeatletter
\newcommand{\definetitlefootnote}[1]{%
	\newcommand\addtitlefootnote{%
		\makebox[0pt][l]{$^{*}$}%
		\footnote{\protect\@titlefootnotetext}
	}%
	\newcommand\@titlefootnotetext{\spaceskip=\z@skip $^{*}$#1}%
}
\makeatother

\begin{document}
\bibliographystyle{plainurl}
\definetitlefootnote{An Extended abstract appeared in the proceedings of ICALP 2025}
\title{Belief Propagation guided decimation on random $k$-XORSAT\addtitlefootnote}
\author{Arnab Chatterjee, Amin Coja-Oghlan, Mihyun Kang, Lena Krieg, Maurice Rolvien, Gregory B. Sorkin}

\address{Arnab Chatterjee, {\tt arnab.chatterjee@tu-dortmund.de}, TU Dortmund, Faculty of Computer Science, 12 Otto-Hahn-St, 44227 Dortmund, Germany.}
\address{Amin Coja-Oghlan, {\tt amin.coja-oghlan@tu-dortmund.de}, TU Dortmund, Faculty of Computer Science and Faculty of Mathematics, 12 Otto-Hahn-St, 44227 Dortmund, Germany.}
\address{Mihyun Kang, {\tt kang@math.tugraz.at}, TU Graz, Institute of Discrete Mathematics, Steyrergasse 30, 8010 Graz, Austria.}
\address{Lena Krieg, {\tt lena.krieg@tu-dortmund.de}, TU Dortmund, Faculty of Computer Science, 12 Otto-Hahn-St, 44227 Dortmund, Germany.}
\address{Maurice Rolvien, {\tt maurice.rolvien@tu-dortmund.de}, University of Hamburg, Department of Informatics, Vogt-Kölln-Str. 30, 22527 Hamburg, Germany.}
\address{Gregory B. Sorkin, {\tt g.b.sorkin@lse.ac.uk}, The London School of Economics and Political Science, Department of Mathematics, Columbia House, Houghton St, London WC2A 2AE, United Kingdom}
\maketitle
\begin{abstract}
	We analyse the performance of {\em Belief Propagation Guided Decimation}, a physics-inspired message passing algorithm, on the random $k$-XORSAT problem.
	Specifically, we derive an explicit threshold up to which the algorithm succeeds with a strictly positive probability $\Omega(1)$ that we compute explicitly, but beyond which the algorithm with high probability fails to find a satisfying assignment.
	In addition, we analyse a thought experiment called the {\em decimation process} for which we identify a (non-)reconstruction and a condensation phase transition.
	The main results of the present work confirm physics predictions from [\RTS: J.\ Stat.\ Mech.\ 2009] that link the phase transitions of the decimation process with the performance of the algorithm, and improve over partial results from a recent article [Yung: Proc.\ ICALP 2024].
\hfill MSc: 60B20, 68W20
\end{abstract}
\section{Introduction and results}\label{Sec_not_intro}

\subsection{Background and motivation}\label{sec_background}
The random $k$-XORSAT problem shares many characteristics of other intensely studied random constraint satisfaction problems (`CSPs') such as random $k$-SAT.
For instance, as the clause/variable density increases, random $k$-XORSAT possesses a sharp satisfiability threshold preceded by a reconstruction or `shattering' phase transition that affects the geometry of the set of solutions~\cite{AchlioptasMolloy,DuboisMandler,pnas,PittelSorkin}.
As in random $k$-SAT, these transitions appear to significantly impact the performance of certain classes of algorithms~\cite{BetterAlg,Ibrahimi}.
At the same time, random $k$-XORSAT is more amenable to mathematical analysis than, say, random $k$-SAT.
This is because the XOR operation is equivalent to addition modulo two, which is why a $k$-XORSAT instance translates into a linear system over $\FF_2$.
In effect, $k$-XORSAT can be solved in polynomial time by means of Gaussian elimination.
In addition, the algebraic nature of the problem induces strong symmetry properties that simplify its study~\cite{Ayre}.

Because of its similarities with other random CSPs combined with said relative amenability, random $k$-XORSAT provides an instructive benchmark.
This was noticed not only in combinatorics, but also in the statistical physics community, which has been contributing intriguing `predictions' on random CSPs since the early 2000s~\cite{MM,MRTZ}.
Among other things, physicists have proposed a message passing algorithm called {\em Belief Propagation Guided Decimation} (`\BPGD') that, according to computer experiments, performs impressively on various random CSPs~\cite{MPZ}.
Furthermore, \RTS~\cite{RTS} put forward a heuristic analysis of \BPGD\ on random $k$-SAT and $k$-XORSAT.
Their heuristic analysis proceeds by way of a thought experiment based on an idealised version of the algorithm.
We call this thought experiment the {\em decimation process}.
Based on physics methods \RTS\ surmise that the decimation process undergoes two phase transitions, specifically a reconstruction and a condensation transition.
A key prediction of \RTS\ is that these phase transitions are directly linked to the performance of the \BPGD\ algorithm.
Due to the linear algebra-induced symmetry properties, in the case of random $k$-XORSAT all of these conjectures come as elegant analytical expressions.

The aim of this paper is to verify the predictions from~\cite{RTS} on random $k$-XORSAT mathematically.
Specifically, our aim is to rigorously analyse the \BPGD\ algorithm on random $k$-XORSAT, and to establish the link between its performance and the phase transitions of the decimation process.
A first step towards a rigorous analysis of \BPGD\ on random $k$-XORSAT was undertaken in a recent contribution by Yung~\cite{Yung}.
However, Yung's analysis turns out to be not tight.
Specifically, apart from requiring spurious lower bounds on the clause length $k$, Yung's results do not quite establish the precise connection between the decimation process and the performance of \BPGD.
One reason for this is that~\cite{Yung} relies on `annealed' techniques, i.e., essentially moment computations.
Here we instead harness `quenched' arguments that were partly developed in prior work on the rank of random matrices over finite fields~\cite{Ayre,Maurice}.

Throughout we let $k\geq3$ and $n\geq k$ be integers and $d>0$ a positive real.
Let $\vm\disteq\Po(dn/k)$ and let $\PHI=\PHI(n,d,k)$ be a random $k$-XORSAT formula\footnote{
		Two random variables $\vX, \vY$ are equal in distribution $\vX\disteq \vY$ if they have the same distribution functions. 
		Thus, $\vm$ follows a Poisson distribution with mean $dn/k$.
		}
	with variables $x_1,\ldots,x_n$ and $\vm$ random clauses of length $k$.
	
To be precise, every clause of $\PHI$ is an XOR of precisely $k$ distinct variables, each of which may or may not come with a negation sign.
The $\vm$ clauses are drawn uniformly and independently out of the set of all $2^k\binom nk$ possibilities.
Thus, $d$ equals the average number of clauses that a given variable $x_i$ appears in.
An event $\cE$ occurs \emph{with high probability} (`\whp') if $\lim_{n\to\infty}\pr\brk{\cE}=1$.
We always keep $d,k$ fixed as $n\to\infty$.

\subsection{Belief Propagation Guided Decimation}\label{sec_bpgd}
The first result vindicates the predictions from~\cite{RTS} concerning the success probability of \BPGD\ algorithm.
\BPGD\ sets its ambitions higher than merely finding a solution to the $k$-XORSAT instance $\PHI$: the algorithm attempts to sample a solution uniformly at random.
To this end \BPGD\ assigns values to the variables $x_1,\ldots,x_n$ of $\PHI$ one after the other.
In order to assign the next variable the algorithm attempts to compute the marginal probability that the variable is set to `true' under a random solution to the $k$-XORSAT instance, given all previous assignments.
More precisely, suppose \BPGD\ has assigned values to the variables $x_1,\ldots,x_t$ already.
Write $\SIGBP(x_1),\ldots,\SIGBP(x_t)\in\{0,1\}$ for their values, with $1$ representing `true' and $0$ `false'.
Further, let $\FBP{t}$ be the simplified formula obtained by substituting $\SIGBP(x_1),\ldots,\SIGBP(x_t)$ for $x_1,\ldots,x_t$.
We drop any clauses from $\FBP{t}$ that contain variables from $\{x_1,\ldots,x_t\}$ only, deeming any such clauses satisfied.
Thus, $\FBP{t}$ is a XORSAT formula with variables $x_{t+1},\ldots,x_n$.
Its clauses contain at least one and at most $k$ variables, as well as possibly a constant (the XOR of the values substituted in for $x_1,\ldots,x_t$).

Let $\SIGMA_{\FBP{t}}$ be a uniformly random solution of the XORSAT formula $\FBP{t}$, assuming that $\FBP{t}$ remains satisfiable.
Then \BPGD\ aims to compute the marginal probability $\pr\brk{\SIGMA_{\FBP{t}}(x_{t+1})=1\mid\FBP{t}}$ that a random satisfying assignment of $\FBP{t}$ sets $x_{t+1}$  to true. 
This is where Belief Propagation (`BP') comes in.
An efficient message passing heuristic for computing precisely such marginals, BP returns an `approximation' $\mu_{\FBP{t}}$ of $\pr\brk{\SIGMA_{\FBP{t}}(x_{t+1})=1\mid\FBP{t}}$.
We will recap the mechanics of BP in \Sec~\ref{sec_bp} (the value $\mu_{\FBP{t}}$ is defined precisely in \eqref{eqmuBPGD}).
Having computed the BP `approximation', \BPGD\ proceeds to assign $x_{t+1}$ the value `true' with probability $\mu_{\FBP{t}}$, otherwise sets $x_{t+1}$ to `false', then moves on to the next variable.
The pseudocode is displayed as Algorithm~\ref{alg_bpgd}.

\begin{algorithm}[h!]
 \KwData{a random $k$-XORSAT formula $\PHI$ with variables $x_1,\ldots,x_n$ conditioned on being satisfiable} 
 \For{$t=0,\ldots,n-1$}{compute the BP approximation $\mu_{\FBP{t}}$\;
		set
		$\SIGBP(x_{t+1})=\begin{cases}
			1&\mbox{with probability }\mu_{\FBP{t}}\\
			0&\mbox{with probability }1-\mu_{\FBP{t}}
					\end{cases}$\;
		}
	\Return $\SIGBP$\;
 \caption{The \BPGD\ algorithm.}\label{alg_bpgd}
 \label{Alg_SC}
\end{algorithm}

Let us pause for a few remarks.
First, if the BP approximations are exact, i.e., if $\FBP{t}$ is satisfiable and $\mu_{\FBP{t}}=\pr\brk{\SIGMA_{\FBP{t}}(x_{t+1})=1\mid\FBP{t}}$ for all $t$, then Bayes' formula shows that \BPGD\ outputs a uniformly random solution of $\PHI$.
However, there is no universal guarantee that BP returns the correct marginals.
Accordingly, the crux of analysing \BPGD\ is precisely to figure out whether this is the case.
Indeed, the heuristic work of~\cite{RTS} ties the accuracy of BP to a phase transition of the decimation process thought experiment, to be reviewed momentarily.

Second, the strategy behind the \BPGD\ algorithm, particularly the message passing heuristic for `approximating' the marginals, generalises well beyond $k$-XORSAT.
For instance, the approach applies to $k$-SAT verbatim.
That said, due to the algebraic nature of the XOR operation, \BPGD\ is {\em far} easier to analyse on $k$-XORSAT.
In fact, in XORSAT the marginal probabilities are guaranteed to be half-integral as seen in Fact \ref{fact_halfint}, i.e., 
	\begin{align}\label{eqHalfInt}
		\pr\brk{\SIGMA_{\FBP{t}}(x_{t+1})=1\mid\FBP{t}}\in\{0,1/2,1\}.
	\end{align}
As a consequence, on XORSAT the \BPGD\ algorithm effectively reduces to a purely combinatorial algorithm called Unit Clause Propagation~\cite{MM,RTS} as per \Prop\ \ref{prop_UCP}, a fact that we will exploit extensively \mk{(see \Sec\ \ref{sec_alg})}.

\subsection{A tight analysis of $\BPGD$}\label{sec_results}
In order to state the main results we need to introduce a few threshold values.
To this end, given $d,k$ and an additional real parameter $\lambda\geq0$ that depends on the time $t$, consider the functions
\footnote{The function $\Ph$ is known in physics parlance as the ``Bethe free entropy'' \cite{Maurice, MM}. The stationary points of $\Ph$ coincide with the fixed points of $\ph$, as we will verify in \Sec\ \ref{sec_calc}.}
\begin{align} \label{eqphi}
	\ph:&[0,1]\to[0,1],& z&\mapsto 1 - \exp\bc{-\lambda-dz^{k-1}},\\
	\Ph:&[0,1]\to\RR,&z &\mapsto \exp\bc{-\lambda-dz^{k-1}}-\frac{d(k-1)}kz^k+dz^{k-1}-\frac dk.\label{eqPhi}
\end{align}
Let $\alpha_*(\lambda)=\alpha_*(d,k,\lambda)\in[0,1]$ be the smallest and $\alpha^*(\lambda)=\alpha^*(d,k,\lambda)\geq\alpha_*(d,k,\lambda)\in[0,1]$ the largest fixed point of $\ph$.
Figure \ref{fig_rainbow} visualises $\Phi(z)$ for different values of $\theta \sim t/n$.
Further, define
\begin{align}\label{eqds} 
	\dmin(k)&=\bcfr{k-1}{k-2}^{k-2},&
	\dcore(k)&=\sup\cbc{d>0:\alpha^*(0)=0},&
	\dsat(k)&=\sup\cbc{d>0:\Phi_{d,k,0}(\alpha^*(0))\leq\Phi_{d,k,0}(0)}.
\end{align}

The value $\dsat(k)$ is the random $k$-XORSAT satisfiability threshold~\cite{Ayre,DuboisMandler,PittelSorkin}.
Thus, for $d<\dsat(k)$ the random $k$-XORSAT formula $\PHI$ possesses satisfying assignments \whp, while $\PHI$ is unsatisfiable for $d>\dsat(k)$ \whp\
Furthermore, $\dcore(k)$ equals the threshold for the emergence of a giant 2-core within the $k$-uniform hypergraph induced by $\PHI$~\cite{Ayre,Molloy}.
This implies that for $d<\dcore(k)$ the set of solutions of $\PHI$ is contiguous in a certain well-defined way, while for $\dcore(k)<d<\dsat(k)$ the set of solutions shatters into an exponential number of well-separated clusters~\cite{Ibrahimi,MM}.
Moreover, a simple linear time algorithm is known to find a solution \whp\ for $d<\dcore(k)$~\cite{Ibrahimi}.
The relevance of $\dmin(k)$ will emerge in \Thm\ \ref{thm_bpgd}. 
A bit of calculus reveals that
\begin{align}\label{eqdrel}
	0<\dmin(k)<\dcore(k)<\dsat(k)<k.
\end{align}

The following theorem determines the precise clause-to-variable densities where \BPGD\ succeeds/fails.
To be precise, in the `successful' regime \BPGD\ does not actually succeed with {\em high} probability, but with an explicit probability strictly between zero and one, which is displayed in Figure~\ref{Psuccess} for $k=3,4,5$.

\noindent
\begin{minipage}{0.5\linewidth}
	\begin{figure}[H]
		\includegraphics[height=40mm]{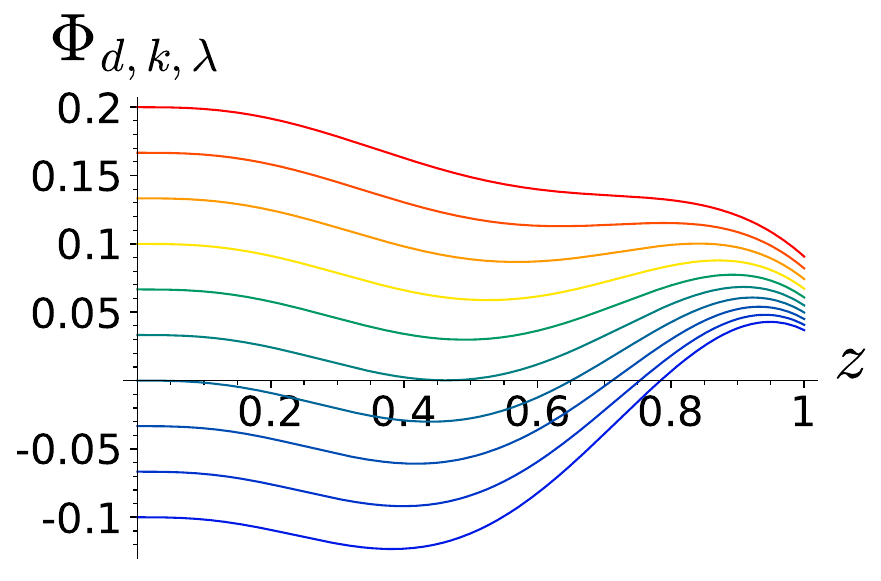}
		\caption{$\Phi_{d,k,\lambda}$ for $k = 3$ and $d = 2.4$, for $\lambda$ from $\textcolor{red}{0}$ to $\textcolor{Dandelion}{0.3}$ (maximum at $z=0$) and from $\color{teal}{0.4}$ to $\textcolor{blue}{0.9}$   }\label{fig_rainbow}
	\end{figure}
\end{minipage}
\begin{minipage}{0.5\linewidth}
	\begin{figure}[H]
		\includegraphics[height=40mm]{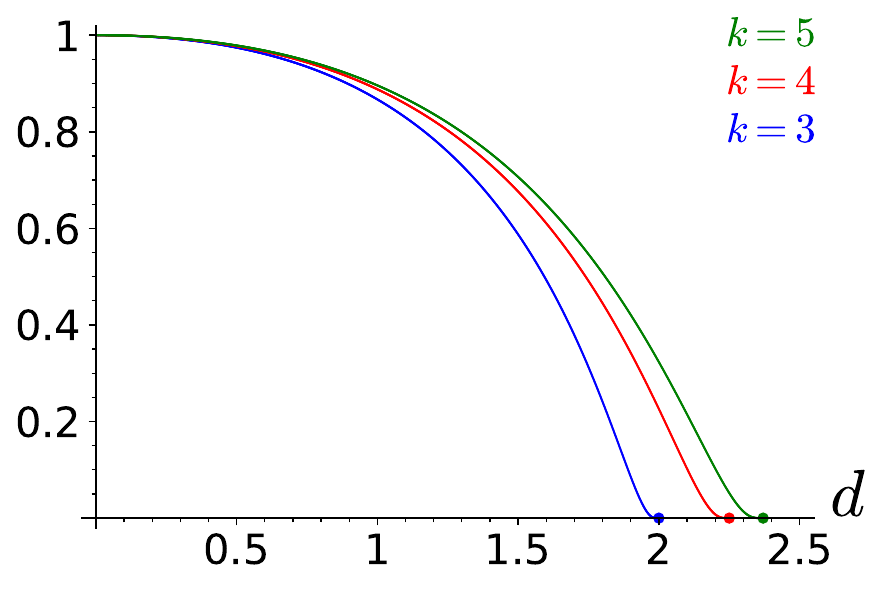}
		\caption{Success probability of \BPGD\ for $0<d<\dmin(k)$ and various $k$.}\label{Psuccess}
	\end{figure}
\end{minipage}

%
%
%

\begin{theorem}\label{thm_bpgd}
	Let $k\geq3$.
	\begin{enumerate}[(i)]
		\item If $d<\dmin(k)$, then 
			\begin{align}\label{eqthm_bpgd}
		\lim_{n\to\infty}\pr\brk{\BPGD(\PHI)\mbox{ finds a satisfying assignment}}&=
		\exp\bc{-\frac{d^2(k-1)^2}4\int_0^1\frac{z^{2k-4}(1-z)}{1-d(k-1)z^{k-2}(1-z)}\,\mathrm d z}.
	\end{align}
\item If $\dmin(k)<d<\dsat(k)$, then  
	$$
	\pr\brk{\BPGD(\PHI)\mbox{ finds a satisfying assignment}}=o(1).
	$$
\end{enumerate}
\end{theorem}

\Thm~\ref{thm_bpgd} vindicates the predictions from \RTS~\cite[\Sec~4]{RTS} as to the performance of \BPGD, and improves over the results from Yung~\cite{Yung}.
Specifically, \Thm~\ref{thm_bpgd}~(i) verifies the formula for the success probability from~\cite[Eq.~(38)]{RTS}. 
Combinatorially, the formula \eqref{eqthm_bpgd} results from the possible presence of bounded length cycles (so-called toxic cycles) that may cause the algorithm to run into contradictions.
This complements Yung's prior work, that has no positive result on the performance of \BPGD. 
Moreover, Yung's negative results~\cite[\Thm s~2--3]{Yung} only apply to $k\geq9$ and to $d>\dcore(k)$, while \Thm~\ref{thm_bpgd}~(ii) covers all $k\geq3$ and kicks in at the correct threshold $\dmin(k)<\dcore(k)$ predicted in~\cite{RTS}.

\subsection{The decimation process}\label{sec_dc}
In addition to the \BPGD\ algorithm itself, the heuristic work~\cite{RTS} considers an idealised version of the algorithm, the {\em decimation process}.
This thought experiment highlights the conceptual reasons behind the success/failure of \BPGD.
Just like \BPGD, the decimation process assigns values to variables one after the other for good.
But instead of the BP `approximations' the decimation process uses the {\em actual} marginals given its previous decisions.
To be precise, suppose that the input formula $\PHI$ is satisfiable and that variables $x_1,\ldots,x_t$ have already been assigned values $\SIGDC(x_1),\ldots,\SIGDC(x_t)$ in the previous iterations.
Obtain $\FDC{t}$ by substituting the values $\SIGDC(x_1),\ldots,\SIGDC(x_t)$ for $x_1,\ldots,x_t$ and dropping any clauses that do not contain any of $x_{t+1},\ldots,x_n$.
Thus, $\FDC t$ is a XORSAT formula with variables $x_{t+1},\ldots,x_n$.
Let $\SIGMA_{\FDC{t}}$ be a random satisfying assignment of $\FDC{t}$.
Then the decimation process sets $x_{t+1}$ according to the true marginal $\pr\brk{\SIGMA_{\FDC{t}}(x_{t+1})=1\mid\FDC t}$, thus ultimately returning a uniformly random satisfying assignment of $\PHI$.

\begin{algorithm}[h!]
 \KwData{a random $k$-XORSAT formula $\PHI$, conditioned on being satisfiable}
 \For{$t=0,\ldots,n-1$}{compute $\pi_{\FDC{t}}=\pr\brk{\SIGMA_{\FDC{t}}(x_{t+1})=1\mid\FDC t}$\;
		set
		$\SIGDC(x_{t})=\begin{cases}
			1&\mbox{with probability }\pi_{\FDC{t}}\\
			0&\mbox{with probability }1-\pi_{\FDC{t}}
					\end{cases}$\;
		}
	\Return $\SIGDC$\;
 \caption{The decimation process.}\label{alg_dec}
 \label{decimationprocess}
\end{algorithm}

Clearly, if indeed the BP `approximations' are correct, then the decimation process and \BPGD\ are identical.
Thus, a key question is for what parameter regimes the two process coincide or diverge, respectively.
As it turns out, this question is best answered by parametrize not only in terms of the average variable degree $d$, but also in terms of the `time' parameter $t$ of the decimation process.

\subsection{Phase transitions of the decimation process}\label{sec_results_dc}
\RTS\ heuristically identify several phase transitions in terms of $d$ and $t$ that the decimation process undergoes.
We will confirm these predictions mathematically and investigate how they relate to the performance of \BPGD.

The first set of relevant phase transitions concerns the so-called non-reconstruction property.
Roughly speaking, non-reconstruction means that the marginal $\pi_{\FDC{t}}=\pr\brk{\SIGMA_{\FDC{t}}(x_{t+1})=1\mid\FDC t}$ is determined by short-range rather than long-range effects.
Since Belief Propagation is essentially a local algorithm, one might expect that the (non-)reconstruction phase transition coincides with the threshold up to which \BPGD\ succeeds; cf.\ the discussions in~\cite{Braunstein,pnas}.

To define (non-)reconstruction precisely, we associate a bipartite graph $G(\FDC{t})$ with the formula $\FDC{t}$.
The vertices of this graph are the variables and clauses of $\FDC{t}$.
Each variable is adjacent to the clauses in which it appears.
For a (variable or clause) vertex $v$ of $G(\FDC{t})$ let $\partial v$ be the set of neighbours of $v$ in $G(\FDC{t})$.
More generally, for an integer $\ell\geq1$ let $\partial^\ell v$ be the set of vertices of $G(\FDC{t})$ at shortest path distance precisely $\ell$ from $v$.
Following~\cite{pnas}, we say that $\FDC{t}$ has the {\em non-reconstruction property} if
\begin{align}\label{eqnonrec}
	\lim_{\ell\to\infty}\limsup_{n\to\infty}\ex\brk{\abs{\pr\brk{\SIGMA_{\FDC{t}}(x_{t+1})=1\,\Big|\,\FDC{t},\,\cbc{\SIGMA_{\FDC{t}}(y)}_{y\in\partial^{2\ell}x_{t+1}}}-\pr\brk{\SIGMA_{\FDC{t}}(x_{t+1})=1\mid\FDC{t}}}\,\big|\,\PHI\mbox{ satisfiable}}&=0.
\end{align}
Conversely, $\FDC t$ has the {\em reconstruction property} if 
\begin{align}\label{eqrec}
	\liminf_{\ell\to\infty}\,\liminf_{n\to\infty}\,\ex\brk{\abs{\pr\brk{\SIGMA_{\FDC{t}}(x_{t+1})=1\,\Big|\,\FDC{t},\,\cbc{\SIGMA_{\FDC{t}}(y)}_{y\in\partial^{2\ell}x_{t+1}}}-\pr\brk{\SIGMA_{\FDC{t}}(x_{t+1})=1\mid\FDC{t}}}\,\big|\,\PHI\mbox{ sat.}}&>0.
\end{align}

To parse \eqref{eqnonrec}, notice that in the left probability term we condition on both the outcome $\FDC{t}$ of the first $t$ steps of the decimation process and on the values $\SIGMA_{\FDC{t}}(y)$ that the random solution $\SIGMA_{\FDC t}$ assigns to the variables $y$ at distance exactly $2\ell$ from $x_{t+1}$.
By contrast, in the right probability term we only condition on $\FDC t$.
Thus, the second probability term matches the probability $\pi_{\FDC{t}}$ from the decimation process.
Hence, \eqref{eqnonrec} compares the probability that a random solution sets $x_{t+1}$ to one given the values $\SIGMA_{\FDC{t}}(y)$ of {\em all} variables $y$ at distance $2\ell$ from $x_{t+1}$ with plain marginal probability that $x_{t+1}$ is set to one.
What~\eqref{eqnonrec} asks is that these two probabilities be asymptotically equal in the limit of large $\ell$, with high probability over the choice of $\PHI$ and the prior steps of the decimation process.
Thus, so long as non-reconstruction holds `long-range effects', meaning anything beyond distance $2\ell$ for large enough but fixed $\ell$, are negligible.

Confirming the predictions from~\cite{RTS}, the following theorem identifies the precise regimes of $d,t$ where (non-)reconstruction holds.
To state the theorem, we need to know that for $\dmin(k)<d<\dsat(k)$ the polynomial $d(k-1)z^{k-2}(1-z)-1$ has precisely two roots $0<z_*=z_*(d,k)<z^*=z^*(d,k)<1$; we are going to prove this as part of \Prop~\ref{prop_greg} below.
Let
\begin{align}\label{eqlambdas}
	\lambda_*&=\lambda_*(d,k)=-\log(1-z_*)-\frac{z_*}{(k-1)(1-z_*)}>\lambda^*=\lambda^*(d,k)=\max\cbc{0,-\log(1-z^*)-\frac{z^*}{(k-1)(1-z^*)}}\geq0,\\
	\	\theta_*&=\theta_*(d,k)=1-\exp(-\lambda_*)>\theta^*=\theta^*(d,k)=1-\exp(-\lambda^*).
\label{eqthetas}
\end{align}
Additionally, let $\lcond(d,k)$ be the solution to the ODE
\begin{align}\label{eqLena}
	\frac{\partial\lcond(d,k)}{\partial d}&=-\frac{\alpha^*(\lcond(d,k))^k-\alpha_*(\lcond(d,k))^k}{k(\alpha^*(\lcond(d,k))-\alpha_*(\lcond(d,k)))},& \lcond(\dsat(k),k)&=0
\end{align}
on the interval $(\dmin,\dsat]$ and set $\tcond=\tcond(d,k)=1-\exp(-\lcond(d,k))$.
Note that $$\theta^*<\tcond < \theta_*.$$

\begin{theorem}\label{thm_recnonrec}
	Let $k\geq3$ and let $0\leq t=t(n)\leq n$ be a sequence such that $\lim_{n\to\infty}t/n=\theta\in(0,1)$.
	\begin{enumerate}[(i)]
		\item If $d<\dmin(k)$, then $\FDC{t}$ has the non-reconstruction property \whp
		\item If $\dmin(k)<d<\dsat(k)$ and $\theta<\theta^*$ or $\theta>\tcond$, then $\FDC{t}$ has the non-reconstruction property \whp\ 
		\item If $\dmin(k)<d<\dsat(k)$ and $\theta^*<\theta<\tcond$, then $\FDC{t}$ has the reconstruction property \whp
	\end{enumerate}
\end{theorem}

\Thm~\ref{thm_recnonrec} shows that $\dmin(k)$ marks the precise threshold of $d$ up to which the decimation process $\FDC{t}$ exhibits non-reconstruction for {\em all} $0\leq t\leq n$ \whp\
By contrast, for $\dmin(k)<d<\dsat(k)$ there is a regime of $t$ where reconstruction occurs.
In fact, as \Prop~\ref{prop_greg} shows, for $d>\dcore(k)$ we have $\theta^*=0$ and thus reconstruction holds even at $t=0$, i.e., for the original, undecimated random formula $\PHI$.
Prior to the contribution~\cite{RTS}, it had been suggested that this precise scenario (reconstruction on the original problem instance) is the stone on which \BPGD\ stumbles~\cite{Braunstein}.
In fact, Yung's negative result kicks in at this precise threshold $\dcore(k)$.
However, \Thm s~\ref{thm_bpgd} and~\ref{thm_recnonrec} show that matters are more subtle.
Specifically, for $\dmin(k)<d<\dcore(k)$ reconstruction, even though absent in the initial formula $\PHI$, occurs at a later `time' $t>0$ as decimation proceeds, which suffices to trip \BPGD\ up.
Also, remarkably, \Thm~\ref{thm_recnonrec} shows that non-reconstruction is not `monotone'.
The property holds for $\theta<\theta^*$ and then again for $\theta>\tcond$, but not on the interval $(\theta^*,\tcond)$ as visualised in Figure \ref{fig_lena}.

But there is one more surprise.
Namely, \Thm~\ref{thm_recnonrec}~(ii) might suggest that for $\dmin(k)<d<\dsat(k)$ Belief Propagation manages to compute the correct marginals for $t/n\sim \theta >\tcond$, as non-reconstruction kicks back in.
But remarkably, this is not quite true.
Despite the fact that non-reconstruction holds, \BPGD\ goes astray because the algorithm starts its message passing process from a mistaken, oblivious initialisation.
As a consequence, for $t/n\sim\theta\in(\tcond,\theta_*)$ the BP `approximations' remain prone to error.
To be precise, the following result identifies the precise `times' where BP succeeds/fails.
To state the result let $\mu_{\FDC{t}}$ denote the BP `approximation' of the true marginal $\pi_{\FDC{t}}$ of variable $x_{t+1}$ in the formula $\FDC t$ created by the decimation process (see \Sec~\ref{sec_bp} for a reminder of the definition).
Also recall that $\pi_{\FDC t}$ denotes the correct marginal as used by the decimation process.

\begin{theorem}\label{thm_cond}
	Let $k\geq3$ and let $0\leq t=t(n)\leq n$ be a sequence such that $\lim_{n\to\infty}t/n=\theta\in(0,1)$.
	\begin{enumerate}[(i)]
		\item If $0<d<\dmin(k)$ then $\mu_{\FDC{t}}=\pi_{\FDC{t}}$ \whp\
		\item If $\dmin(k)<d<\dsat(k)$ and $\theta<\tcond$ or $\theta>\theta_*$, then $\mu_{\FDC{t}}=\pi_{\FDC{t}}$ \whp
		\item If $\dmin(k)<d<\dsat(k)$ and $\tcond<\theta<\theta_*$, then $\ex\abs{\mu_{\FDC{t}}-\pi_{\FDC{t}}}=\Omega(1).$
	\end{enumerate}
\end{theorem}

The upshot of \Thm s~\ref{thm_recnonrec}--\ref{thm_cond} is that the relation between the accuracy of BP and reconstruction is subtle.
Everything goes well so long as $d<\dmin$ as non-reconstruction holds throughout and the BP approximations are correct.
But if $\dmin<d<\dsat$ and $\theta^*<\theta<\tcond$, then \Thm~\ref{thm_recnonrec}~(iii) shows that reconstruction occurs.
Nonetheless, \Thm~\ref{thm_cond}~(ii) demonstrates that the BP approximations remain valid in this regime.
By contrast, for $\tcond<\theta<\theta_*$ we have non-reconstruction by \Thm~\ref{thm_recnonrec}~(iii), but \Thm~\ref{thm_cond}~(iii) shows that BP misses its mark with a non-vanishing probability.
Finally, for $\theta>\theta_*$ everything is in order once again as BP regains its footing and non-reconstruction holds.
Unfortunately \BPGD\ is unlikely to reach this happy state because the algorithm is bound to make numerous mistakes at times $t/n \sim \theta \in(\tcond,\theta_*)$.

\begin{figure} 
	\begin{subfigure}[b]{0.3\textwidth}
	\includegraphics[height=41mm]{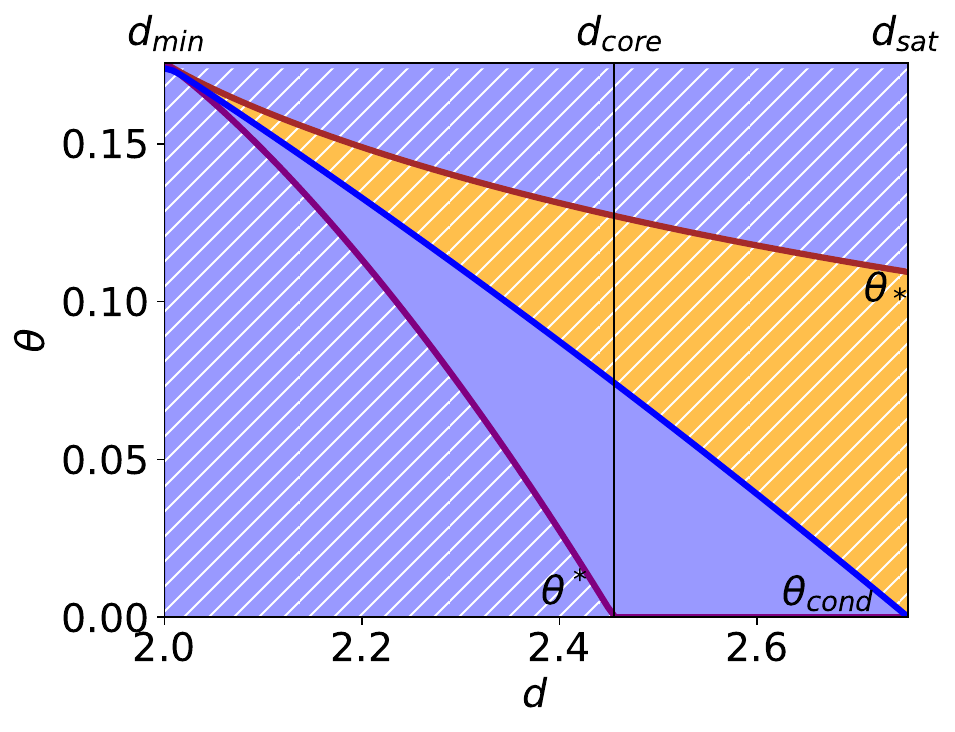}
	\caption{$k=3$}
	\end{subfigure}
	\hspace{5mm}
	\begin{subfigure}[b]{0.3\textwidth}
	\includegraphics[height=41mm]{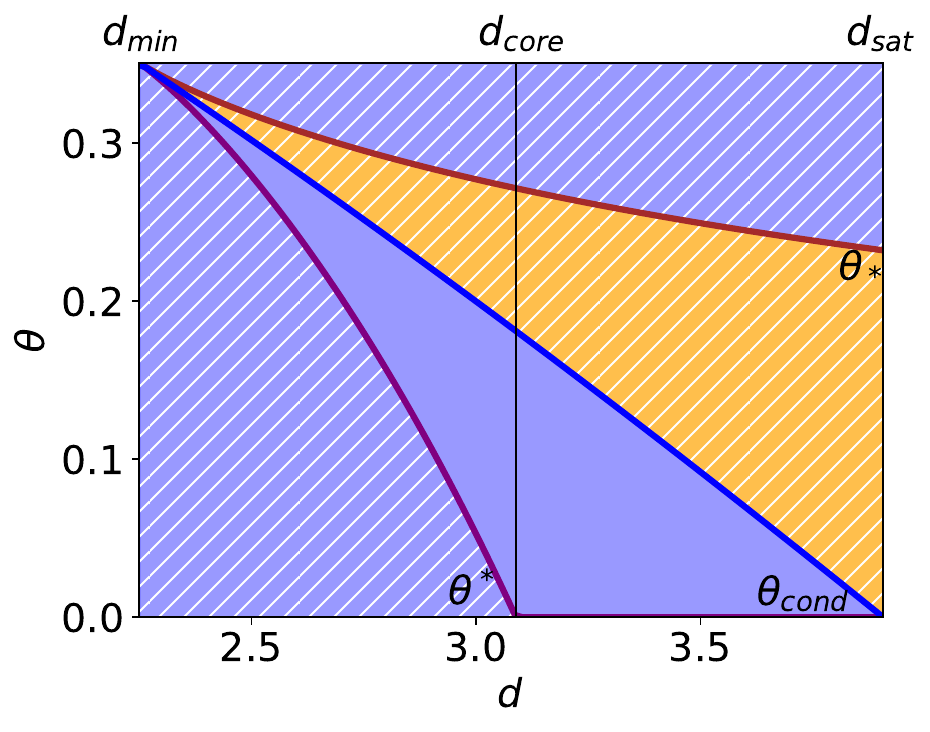}
	\caption{$k=4$}
	\end{subfigure}	
	\hspace{5mm}
	\begin{subfigure}[b]{0.3\textwidth}
	\includegraphics[height=41mm]{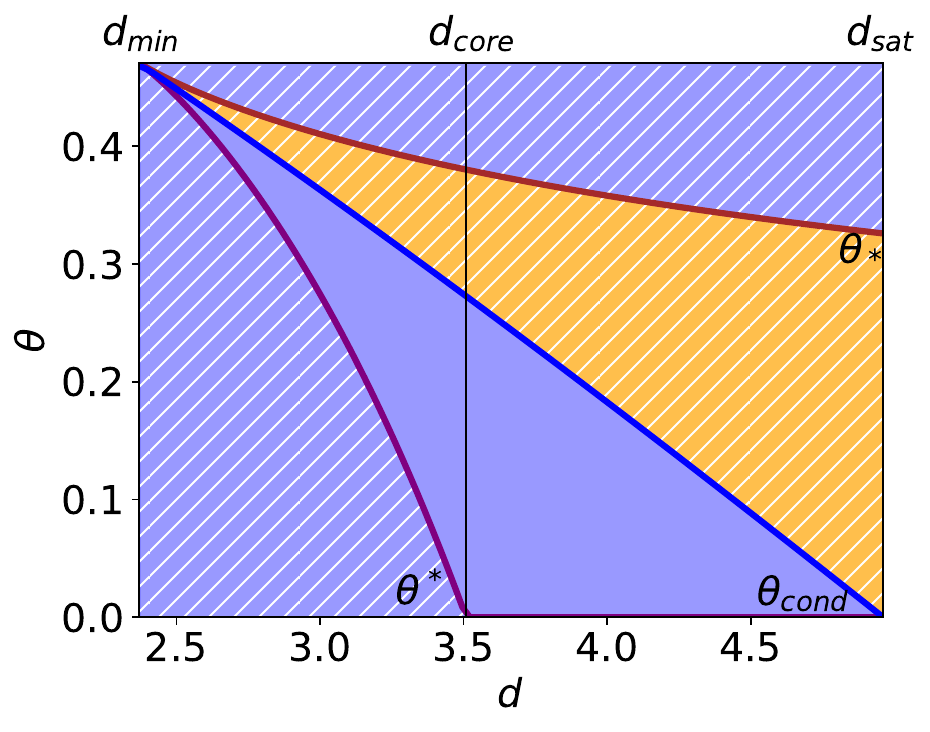}
	\caption{$k=5$}
	\end{subfigure}
	\caption{
		The phase diagrams for $k=3,4,5$ with $d\in(\dmin,\dsat)$ on the horizontal and $\theta$ on the vertical axis.
			The hatched area displays the regime $\theta<\theta^*$ and $\tcond<\theta$ where non-reconstruction holds.
			In the non-hatched area, where $\theta^*<\theta<\tcond$, we have reconstruction.
			Similarly, the blue area displays $\theta<\tcond$ and $\theta>\theta_*$ where BP is correct whereas in the orange area, BP is inaccurate. 
}\label{fig_lena}
\end{figure}

\Thm s~\ref{thm_recnonrec} and~\ref{thm_cond} confirm the predictions from~\cite[\Sec~4]{RTS}.
To be precise, while $\tcond$ matches the predictions of \RTS, the ODE formula~\eqref{eqLena} for the threshold, which is easy to evaluate numerically, does not appear in~\cite{RTS}.
Instead of the ODE formulation, \RTS\ define $\lcond$ as the (unique) $\lambda\geq0$ such that $\Ph(\alpha_*)=\Ph(\alpha^*)$; \Prop~\ref{prop_greg} below shows that both are equivalent.
Illustrating \Thm s~\ref{thm_recnonrec}--\ref{thm_cond}, Figure~\ref{fig_lena} displays the phase diagram in terms of $d$ and $\theta \sim t/n$ for $k=3,4,5$.

\section{Overview}\label{sec_over}

\noindent
This section provides an overview of the proofs of \Thm s~\ref{thm_bpgd}--\ref{thm_cond}.
In the final paragraph we conclude with a discussion of further related work.
{\em We assume throughout that $k\geq3$ is an integer and that $0<d<\dsat(k)$. Moreover, $t=t(n)$ denotes an integer sequence $0\leq t(n)\leq n$ such that $\lim_{n\to\infty}t(n)/n=\theta\in(0,1)$.}

\subsection{Fixed points and thresholds}\label{sec_calc}
The first item on our agenda is to study the functions $\ph,\Ph$ from~\eqref{eqphi}--\eqref{eqPhi}.
Specifically, we are concerned with the maxima of $\Ph$ and the fixed points of $\ph$, the combinatorial relevance of which will emerge as we analyse \BPGD\ and the decimation process.
We begin by observing that the fixed points of $\ph$ are precisely the stationary points of $\Ph$.

\begin{fact}\label{lem_max}
	For any $d>0,\lambda\geq0$ the stationary points $z\in(0,1)$ of $\Phi_{d,k,\lambda}$ coincide with the fixed points of $\phi_{d,k,\lambda}$ in $(0,1)$.
	Furthermore, for a fixed point $z\in(0,1)$ of $\phi_{d,k,\lambda}$ we have
	\begin{align}\label{eq_lem_max}
		\Phi''_{d,k,\lambda}(z)&\begin{cases}
			<0&\mbox{ if }\phi'_{d,k,\lambda}(z)<1,\\
			=0&\mbox{ if }\phi'_{d,k,\lambda}(z)=1,\\
			>0&\mbox{ if }\phi'_{d,k,\lambda}(z)>1.
			\end{cases}
	\end{align}
\end{fact}
\begin{proof}
	Differentiating $\Phi_{d,k,\lambda}$, we obtain
	\begin{align}\label{eq_lem_max_1}
		\Phi_{d,k,\lambda}'(z)&=d(k-1)z^{k-2}\bc{\phi_{d,k,\lambda}(z)-z}.
	\end{align}
	Hence, a point $z\in(0,1)$ is a fixed point of $\phi_{d,k,\lambda}$ iff $\Phi'_{d,k,\lambda}(z)=0$.
	Differentiating \eqref{eq_lem_max_1} once more, we obtain
	\begin{align}\label{eq_lem_max_2}
		\Phi_{d,k,\lambda}''(z)&=d(k-1)z^{k-3}\brk{(k-2)\bc{\phi_{d,k,\lambda}(z)-z}
		+z\bc{\phi_{d,k,\lambda}'(z)-1}}.
	\end{align}
	Clearly, if $\phi_{d,k,\lambda}(z)=z$, then \eqref{eq_lem_max_2} simplifies to $\Phi_{d,k,\lambda}''(z)=d(k-1)z^{k-2}(\phi_{d,k,\lambda}'(z)-1)$, whence \eqref{eq_lem_max} follows.
\end{proof}

We recall that $0\leq\alpha_*=\alpha_*(d,k,\lambda)\leq\alpha^*=\alpha^*(d,k,\lambda)\leq1$ are the smallest and the largest fixed point of $\ph$ in $[0,1]$, respectively.
Fact~\ref{lem_max} shows that $\Ph$ attains its global maximum in $[0,1]$ at $\alpha_*$ or $\alpha^*$.
Let $$\amax=\amax(d,k,\lambda)\in\{\alpha_*,\alpha^*\}$$ be the maximiser of $\Ph$; if $\Ph(\alpha_*)=\Ph(\alpha^*)$, set $\amax=\alpha_*$.
The following proposition characterises the fixed points of $\ph$ and the maximiser $\amax$.

\begin{proposition}\label{prop_greg}
~
	\begin{enumerate}[(i)]
		\item If $d<\dmin(k)$, then for all $\lambda>0$ we have $\alpha_*(d,k,\lambda)=\alpha^*(d,k,\lambda)$, the function $\lambda\in(0,\infty)\mapsto\alpha_*(d,k,\lambda)\in(0,1)$ is analytic, and $\alpha_*(d,k,\lambda)$ is the unique stable fixed point of $\ph$.
		\item If $\dmin(k)<d<\dsat(k)$, then the polynomial $d(k-1)z^{k-2}(1-z)-1$ has precisely two roots $0<z_*<z^*<1$, the numbers $\lambda_*,\lambda^*$ from~\eqref{eqlambdas} satisfy $0\leq \lambda^*<\lambda_*$ and the following is true.
			\begin{enumerate}[(a)]
				\item If $\lambda<\lambda^*$ or $\lambda>\lambda_*$, then $\alpha_*(d,k,\lambda)=\alpha^*(d,k,\lambda)\in(0,1)$ is the unique stable fixed point of $\ph$.
				\item If $\lambda^*<\lambda<\lambda_*$, then $0<\alpha_*(d,k,\lambda)<\alpha^*(d,k,\lambda)<1$ are the only stable fixed points of $\ph$.
				\item  The functions $\lambda\in(0,\lambda_*)\mapsto\alpha_*(d,k,\lambda)$ and $\lambda\in(\lambda^*,\infty)\mapsto\alpha^*(d,k,\lambda)$ are analytic.
		\item If $\dmin(k)<d<\dsat(k)$, then the solution $\lcond$ of~\eqref{eqLena} satisfies $\lambda^*<\lcond=\lcond(d)<\lambda_*$ and
			\begin{align*}
				\amax(d,k,\lambda)&=\begin{cases}
					\alpha_*(d,k,\lambda)&\mbox{ if }\lambda<\lcond,\\
					\alpha^*(d,k,\lambda)&\mbox{ if }\lambda>\lcond.
				\end{cases}
			\end{align*}
			Furthermore, $\Ph(\alpha^*(d,k,\lambda))\neq\Ph(\alpha_*(d,k,\lambda))$ unless $\lambda=\lcond$.
			Thus, the function $\lambda\mapsto\amax(d,k,\lambda)$ is analytic on $(0,\lcond)$ and on $(\lcond,\infty)$, but discontinuous at $\lambda=\lcond$.
		\end{enumerate}
		\end{enumerate}
	\end{proposition}

\begin{figure}
	\begin{subfigure}[b]{0.45\textwidth}
		\includegraphics[height=50mm]{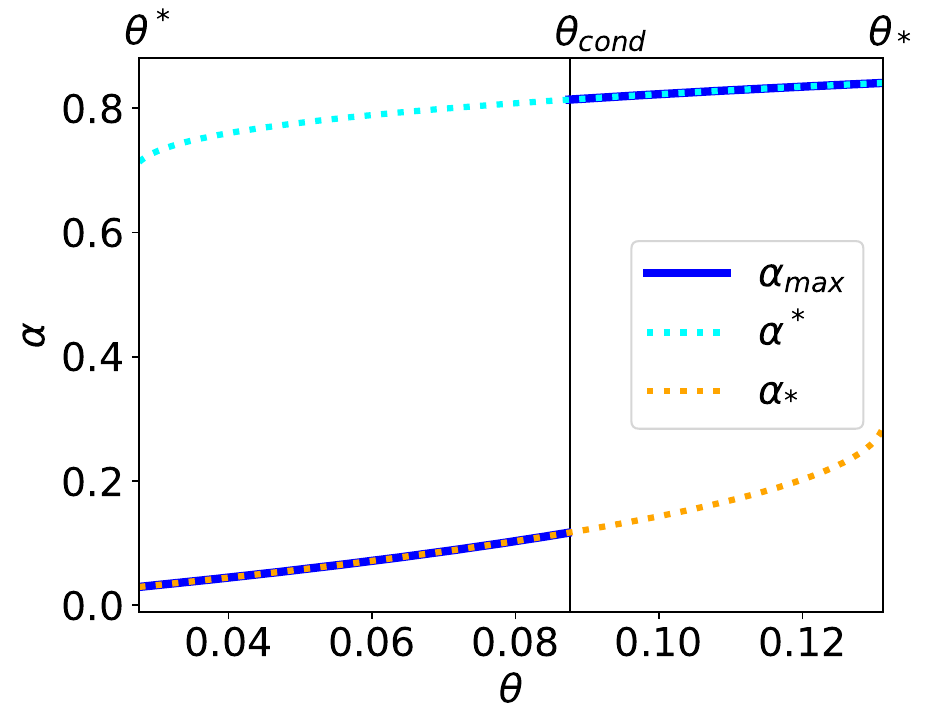}
		\caption{$\amax$ }
	\end{subfigure}
	\begin{subfigure}[b]{0.45\textwidth}		
		\includegraphics[height=50mm]{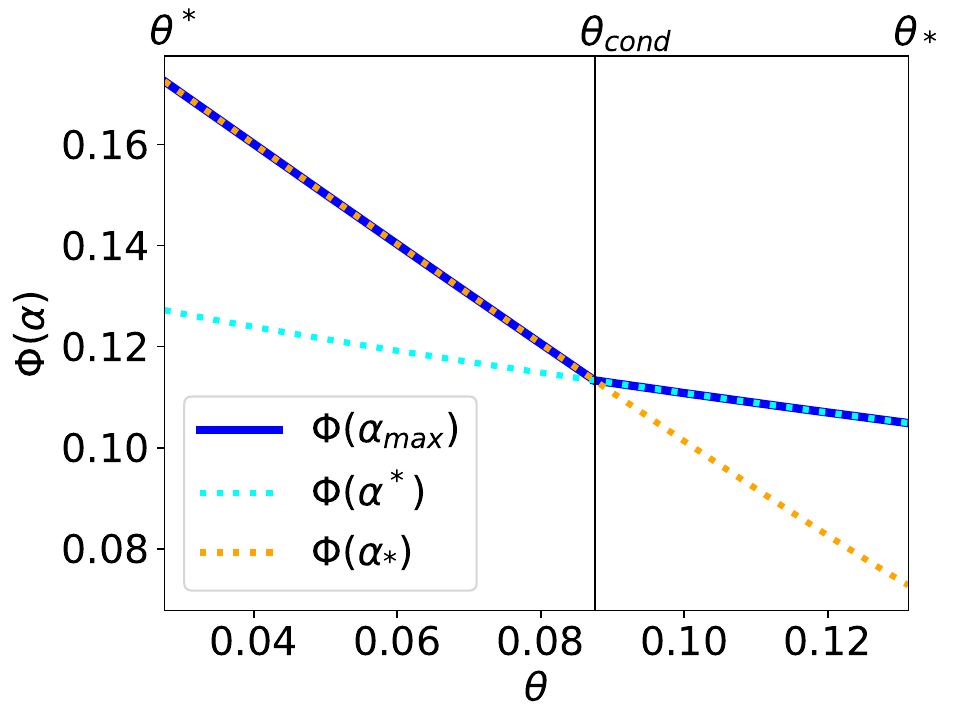}
		\caption{$\Phi(\amax)$}
	\end{subfigure}
	\caption{$\amax$ and $\Phi(\amax) = \Phi_{d,k,\lambda}(\amax)$ for $d=2.4$ and $k=3$ from $\theta^*$ to $\theta_*$.}
	\label{fig_amax}
\end{figure}

\subsection{Belief Propagation}\label{sec_bp}
Having done our analytic homework, we proceed to
recall how Belief Propagation computes the `approximations' $\mu_{\FBP{t}}$ that the \BPGD\ algorithm relies upon.
We will see that due to the inherent symmetries of XORSAT the Belief Propagation computations simplify and boil down to a simpler message passing process called Warning Propagation.
Subsequently we will explain the connection between Warning Propagation and the fixed points $\alpha_*,\alpha^*$ of $\ph$.

It is probably easiest to explain BP on a general XORSAT instance $F$ with a set $V(F)$ of variables and a set $C(F)$ of clauses of lengths between one and $k$.
As in \Sec~\ref{sec_results_dc} we consider the graph $G(F)$ induced by $F$, with vertex set $V(F)\cup  C(F)$ and an edge $xa$ between $x\in V(F)$ and $a\in C(F)$ iff $a$ contains $x$.
Let $\partial v=\partial_Fv$ be the set of neighbours of $v\in V(F)\cup C(F)$.
Additionally, given an assignment $\tau\in\{0,1\}^{\partial a}$ of the variables that appear in $a$, we write $\tau\models a$ iff $\tau$ satisfies $a$.

With each variable/clause pair $x,a$ such that $x\in\partial a$ Belief Propagation associates two sequences of `messages' $(\mu_{F,x\to a,\ell})_{\ell\geq0}$, $(\mu_{F,a\to x,\ell})_{\ell\geq0}$ directed from $x$ to $a$ and from $a$ to $x$, respectively.
These messages are probability distributions on $\{0,1\}$, i.e., 
\begin{align}\label{eqBPmsg1}
	\mu_{F,x\to a,\ell}=(\mu_{F,x\to a,\ell}(0),\mu_{F,x\to a,\ell}(1)),\,\mu_{F,a\to x,\ell}=(\mu_{F,a\to x,\ell}(0),\mu_{F,a\to x,\ell}(1)),\\
	\mu_{F,x\to a,\ell}(0)+\mu_{F,x\to a,\ell}(1)=\mu_{F,a\to x,\ell}(0)+\mu_{F,a\to x,\ell}(1)&=1.\label{eqBPmsg2}
	\end{align}
The initial messages are uniform, i.e.,
\begin{align}\label{eqBPupdate0}
	\mu_{F,x\to a,0}(s)=\mu_{F,a\to x,0}(s)&=1/2&&(s\in\{0,1\}).
\end{align}
Further, the messages at step $\ell+1$ are obtained from the messages at step $\ell$ via the {\em Belief Propagation equations}
\begin{align}\label{eqBPupdate1}
	\mu_{F,a\to x,\ell+1}(s)&\propto\sum_{\tau\in\{0,1\}^{\partial a}}\vecone\{\tau_{x}=s,\,\tau\models a\}\prod_{y\in\partial a\setminus\{x\}}\mu_{F,y\to a,\ell}(\tau_y),\\
	\mu_{F,x\to a,\ell+1}(s)&\propto\prod_{b\in\partial x\setminus\{a\}}\mu_{F,b\to x,\ell}(s).\label{eqBPupdate2}
\end{align}
In \eqref{eqBPupdate1}--\eqref{eqBPupdate2} the $\propto$-symbol represents the normalisation required to ensure that the updated messages satisfy~\eqref{eqBPmsg2}.
In the case of~\eqref{eqBPupdate2} such a normalisation may be impossible because the expressions on the r.h.s.\ could vanish for both $s=0$ and $s=1$.
In this event we agree that
	\begin{align*}
		\mu_{F,x\to a,\ell+1}(s)=\begin{cases}
			\mu_{F,x\to a,\ell}(s)&\mbox{ if }\mu_{F,x\to a,\ell}(s)\neq1/2\\
			\vecone\{s=0\}&\mbox{ otherwise}
		\end{cases}&&(s\in\{0,1\});
	\end{align*}
in other words, we retain the messages from the previous iteration unless its value was $1/2$, in which case we set $\mu_{F,x\to a,\ell+1}(0)=1$. 
The same convention applies to $\mu_{F,a\to x,\ell+1}(s)$. 
Further, at any time $t$ the BP messages render a heuristic `approximation' of the marginal probability that a random solution to the formula $F$ sets a variable $x$ to $s\in\{0,1\}$:
\begin{align}\label{eqBPmarg}
	\mu_{F,x,\ell}(s)&\propto\prod_{b\in\partial x}\mu_{F,b\to x,\ell}(s).
\end{align}
We set $\mu_{F,x,\ell}(0)=1-\mu_{F,x,\ell}(1)=1$ if the normalisation in~\eqref{eqBPmarg} fails, i.e., if $\sum_{s\in\{0,1\}}\prod_{b\in\partial x}\mu_{F,b\to x,\ell}(s)=0$.

\begin{fact}\label{fact_halfint}
	The BP messages and marginals are half-integral for all $t$, i.e., for all $t\geq0$ and $s\in\{0,1\}$ we have
	\begin{align}\label{eqfact_halfint}
			\mu_{F,x\to a,\ell}(s),\mu_{F,a\to x,\ell}(s),\mu_{F,x,\ell}(s)\in\{0,1/2,1\}.
		\end{align}
		Furthermore, for all $\ell>2\sum_{a\in C(F)}|\partial a|$ we have $\mu_{F,x,\ell}(s)=\mu_{F,x,\ell+1}(s)$.
\end{fact}
\begin{proof}
	The half-integrality~\eqref{eqfact_halfint} follows from a straightforward induction on $\ell$.
	Furthermore, another induction on $\ell$ and inspection of \eqref{eqBPupdate1}--\eqref{eqBPupdate2} shows that for any $x,a,\ell$ such that $\mu_{F,x\to a,\ell}(1)\neq1/2$ we have $\mu_{F,x\to a,\ell+1}(s)=\mu_{F,x\to a,\ell}(s)$ ($s\in\{0,1\}$).
	A similar statement holds for $\mu_{F,a\to x,\ell+1}(s)$.
	In particular, the number of messages that take the value $1/2$ is monotonically decreasing in $\ell$.
	Since the total number of messages is bounded by $2\sum_{a\in C(F)}|\partial a|$, we conclude that the messages will have converged pointwise after this number of iterations.
\end{proof}

\noindent
Finally, in light of Fact~\ref{fact_halfint} it makes sense to define the approximations for \BPGD\ by letting
\begin{align}\label{eqmuBPGD}
	\mu_{\FBP t}&=\lim_{\ell\to\infty}\mu_{\FBP t,x_{t+1},\ell}(1),&
	\mu_{\FDC t}&=\lim_{\ell\to\infty}\mu_{\FDC t,x_{t+1},\ell}(1).
\end{align}

\subsection{Warning Propagation}\label{sec_wp}
Thanks to the half-integrality~\eqref{eqfact_halfint} of the messages, Belief Propagation is equivalent to a purely combinatorial message passing procedure called {\em Warning Propagation} (`WP')~\cite{MM}.
Similar as BP, WP also associates two message sequences $(\omega_{F,x\to a,\ell},\omega_{F,a\to x,\ell})_{\ell\geq0}$ with every adjacent variable/clause pair.
The messages take one of three possible discrete values $\{\fzn,\uzn,\nll\}$ (`frozen', `uniform', `null').
Essentially, $\nll$ indicates that the value of a variable is determined by unit clause propagation.
Moreover, $\fzn$ indicates that a variable is forced to take the value $0$ once all variables in the $2$-core of the hypergraph representation of the formula are set to $0$.
The remaining label $\uzn$ indicates that neither of the above applies.
To trace the BP messages from \Sec~\ref{sec_bp} actually only the two values $\{\nll,\uzn\}$ would be necessary.
However, the third value $\fzn$ will prove useful in order to compare the BP approximations with the actual marginals.
Perhaps unexpectedly given the all-uniform initialisation~\eqref{eqBPupdate0}, we launch WP from all-frozen start values:
\begin{align}\label{eqWPupdate0}
	\omega_{F,x\to a,0}=\omega_{F,a\to x,0}=\fzn&&\mbox{for all }a,x.
\end{align}
Subsequently the messages get updated according to the rules
\begin{align}\label{eqWP1}
	\omega_{F,a\to x,\ell+1}&=\begin{cases}
		\nll&\mbox{ if }\omega_{F,y\to a,\ell}=\nll\mbox{ for all }y\in\partial a\setminus\{x\},\\
		\fzn&\mbox{ if }\omega_{F,y\to a,\ell}\neq\uzn\mbox{ for all }y\in\partial a\setminus\{x\}
		\mbox{ and }\omega_{F,y\to a,\ell}\neq\nll\mbox{ for at least one }y\in\partial a\setminus\{x\},\\
		\uzn&\mbox{ otherwise},
		\end{cases}\\
	\omega_{F,x\to a,\ell+1}&=\begin{cases}
		\nll&\mbox{ if }\omega_{F,b\to x,\ell}=\nll\mbox{ for at least one }b\in\partial x\setminus\{a\},\\
		\fzn&\mbox{ if }\omega_{F,b\to x,\ell}\neq\nll\mbox{ for all }b\in\partial x\setminus\{a\}
		\mbox{ and }\omega_{F,b\to x,\ell}=\fzn\mbox{ for at least one }b\in\partial x\setminus\{a\},\\
		\uzn&\mbox{ otherwise}.
		\end{cases}\label{eqWP2}
\end{align}
In addition to the messages we also define the {\em mark} of variable node $x$ by letting
\begin{align}\label{eqmarks}
	\omega_{F,x,\ell}&=\begin{cases}
		\nll&\mbox{ if }\omega_{F,b\to x,\ell}=\nll\mbox{ for at least one }b\in\partial x,\\
		\fzn&\mbox{ if }\omega_{F,b\to x,\ell}\neq\nll\mbox{ for all }b\in\partial x
		\mbox{ and }\omega_{F,b\to x,\ell}=\fzn\mbox{ for at least one }b\in\partial x,\\
		\uzn&\mbox{ otherwise}.
		\end{cases}
\end{align}
The following statement summarises the relationship between BP and WP.

\begin{fact}\label{fact_wp}
	For all $t\geq0$ and all $x,a$ we have
	\begin{align}\label{eqfact_wp1}
		\mu_{x\to a,\ell}(1)&=1/2&&\Leftrightarrow&\omega_{F,x\to a,\ell}&\neq\nll,\\
		\mu_{a\to x,\ell}(1)&=1/2&&\Leftrightarrow&\omega_{F,a\to x,\ell}&\neq\nll,\label{eqfact_wp2}\\
		\mu_{x,\ell}(1)&=1/2&&\Leftrightarrow&\omega_{F,x,\ell}&\neq\nll.\label{eqfact_wp3}
		\end{align}
	Moreover, for all $\ell>2|C(F)|$ we have $\omega_{F,x\to a,\ell}=\omega_{F,x\to a,\ell+1}$ and $\omega_{F,a\to x,\ell}=\omega_{F,a\to x,\ell+1}$.
\end{fact}
\begin{proof}
	The fact that $\omega_{F,x\to a,\ell}=\omega_{F,x\to a,\ell+1}$ and $\omega_{F,a\to x,\ell}=\omega_{F,a\to x,\ell+1}$ for all $\ell>2|C(F)|$ follows from the observation that the number of $\fzn$-messages is monotonically decreasing, while the number of $\nll$-messages is monotonically increasing.
	The equations \eqref{eqfact_wp1}--\eqref{eqfact_wp3} follow by induction on $\ell$.
	Initially all the messages are uniform in BP, i.e., $\mu_{x\to a,0}(1)=\mu_{a\to x,0}(1)=1/2$. By contrast, in WP, we start with all frozen values to both variables and clauses as given by \eqref{eqWPupdate0}. 
	Then from \eqref{eqWP1},\eqref{eqWP2} and \eqref{eqmarks}, for $\ell=0$,\eqref{eqfact_wp1}--\eqref{eqfact_wp3} hold true.
		For $\ell=1$, we get the messages and marginals in BP obtained from the messages at initial step. From \eqref{eqBPupdate1} it follows that if the marginals are uniform then from WP arguments \eqref{eqWP1}, it is sure that $\omega_{F,a\to x,1}\neq \nll$ because $\omega_{F,y\to a,0}=\fzn$.
		The same argument is valid for the other way round. If the WP message at step $\ell=1$ is not null, then the BP message from \eqref{eqBPupdate1} after normalization become $1/2$. So for $\ell=1$, \eqref{eqfact_wp1} holds true.\\
		Let us assume the \eqref{eqfact_wp1} is true for any step $\ell$.Then for step $\ell+1$ the messages in BP is obtained from step $\ell$ as in \eqref{eqBPupdate1} is $\frac{1}{2}$ implies in WP message $\omega_{F,a\to x,\ell+1}\neq \nll$ because $\omega_{F,y\to a,\ell}=\uzn \mbox{ for at least one }y\in\partial a\setminus\{x\}$. Similarly, if the WP message $\omega_{F,a\to x,\ell+1}\neq \nll$ implies this can be either "uniform" or "frozen". Now, if there will be at least one uniform incoming message then $\mu_{a\to x,\ell+1}(1)=1/2$ and for all frozen incoming messages it is straightforward from the initialization of WP \eqref{eqWPupdate0} which corresponds to $\mu_{a\to x,\ell+1}(1)=1/2$. So at step $\ell+1$, \eqref{eqfact_wp1} holds true.
		We conclude that \eqref{eqfact_wp1} holds true for every $\ell$. 
		Similarly, by induction on $\ell$ we can conclude that \eqref{eqfact_wp2}--\eqref{eqfact_wp3} also hold true for every $\ell$.
\end{proof}

Fact~\ref{fact_wp} implies that the WP messages and marks `converge' in the limit of large $\ell$, in the sense that eventually they do not change any more.
Let $\omega_{F,x\to a},\omega_{F,a\to x},\omega_{F,x}\in\{\fzn,\uzn,\nll\}$ be these limits.
Furthermore, let $V_{\fzn,\ell}(F)$, $V_{\uzn,\ell}(F)$, $V_{\nll,\ell}(F)$ be the sets of variables with the respective mark after $\ell\geq0$ iterations. 
Also let $V_{\fzn}(F),V_{\uzn}(F),V_{\nll}(F)$ be the sets of variables where the limit $\omega_{F,x}$ takes the respective value.
The following statement traces WP on the random formula $\FDC{t}$ produced by the decimation process.

\begin{proposition}\label{prop_WP}
	Let $\eps>0$ and assume that $d>0$, $t=t(n)\lk{\sim \theta n}$ satisfy one of the following conditions:
	\begin{enumerate}[(i)]
		\item $d<\dmin$, or
		\item $d>\dmin$ and $\theta\not\in\{\theta_*,\theta^*\}$.
	\end{enumerate}
	Then there exists $\ell_0=\ell_0(d,\theta,\eps)>0$ such that for any fixed $\ell\geq\ell_0$ with $\lambda=-\log(1-\theta)$ \whp\ we have 
	\begin{align}\label{eqprop_WP}
		\abs{t+|V_{\nll,\ell}(\FDC{t})|-\alpha_*n}&<\eps n,&\abs{t+|V_{\fzn,\ell}(\FDC{t})|-(\alpha^*-\alpha_*)n}&<\eps n,&
		\abs{V_{\nll}(\FDC{t})\triangle V_{\nll,\ell}(\FDC{t}) }&<\eps n. 
	\end{align}
\end{proposition}
	
\subsection{The check matrix}\label{sec_check}
Since the XOR operation is equivalent to addition modulo two, a XORSAT formula $F$ with variables $x_1,\ldots,x_n$ and clauses $a_1,\ldots,a_m$ translates into a linear system over $\FF_2$, as follows.
Let $A_F$ be the $m\times n$-matrix over $\FF_2$ whose $(i,j)$-entry equals one iff variable $x_j$ appears in clause $a_i$.
Adopting coding parlance, we refer to $A_F$ as the {\em check matrix} of $F$.
Furthermore, let $y_F\in\FF_2^m$ be the vector whose $i$th entry is one plus the sum of any constant term and the number of negation signs of clause $a_i$ mod two.
Then the solutions $\sigma\in\FF_n^n$ of the linear system $A_F\sigma=y_F$ are precisely the satisfying assignments of $F$.

The algebraic properties of $A_F$ therefore have a direct impact on the satisfiability of $F$.
For example, if $A_F$ has rank $m$, we may conclude immediately that $F$ is satisfiable.
Furthermore, the set of solutions of $F$ is an affine subspace of $\FF_2^n$ (if non-empty).
In effect, if $F$ is satisfiable, then the number of satisfying assignments equals the size of the kernel of $A_F$.
Hence the nullity $\nul A_F=\dim\ker A_F$ of the check matrix is a key quantity.

Indeed, the single most significant ingredient towards turning the heuristic arguments from~\cite{RTS} into rigorous proofs is a formula for the nullity of the check matrix of the XORSAT instance $\FDC{t}$ from the decimation process.
To unclutter the notation set $\ADC{t}=A_{\FDC{t}}$.
We derive the following proposition from a recent general result about the nullity of random matrices over finite fields~\cite[\Thm~1.1]{Maurice}.
The proposition clarifies the semantics of the function $\Ph$ and its maximiser $\amax$.
In physics jargon $\Ph$ is known as the Bethe free entropy.

\begin{proposition}\label{prop_nul}
	Let $d>0$ and $\lambda=-\log(1-\theta)$.
	Then
	\begin{align*}
		\lim_{n\to\infty}\nul\ADC{t}&=\Phi_{d,k,\lambda}(\amax)&&\mbox{in probability}.
	\end{align*}
\end{proposition}

\subsection{Null variables}\label{sec_frozen}
\Prop~\ref{prop_nul} enables us to derive crucial information about the set of satisfying assignments of $\FDC{t}$.
Specifically, for any XORSAT instance $F$ with variables $x_1,\ldots,x_n$ let $V_0(F)$ be the set of variables $x_i$ such that $\sigma_i=0$ for all $\sigma\in\ker A_F$.
We call the variables $x_i\in V_0(F)$ {\em null variables}.
Since the set of solutions of $F$, if non-empty, is a translation of $\ker A_F$, any two solutions $\sigma,\sigma'$ of $F$ set the variables in $V_0(F)$ to exactly the same values.
The following proposition shows that WP identifies certain variables as null.

\begin{proposition}\label{prop_nlluzn}
	\Whp\ the following two statements are true for any fixed integer $\ell>0$.
	\begin{enumerate}[(i)]
		\item We have $V_{\nll,\ell}(\FDC t)\subset\frz(\FDC t)$.
		\item We have $|V_{\uzn,\ell}(\FDC{t})\cap\frz(\FDC{t})|=o(n)$.
	\end{enumerate}
\end{proposition}

\Prop s~\ref{prop_nul} and~\ref{prop_nlluzn} enable us to calculate the number of null variables of $\FDC{t}$, so long as we remain clear of the point $\tcond$ where $\amax$ is discontinuous.

\begin{proposition}\label{cor_frz}
	If $\theta\neq\tcond$ then $|\frz(\FDC{t})|=\amax n+o(n)$ \whp
\end{proposition}

Let us briefly summarise what we have learned thus far.
First, because all Belief Propagation messages are half-integral, BP reduces to WP.
Second, \Prop~\ref{prop_WP} shows that the fixed points $\alpha_*,\alpha^*$ of $\ph$ determine the number of variables marked $\nll$ or $\fzn$ by WP.
Third, the function $\Ph$ and its maximiser $\amax$ govern the nullity of the check matrix and thereby the number of null variables of $\FDC{t}$.
Clearly, the null variables $x_i$ are precisely the ones whose actual marginals $\pr\brk{\SIGMA_{\FDC t}(x_i)=s\mid\FDC{t}}$ are {\em not} uniform.
As a next step, we investigate whether BP/WP identify these variables correctly.

In light of \Prop~\ref{prop_WP}, in order to investigate the accuracy of BP it suffices to compare the {\em numbers} of variables marked $\nll$ by WP with the true marginals.
The following corollary summarises the result.

\begin{corollary}\label{cor_nll}
	For any $d$, $\theta$ the following statements are true.
	\begin{enumerate}[(i)]
		\item If $d<\dmin$, \emph{or} $d>\dmin$ and $\theta<\tcond$, \emph{or} $d>\dmin$ and $\theta>\theta_*$, then $|\frz(\FDC{t})\triangle V_{\nll}(\FDC{t})|=o(n)$ \whp
		\item If $d>\dmin$ and $\tcond<\theta<\theta_*$, then $|\frz(\FDC{t})\triangle V_{\nll}(\FDC{t})|=\Omega(n)$ \whp
	\end{enumerate}
\end{corollary}

Thus, so long as $d<\dmin$ or $d>\dmin$ and $\theta<\tcond$ or $\theta>\theta_*$, the BP/WP approximations are mostly correct.
By contrast, if $d>\dmin$ and $\tcond<\theta<\theta_*$, the BP/WP approximations are significantly at variance with the true marginals \whp\
Specifically, \whp\ BP deems $\Omega(n)$ frozen variables unfrozen, thereby setting itself up for failure.
Indeed, \Cor~\ref{cor_nll} easily implies \Thm~\ref{thm_cond}, which in turn implies \Thm~\ref{thm_bpgd}~(ii) without much ado.

In addition, to settle the (non-)reconstruction thresholds set out in \Thm~\ref{thm_recnonrec} we need to investigate the {\em conditional} marginals given the values of variables at a certain distances from $x_{t+1}$ as in~\eqref{eqnonrec}.
This is where the extra value $\fzn$ from the construction of WP enters.
Indeed, for a XORSAT instance $F$ with variables $x_1,\ldots,x_n$ and an integer $\ell$ let $V_{0,\ell}(F)$ be the set of variables $x_i$ such that $\sigma_i=0$ for all $\sigma\in\ker A_F$ for which $\sigma_h=0$ for all variables $x_h\in\partial^\ell x_i$.
Hence, $V_{0,\ell}(F) \subset V_0(F)$ is the set of variables whose $\ell$-neighbourhood is contained in $V_0(F)$. 

\begin{corollary}\label{lem_fzn}
	Assume that $d>\dmin$ and let $\eps>0$. 
	\begin{enumerate}[(i)]
		\item If $\theta<\tcond$, then for any fixed $\ell$ we have $|V_{\fzn,\ell}(\FDC{t})\cap V_{0,\ell}(\FDC{t})|<\eps n$ \whp\
		\item If $\theta>\tcond$, then there exists $\ell_0=\ell_0(d,\theta,\eps)$ such that for any fixed $\ell>\ell_0$ we have
			\begin{align*}
				|(V_{\nll,\ell}(\FDC{t})\cup V_{\fzn,\ell}(\FDC{t}))\triangle V_{0,\ell}(\FDC{t})|<\eps n\qquad\mbox\whp
			\end{align*}
	\end{enumerate}
\end{corollary}

\noindent
Comparing the number of actually frozen variables with the ones marked $\fzn$ by WP, we obtain \Thm~\ref{thm_recnonrec}.

\subsection{Proving \BPGD\ successful}\label{sec_bpgd_success}
We are left to prove \Thm~\ref{thm_bpgd}.
First, we need to compute the (strictly positive) success probability of \BPGD\ for $d<\dmin$.
At this point, the fact that \BPGD\ has a fair chance of succeeding for $d<\dmin$ should not come as a surprise.
Indeed, \Cor~\ref{cor_nll} implies that the BP approximations of the marginals are mostly correct for $d<\dmin$, at least on the formula $\FDC t$ created by the decimation process.
Furthermore, so long as the marginals are correct, the decimation process $\FDC{t}$ and the execution of the \BPGD\ algorithm $\FBP{t}$ move in lockstep.
The sole difficulty in analysing \BPGD\ lies in proving that the estimates of the algorithm are not just mostly correct, but correct up to only a {\em bounded} expected number of discrepancies over the entire execution of the algorithm.
To prove this fact we combine the method of differential equations with a subtle analysis of the sources of the remaining bounded number of discrepancies.
These discrepancies result from the presence of short (i.e., bounded-length) cycles in the graph $G(\PHI)$.
Finally, the proof of the second (negative) part of \Thm~\ref{thm_bpgd} follows by coupling the execution of \BPGD\ with the decimation process, and invoking \Thm~\ref{thm_cond}.
The details of both arguments can be found in \Sec~\ref{sec_alg}.

\subsection{Discussion}\label{sec_discussion}
The thrust of the present work is to verify the predictions from~\cite{RTS} on the \BPGD\ algorithm and the decimation process rigorously.
Concerning the decimation process, the main gap in the deliberations of \RTS\ \cite{RTS} that we needed to plug is the proof of \Prop~\ref{cor_frz} on the actual number of null variables in the decimation process.
The proof of \Prop~\ref{cor_frz}, in turn, hinges on the formula for the nullity from \Prop~\ref{prop_nul}, whereas \RTS\ state the (as it turns out, correct) formulas for the nullity and the number of null variables based on purely heuristic arguments.

Regarding the analysis of the \BPGD\ algorithm, 
\RTS\ state that they rely on the heuristic techniques from the insightful article~\cite{DeroulersMonasson} to predict the formula~\eqref{eqthm_bpgd}, but do not provide any further details; the article~\cite{DeroulersMonasson} principally employs heuristic arguments involving generating functions.
By contrast, the method that we use to prove~\eqref{eqthm_bpgd} is a bit more similar to that of Frieze and Suen~\cite{FriezeSuen} for the analysis of a variant of the unit clause algorithm on random $k$-SAT instances, for which they also obtain the asymptotic success probability.
Yet by comparison to the argument of Frieze and Suen, we pursue a more combinatorially explicit approach that demonstrates that certain small sub-formulas that we call `toxic cycles' are responsible for the failure of \BPGD.
Specifically, the proof of~\eqref{eqthm_bpgd} combines the method of differential equations with Poissonisation.
Finally, the proof of \Thm~\ref{thm_bpgd}~(ii) is an easy afterthought of the analysis of the decimation process.

Yung's work~\cite{Yung} on random $k$-XORSAT is motivated by the `overlap gap paradigm'~\cite{Gamarnik}, the basic idea behind which is to show that a peculiar clustered geometry of the set of solutions is an obstacle to certain types of algorithms.
Specifically, Yung only considers the Unit Clause Propagation algorithm and (a truncated version of) \BPGD.
Following the path beaten in~\cite{Mora}, Yung performs moment computations to establish the overlap gap property.
However, moment computations (also called `annealed computations' in physics jargon) only provide one-sided bounds.
Yung's results require spurious lower bounds on the clause length $k$ ($k\geq9$ for Unit Clause and $k\geq13$ for \BPGD).
By contrast, the present proof strategy pivots on the number of null variables rather than overlaps, and \Prop~\ref{cor_frz} provides the precise `quenched' count of null variables. 
A further improvement over~\cite{Yung} is that the present analysis pinpoints the {\em precise} threshold up to which \BPGD\ (as well as Unit Clause) succeeds for any $k\geq3$.
Specifically, Yung proves that these algorithms fail for $d>\dcore$, while \Thm~\ref{thm_bpgd} shows that failure occurs already for $d>\dmin$ with $\dmin<\dcore$.
Conversely, \Thm~\ref{thm_bpgd} shows that the algorithms succeed with a non-vanishing probability for $d<\dmin$. 
Thus, \Thm~\ref{thm_bpgd} identifies the correct threshold for the success of \BPGD, as well as the correct combinatorial phenomenon that determines this threshold, namely the onset of reconstruction in the decimation process (\Thm s~\ref{thm_recnonrec} and~\ref{thm_cond}).

The \BPGD\ algorithm as detailed in \Sec~\ref{sec_bp} applies to a wide variety of problems beyond random $k$-XORSAT.
Of course, the single most prominent example is random $k$-SAT.
Lacking the symmetries of XORSAT, random $k$-SAT does not allow for the simplification to discrete messages; in particular, the BP messages are not generally half-integral.
In effect, BP and WP are no longer equivalent.
In addition to random $k$-XORSAT, the article~\cite{RTS} also provides a heuristic study of \BPGD\ on random $k$-SAT.
But once again due to the lack of half-integrality, the formulas for the phase transitions no longer come as elegant finite-dimensional expressions.
Instead, they now come as infinite-dimensional variational problems.
Furthermore, the absence of half-integrality also entails that the present proof strategy does not extend to $k$-SAT.

The lack of inherent symmetry in random $k$-SAT can partly be compensated by assuming that the clause length $k$ is sufficiently large (viz.\ larger than some usually unspecified constant $k_0$).
Under this assumption the random $k$-SAT version of both the decimation process and the \BPGD\ algorithm have been analysed rigorously~\cite{BP,Angelica}.
The results are in qualitative agreement with the predictions from~\cite{RTS}.
In particular, the \BPGD\ algorithm provably fails to find satisfying assignments on random $k$-SAT instances even below the threshold where the set of satisfying assignments shatters into well-separated clusters~\cite{Barriers,pnas}.
Furthermore, on random $k$-SAT a more sophisticated message passing algorithm called Survey Propagation Guided Decimation has been suggested~\cite{MPZ,RTS}.
While on random XORSAT Survey Propagation and Belief Propagation are equivalent, the two algorithms are substantially different on random $k$-SAT.
One might therefore hope that Survey Propagation Guided Decimation outperforms \BPGD\ on random $k$-SAT and finds satisfying assignments up to the aforementioned shattering transition.
A negative result to the effect that Survey Propagation Guided Decimation fails asymptotically beyond the shattering transition point for large enough $k$ exists~\cite{Hetterich}.
Yet a complete analysis of Belief/Survey Propagation Guided Decimation on random $k$-SAT for any $k\geq3$ in analogy to the results obtained here for random $k$-XORSAT remains an outstanding challenge.

Finally, returning to random $k$-XORSAT, a question for future work may be to investigate the performance of various types of algorithms such as greedy, message passing or local search that aim to find an assignment that violates the least possible number of clauses.
Of course, this question is relevant even for $d>\dsat(k)$.
A first step based on the heuristic `dynamical cavity method' was recently undertaken by Maier, Behrens and Zdeborov\'a~\cite{MBZ}.



\subsection{Preliminaries and notation}\label{sec_notation}
Throughout we assume that $k\geq3$ and $0<d<\dmin$ and $\theta\in(0,1)$ are fixed independently of $n$.
We always let $t=t(n)\in\{0,1,\ldots,n\}$ be an integer sequence such that $\lim_{n\to\infty}t/n=\theta$.
Unless specified otherwise we tacitly assume that $n$ is sufficiently large for our various estimates to hold.
Asymptotic notation such as $O(\nix)$ refers to the limit of large $n$ by default, with $k,d,\theta$ fixed.
We continue to denote by $\alpha_*=\alpha_*(\lambda)=\alpha_*(d,k,\lambda)$ and $\alpha^*=\alpha^*(\lambda)=\alpha^*(d,k,\lambda)$ the smallest/largest fixed points of $\ph$ in $[0,1]$ and by $\lambda_*=\lambda_*(d,k)$, $\lambda^*=\lambda^*(d,k)$, $\theta_*=\theta_*(d,k)$, $\theta^*=\theta^*(d,k)$ the quantities defined in~\eqref{eqlambdas}--\eqref{eqthetas}.

For a formula $F$ and a partial assignment $\sigma:U\to\{0,1\}$ with $U\subset V(F)$ let $F[\sigma]$ be the simplified formula obtained by substituting constants for the variables in $U$.
The {\em length} of a clause of $F[\sigma]$ is defined as the number of variables from $V(F)\setminus U$ that the clause contains.

	The following fact provides the correctness of BP on formulas represented by acyclic graphs $G(F)$.
	\begin{fact}[{\cite[Chapter 14]{MM}}]\label{fact_BP_acyclic}
		For a XORSAT Formula $F$ with an acyclic bipartite graph $G(F)$ the BP marginals as defined in \eqref{eqBPmarg} are exact, i.e.
		\begin{align*}
			\lim_{\ell \to \infty}\mu_{F,x,\ell}(1) &=\pr\brk{\SIGMA_{F}(x)=1}.
		\end{align*}
	\end{fact}

\subsection{Organisation}\label{sec_org}
The rest of the paper is organised as follows.
\Sec~\ref{sec_greg} contains the proof of \Prop~\ref{prop_greg}.
Subsequently in \Sec~\ref{sec_prop_WP} we investigate Warning Propagation to prove \Prop s~\ref{prop_WP} and~\ref{prop_nlluzn}.
Furthermore, \Sec~\ref{sec_nul} deals with the study of the check matrix; here we prove \Prop s~\ref{prop_nul} and~\ref{cor_frz} as well as Corollaries~\ref{cor_nll} and~\ref{lem_fzn}.
Additionally, with all these preparations completed we put all the pieces together to complete the proofs of \Thm s~\ref{thm_recnonrec} and~\ref{thm_cond} in \Sec~\ref{sec_thm_recnonrec}.
Finally, \Sec~\ref{sec_alg} contains the proof of \Thm~\ref{thm_bpgd}.

\section{Proof of \Prop~\ref{prop_greg}}\label{sec_greg}

\noindent
Even though a few steps are mildly intricate, the proof of \Prop~\ref{prop_greg} mostly consists of `routine calculus'.
As a convenient shorthand we introduce $$\zeta_\lambda(z)=\zeta_{d,k,\lambda}(z)=\phi_{d,k,\lambda}(z)-z = \mk{1-\exp\bc{-\lambda -d z^{k-1}} - z}.$$
Its derivatives read
\begin{align}
\zt'(z)  
 &=d(k-1)  \; z^{k-2} \;  \expld
 \; -1   \;
 \label{zt'}
 \quad\text{and} \\
\zt''(z)
  &=
  d (k-1) \; z^{k-3} \; \expld \; 
  \left[  (k-2) - d(k-1) z^{k-1} \right] .
\label{zt''}
\end{align}
Also let
\begin{align}\label{inflection}
    z_0 &= z_0(d,k) = \left(\frac{k-2}{d(k-1)}\right)^{\frac{1}{{k-1}}}. 
\end{align}
We begin by investigating the zeros of $\zt$, obviously identical with fixed points of $\ph$.

\begin{lemma} \label{Claim13}
	Assume that $\lambda>0$.
	\begin{enumerate}[(i)]
		\item The function $\zt$ has either one or three zeros in $z \in [0,1]$, possibly including multiple zeros.
			If $\zt$ has three zeros, then at least one lies in the interval $[0,z_0]$ and at least one lies in the interval $[z_0,1]$.
		\item Also, $\zt$ has at most two stationary points, a minimum and a maximum, and if it has both, the minimum occurs left of the maximum.
		\item If $\zt$ has a unique zero, then $\alpha_*$ is a stable fixed point of $\ph$ and $\sup_{z\in[0,1]}\ph'(z)<1$.
		\item If $\zt$ has three zeros but no double zero, then $\alpha_*,\alpha^*$ are stable fixed points of $\ph$. Additionally, $\ph$ possesses an unstable fixed point $\alpha_\mathrm{u}\in(\alpha_*,\alpha^*)$.
			Furthermore, there exists $\eps=\eps(d,\lambda)>0$ such that
			\begin{align*}
				\sup_{z\in[0,\alpha_*+\eps]}\ph'(z)&<1,&\sup_{z\in[\alpha^*-\eps,1]}\ph'(z)&<1.
			\end{align*}
	\end{enumerate}
\end{lemma}
\begin{proof}
Since $\zt(0)>0$ and $\zt(1)<0$, the number of zeros must be odd, so towards (i) it suffices to show that there cannot be more than three zeros. 
Indeed, by Rolle's theorem, between any two zeros of $\zt$ there is a zero of $\zt'$.
So, if $\zt$ had four or more zeros then $\zt'$ would have at least three zeros in $(0,1]$, and in turn $\zt''$ would have at least two.
From \eqref{zt''} it is clear that $\zt''$ has only two zeros, at $z=0$ (outside the relevant range) and at the inflection point where $k-2 = d(k-1) z^{k-1}$, namely for $z=z_0$.
So, $\zt''$ has at most two zeros, thus $\zt$ has at most three zeros, therefore either one or three.

The second assertion follows from $\zt''(z_0)=0$ and that by inspection of \eqref{zt''}, $\zt''(z)$ is decreasing in $z$, so a local minimum of $\zt$ at $z_1$ implies $\zt''(z_1) > 0$ thus $z_1<z_0$, and symmetrically a local maximum at $z_2$ implies that $z_2>z_0$.

Moving on to (iii), we observe that $\zt(\alpha_*)=0$.
Furthermore, since $\zt(0)>0$ while $\zt(1)<0$, we conclude that $\zt'(\alpha_*)<0$, which implies that $0<\ph'(\alpha_*)<1$.
Hence, $\alpha_*$ is a stable fixed point.

With respect to (iv), if $\zt$ has three zeros, then $\alpha_*<\alpha^*$ are the smallest and the largest zero, respectively.
Since we assume that $\zt$ does not have a double zero, the same reasoning as under (iii) shows that $\zt'(\alpha_*)<0$ and thus $0<\ph'(\alpha_*)<1$.
Further, if $\zt$ has three zeros, then by Rolle's theorem and (ii) the function has a local minimum followed by a local maximum, which is followed by the zero $\alpha^*$.
Hence, $\zt'(\alpha^*)<0$, and thus $0<\ph'(\alpha^*)<1$.
\end{proof}

The following statement implies that $\ph$ has only a single fixed point if $d<\dmin$.

\begin{lemma}\label{lem_belowdmin}
	Let $\lambda>0$.
	If $d<\dmin$, then $\zt$ has a unique zero and is strictly decreasing.
\end{lemma}
\begin{proof}
	Suppose that $z$ is a zero of $\zt$.
	Then $\exp(-\lambda-dz^{k-1})=1-z$ and thus
	\begin{align}\label{eq_lem_zero1}
		\ph'(z)&=d(k-1)z^{k-2}\exp(-\lambda-dz^{k-1})=d(k-1)(z^{k-2}-z^{k-1}).
	\end{align}
	The expression on the r.h.s.\ is positive for $z\in(0,1)$ and zero at $z\in\{0,1\}$.
	Moreover, its derivative works out to be
	\begin{align*}
		\frac{\dd}{\dd z}d(k-1)(z^{k-2}-z^{k-1})&=d(k-1)z^{k-3}(k-2-(k-1)z).
	\end{align*}
	Thus, the expression on the r.h.s.\ of~\eqref{eq_lem_zero1} takes its maximum value of	$d((k-2)/(k-1))^{k-2}$ at $z^\dagger=(k-2)/(k-1)$.
	Hence, \eqref{eq_lem_zero1} implies that $\ph'(z)<1$ and thus $\zt'(z)<0$.
	Consequently, the function $\ph$ only has stable fixed points and thus has only a single fixed point by \Lem~\ref{Claim13}.
\end{proof}

Proceeding to average degrees $d>\dmin$, we verify that the values $\lambda_*,\lambda^*$ from \Sec~\ref{sec_results_dc} are well defined and satisfy the inequality~\eqref{eqlambdas}.

\begin{lemma}\label{lem_abovedmin}
	If $d>\dmin$, then the polynomial $d(k-1)z^{k-2}(1-z)-1$ has precisely two roots $0<z_*<z^*<1$ and the values $\lambda_*,\lambda^*$ defined in~\eqref{eqlambdas} satisfy $\lambda_*>\lambda^*$.
	Furthermore, $\dcore>\dmin$ and $\lambda^*=0$ iff $d\geq\dcore$.
\end{lemma}
\begin{proof}
	Let $z^\dagger=(k-2)/(k-1)$.
	The polynomial $z^{k-2}(1-z)$ is non-negative on $[0,1]$, strictly increasing on $[0,z^\dagger]$ and strictly decreasing on $[z^\dagger,1]$.
	Hence, at $z^\dagger$ the polynomial attains its maximum value of
	\begin{align}\label{eqmaxpoly}
		\max_{0\leq z\leq 1}z^{k-2}(1-z)&=\frac{(k-2)^{k-2}}{(k-1)^{k-1}}.
	\end{align}
	If $d>\dmin$, the equation 
	\begin{align} \label{dOfkz}
		z^{k-2}(1-z)&=\frac1{d(k-1)}
	\end{align}
	therefore has two distinct solutions $0<z_*<z^\dagger<z^*<1$.
	Letting
	\begin{align*}
		\fl(z)=-\ln(1-z)-\frac{z}{(1-z)(k-1)},
	\end{align*}
	we obtain $\lambda_*=\fl(z_*)$ and $\lambda^*=\max\{\fl(z^*),0\}$.

	The function $\fl(z)$ is positive and monotonically increasing on $(0,z^\dagger)$, and monotonically decreasing on $(z^\dagger,1)$.
	Indeed, the derivative works out to be
	\begin{align}\label{eqfunnyderiv}
		\fl'(z)&=\frac{k-2-(k-1)z}{(k-1)(1-z)^2},
	\end{align}
	which is positive for small $z>0$ and has its unique zero at $z^\dagger$.
	Since $z_*<z^\dagger$, we conclude that $\lambda_*>0$.

	Further, \cite[\Thm~1.2]{Maurice} shows that at $d=\dcore$ we have $\fl(z^*)=0$.
	Since $z^*$ is an increasing function of $d$ while $\fl(z)$ is strictly decreasing in $z>z^\dagger$, we conclude that $\fl(z^*)<0$ for $d>\dcore$, $\fl(z^*)=0$ for $d=\dcore$ and $\fl(z^*)=\lambda^*>0$ for $\dmin<d<\dcore$.

	Thus, we are left to verify that $\lambda_*>\lambda^*$, which amounts to showing that $\fl(z^*)<\fl(z_*)$.
	Rearranging \eqref{dOfkz} into $d=1/((k-1)(1-z_*)z_*^{k-2})$ and $d=1/((k-1)(1-z^*)z^{*\,k-2})$ and applying the inverse function theorem, we obtain
	\begin{align}\label{eqfunnyderiv2}
		\frac{\partial z_*}{\partial d}&=-\frac{(k-1)(1-z_*)^2z_*^{k-1}}{k-2-(k-1)z_*},&
		\frac{\partial z^*}{\partial d}&=-\frac{(k-1)(1-z^*)^2z^{*\,k-1}}{k-2-(k-1)z^*}.
	\end{align}
	Combining~\eqref{eqfunnyderiv} and~\eqref{eqfunnyderiv2} with the chain rule, we arrive at
	\begin{align}\label{eqfunnyderiv3}
		\frac\partial{\partial d}\fl(z_*)&=-z_*^{k-1},&	
		\frac\partial{\partial d}\fl(z^*)&=-z^{*\,k-1}.
	\end{align}
	Since $z^*>z_*$ for all $d>\dmin$, integrating~\eqref{eqfunnyderiv3} on $d$ shows that $\lambda_*>\lambda^*$, thereby completing the proof.
\end{proof}

We are ready to identify the zeros of $\zt$ for $d>\dmin$, depending on the regime of $\lambda$.

\begin{lemma}\label{lem_zero}
	Let $\lambda>0$ and assume that $d>\dmin$.
	\begin{enumerate}[(i)]
		\item If  $\lambda<\lambda^*$, then $\zt$ has a unique zero.
		\item If  $\lambda^*<\lambda<\lambda_*$, then $\zt$ has three distinct zeros.
		\item If  $\lambda>\lambda_*$, then $\zt$ has a unique zero.
	\end{enumerate}
\end{lemma}
\begin{proof}
	Assume that $d>\dmin$.
	For fixed $k$ and $d$, the function $\zeta_\lambda$ varies continuously with $\la$, so there are contiguous regimes of $\la$ where it has one zero, regimes where it has three zeros, and these regimes are divided by critical values of $\la$ where $\zt$ has three zeros two of which consist of a double zero.
	In this case, the slope at the double zero is also~0.
	(By Rolle's theorem, the slope is 0 somewhere between the two zeros, and this is the limiting case.)

	Thus, the separation between the regimes with one and three zeros occurs at values of $\lambda$ such that $\zt(z)=\zt'(z)=0$.
	Recalling the definition of $\zt$ and the derivative $\zt'$ from~\eqref{zt'}, we obtain
	\begin{align} \label{zt0}
		1-z=&\expld&\mbox{and}&& d(k-1)z^{k-2}=\frac{1}{\expld} .
	\end{align}
	Substituting the left equation for the exponential in the right equation, we conclude that \eqref{zt0} holds only if $z$ is a solution to \eqref{dOfkz}.
	Further, substituting the two solutions $0<z_*<z^\dagger=(k-2)/(k-1)<z^*$ into either one of the equations from~\eqref{zt0} and solving for $\lambda$, we obtain
	\begin{align*}
		\lambda_*&=-\ln(1-z_*) -\frac{z_*}{(1-z_*)(k-1)},&
		\lambda^\star&=- \ln(1-z^*) -\frac{z^*}{(1-z^*)(k-1)}.
	\end{align*}
	Observe that $\lambda^*=\max\{\lambda^\star,0\}$.

	Suppose $0<\lambda<\lambda_*$.
	Since $\zeta_{\lambda_*}(z_*)=0$, the function $\lambda\mapsto\zeta_\lambda(z_*)$ is strictly increasing and $\zeta_\lambda(0)>0$, we conclude that $\zeta_\lambda$ has a zero in the interval $(0,z_*)$.
	Similarly, if $\lambda>\lambda^*$, then the function $\zeta_\lambda$ has a zero in the interval $(z^*,1)$.
	Hence, (ii) is an immediate consequence of \Lem~\ref{Claim13}.

	Now assume that $0<\lambda<\lambda^*$.
	Since $\lambda_*>\lambda^*$ by \Lem~\ref{lem_abovedmin}, \Lem~\ref{Claim13} implies that $\zeta_{\lambda^*}$ has precisely three zeros. 
	The largest one is $\alpha^*=z^*$, satisfies $\alpha^*>z^\dagger>z_0$, is a double zero and simultaneously a local maximum of $\zeta_{\lambda^*}$.
	Since $\alpha^*$ is a double zero and a local maximum, the smallest zero $\alpha_*$ satisfies $\alpha_*<z_0$ by Rolle's theorem.
	Hence, $\zeta_{\lambda^*}'(z)<0$ for all $0<z<\alpha_*$.
	Since the function $\lambda\mapsto\zeta_\lambda(z)$ is strictly increasing for all $z\in(0,1)$, \Lem~\ref{Claim13}(i) implies that for $\lambda<\lambda^*$ only a single zero remains, which is smaller than $z_0$.

	Finally, suppose that $\lambda>\lambda_*>\lambda^*$.
	\Lem~\ref{Claim13}(i) implies that $\zeta_{\lambda_*}$ has precisely three zeros, with a double zero occurring at $z_*$ and another zero at $\alpha^*(\lambda_*)>z^\dagger>z_0$.
	By \Lem~\ref{Claim13} and the choice of $z_*,\lambda_*$, the double zero at $z_*$ is a local minimum.
	Therefore, $\zeta_{\lambda_*}'(z)<0$ for all $z>\alpha^*$.
	Since the function $\lambda\mapsto\zeta_\lambda(z)$ is strictly increasing for all $z\in(0,1)$, we conclude that $\zeta_\lambda(z)>0$ for all $\lambda>\lambda_*$ and $z\in[0,z_0]$.
	Hence, by \Lem~\ref{Claim13}(i) for $\lambda>\lambda_*$ only a single zero remains, which lies in the interval $[z_0,1]$.
\end{proof}

Combining \Lem s~\ref{lem_abovedmin} and~\ref{lem_zero}, we can now verify the analytic properties of the function $\lambda\mapsto\alpha_*$ and $\lambda\mapsto\alpha^*$.

\begin{lemma}\label{lem_alphas}
	Let $0<d<\dsat$ and $\lambda>0$.
	\begin{enumerate}[(i)]
		\item If $d<\dmin$, then the function $\lambda\in(0,\infty)\mapsto\alpha_*=\alpha^*$ is analytic with derivative
		\begin{align}\label{eq_alphas}
	\frac{\partial\alpha_*}{\partial \lambda}&=\frac{1- \alpha_*}{1-d(k-1)\alpha_*^{k-2}(1-\alpha_*)}<1.
	\end{align}
		\item If $d>\dmin$, then $\lambda\in(0,\lambda_*)\mapsto\alpha_*$ is analytic with derivative \eqref{eq_alphas}.
\item If $d>\dmin$, then $\lambda\in(\lambda^*,\infty)\mapsto\alpha^*$ is analytic differentiable with derivative
\begin{align*}
	\frac{\partial\alpha^*}{\partial \lambda}&=\frac{1- \alpha^*}{1-d(k-1)\alpha^{*\,k-2}(1-\alpha^*)}.
	\end{align*}
	\end{enumerate}
\end{lemma}
\begin{proof}
	Assume that $d>\dmin$ and $\lambda\in(0,\lambda_*)$.
	We know from the proof of \Lem~\ref{lem_zero} that $z_*$ is a double root and local minimum of $\zeta_{\lambda_*}$.
	Furthermore, $z_*<z_0$ and the function $\lambda\mapsto\zeta_\lambda(z)$ is strictly increasing in $\lambda$. 
	Hence, \Lem~\ref{Claim13} implies that for any $0<\lambda<\lambda_*$ the function $\zt$ has a unique zero in $(0,z_*)$.
	Similarly, if $d<\dmin$ then \Lem~\ref{lem_belowdmin} shows that $\zt$ has a unique zero at $\alpha_*$.
	Therefore, the implicit function theorem implies that in cases (i) and (ii) the function $\lambda\mapsto\alpha_*$ is continuously differentiable.

	Thus, we are left to work out $\partial \alpha_*(d,k,\lambda)/\partial \lambda$.
	Consider the function $\cA:\binom z \lambda\mapsto\binom{\zt(z)}{\lambda},$ which is one-to-one in an open interval around $\alpha_*$.
	The Jacobi matrix reads
	\begin{align*}
		D\cA&=\begin{pmatrix} \partial(\phi_{d,k,\lambda})/\partial\alpha-1&\partial\phi_{d,k,\lambda}/\partial \lambda\\0&1 \end{pmatrix}.
	\end{align*}
	Furthermore,
	\begin{align*}
		\frac{\partial\phi_{d,k,\lambda}}{\partial\alpha}\big|_{\alpha=\alpha_*}&=d(k-1)\alpha^{k-2}\exp(-\lambda-d\alpha^{k-1})\big|_{\alpha=\alpha_*}=d(k-1)\alpha_*^{k-2}(1-\alpha_*),\\
		\frac{\partial\phi_{d,k,\lambda}}{\partial t}\big|_{\alpha=\alpha_*}&=\exp(-\lambda-d\alpha^{k-1})\big|_{\alpha=\alpha_*}=1-\alpha_*.
	\end{align*}
	Hence, by the inverse function theorem the derivative of $\cA^{-1}$ reads
	\begin{align*}
		(D\cA)^{-1}&=\begin{pmatrix} \brk{\partial\phi_{d,k,\lambda}/\partial\alpha-1}^{-1}&-\brk{\partial\phi_{d,k,\lambda}/\partial \lambda}/\brk{\partial\phi_{d,k,\lambda}/\partial\alpha-1}\\0&1 \end{pmatrix},&&\mbox{and thus}\\
	\frac{\partial\alpha_*}{\partial \lambda}&=-\frac{\partial\phi_{d,k,\lambda}/\partial \lambda}{\partial\phi_{d,k,\lambda}/\partial\alpha-1}=\frac{1- \alpha_*}{1-d(k-1)\alpha_*^{k-2}(1-\alpha_*)}.
	\end{align*}
	Thus, we obtain (i) and (ii).
	A similar argument applies to $\lambda\in(\lambda^*,\infty)\mapsto\alpha^*$ in the case $d>\dmin$ and yields (iii).
\end{proof}

As a final preparation towards the proof of \Prop~\ref{prop_greg} we investigate the solution $\lcond$ to the differential equation~\eqref{eqLena}; notice that \Lem~\ref{lem_alphas} shows that this ODE does indeed possess a unique solution on $(\dmin,\dsat]$.

\begin{lemma}\label{lem_maxPhi}
	For any $0<d<\dsat$ we have $0<\lcond<\lambda_*$.
	Furthermore, for all $0<\lambda<\lcond$ we have $\Phi_{d,k,\lambda}(\alpha_*)>\Phi_{d,k,\lambda}(\alpha^*)$, while $\Phi_{d,k,\lambda}(\alpha_*)<\Phi_{d,k,\lambda}(\alpha^*)$ for $\lambda_*>\lambda>\lcond$. 
\end{lemma}
\begin{proof}
	For $d<\dsat$ define 
	\begin{align}\label{eqlem_maxPhi0}
		\lcond^*=\inf\{\lambda\geq0:\Phi_{d,k,\lambda}(\alpha^*)>\Phi_{d,k,\lambda}(\alpha_*)\}.
	\end{align}
	For any $d<\dsat$ we have $\Phi_{d,k,0}(0)>\Phi_{d,k,0}(z)$ for all $0<z\leq1$; this follows from the characterisation of the $k$-XORSAT threshold from~\cite[\Thm~1.1]{Ayre}.
	Hence, $\lcond^*>0$ for all $d<\dsat$.

	Further, the function $\zeta_{\lambda_*}$ has a double zero and a local minimum at $\alpha_*=z_*$.
	Since the sign of $\zeta_{\lambda_*}(z)$ matches the sign of $\Phi_{d,k,\lambda_*}'(z)$, this means that $\Phi_{d,k,\lambda_*}(\alpha^*)>\Phi_{d,k,\lambda_*}(\alpha_*)$.
	Hence, there exists $\eps>0$ such that for $0<\lambda_*-\eps<\lambda<\lambda_*$ we have $\Phi_{d,k,\lambda}(\alpha^*)>\Phi_{d,k,\lambda}(\alpha_*)$.
	Therefore,
	\begin{align}\label{eqlem_maxPhi1}
		0<\lcond^*<\lambda_*.
	\end{align}

	As a next step we show that 
\begin{align}\label{eqlem_maxPhi2}
	\Phi_{d,k,\lambda}(\alpha_*)<\Phi_{d,k,\lambda}(\alpha^*)&&\mbox{ for $\lambda_*>\lambda>\lcond^*$}.
	\end{align}
	To this end, we compute the derivatives of $\Ph(\alpha_*)$, $\Ph(\alpha^*)$ with respect to $0<\lambda<\lambda_*$.
	Since $\alpha_*,\alpha^*$ are stationary points of $\Ph$, the chain rule yields
	\begin{align}\label{eq_lem_maxPhi_1}
		\frac\partial{\partial\lambda}\Ph(\alpha_*)&=\frac{\partial\Ph}{\partial\lambda}\big|_{\alpha_*}+\frac{\partial\Ph}{\partial\alpha}\big|_{\alpha_*}\frac{\partial\alpha_*}{\partial\lambda}=\frac{\partial\Ph}{\partial\lambda}\big|_{\alpha_*}=-\exp(-\lambda-d\alpha_*^{k-1})=\alpha_*-1,\\
		\frac\partial{\partial\lambda}\Ph(\alpha^*)&=\alpha^*-1.\label{eq_lem_maxPhi_2}
	\end{align}
	Since $\alpha_*<\alpha^*$ for all $\lambda^*<\lambda<\lambda_*$, \eqref{eqlem_maxPhi2} follows from~\eqref{eq_lem_maxPhi_1}--\eqref{eq_lem_maxPhi_2}.

	Finally, we verify that $\lcond^*$ equals the solution $\lcond$ to the differential equation~\eqref{eqLena}.
	Recalling the definition \eqref{eqlem_maxPhi0}, we see that it suffices to check that $\Phi_{d,k,\lcond}(\alpha_*)=\Phi_{d,k,\lcond}(\alpha^*)$ for all $\dmin<d<\dsat$.
	To this end, we notice that by definition of $\dsat$ we have $\Phi_{\dsat,k,0}(0)=\Phi_{\dsat,k,0}(\alpha^*)$, in line with the initial condition $\lcond(\dsat)=0$.
	Additionally, we claim that $\lcond(\dsat)$ satisfies
	\begin{align*}
		\frac{\partial\Phi_{d,k,\lcond}}{\partial d}\big|_{\alpha^*}&=\frac{\partial\Phi_{d,k,\lcond}}{\partial d}\big|_{\alpha_*}.
	\end{align*}
	Indeed, using the chain rule and the fact that $\alpha_*,\alpha^*$ are stationary points, with $\lambda=\lambda(d)$ we obtain
	\begin{align*}
		\frac{\partial\Phi_{d,k,\lcond}(\alpha_*)}{\partial d}&=\frac{\partial\Phi_{d,k,\lcond}}{\partial d}\big|_{\alpha_*,\lcond}+\frac{\partial\Phi_{d,k,\lcond}}{\partial \alpha}\big|_{\alpha_*,\lcond}\frac{\partial\alpha_*}{\partial d}+\frac{\partial\Phi_{d,k,\lambda}}{\partial\lambda}\big|_{\alpha_*,\lambda}\\
																	&=\frac{\partial\Phi_{d,k,\lcond}}{\partial d}\big|_{\alpha_*,\lcond}+\frac{\partial\Phi_{d,k,\lambda}}{\partial\lambda}\big|_{\alpha_*,\lambda}
																	=\alpha_*^{k-1}+(\alpha_*-1)\frac{\partial\lcond}{\partial d}.
	\end{align*}
	Analogously,
	\begin{align*}
		\frac{\partial\Phi_{d,k,\lcond}(\alpha^*)}{\partial d}&=\alpha^{*\,k-1}+(\alpha^*-1)\frac{\partial\lcond}{\partial d}.
	\end{align*}
	Hence, the solution $\lcond$ to~\eqref{eqLena} satisfies $\Phi_{d,k,\lcond}(\alpha_*)=\Phi_{d,k,\lcond}(\alpha^*)$, and thus $\lcond=\lcond^*$.
	Therefore, the assertion follows from \eqref{eq_lem_maxPhi_1} and~\eqref{eqlem_maxPhi2}.
\end{proof}

\begin{proof}[Proof of \Prop~\ref{prop_greg}]
	The first assertion is an immediate consequence of \Lem s~\ref{lem_belowdmin} and~\ref{lem_alphas}.
	Moreover, the second assertion follows from \Lem s~\ref{lem_abovedmin}, \ref{lem_zero} and~\ref{lem_alphas}.
	Finally, the last assertion follows from \Lem~\ref{lem_maxPhi}.
\end{proof}

\section{Warning Propagation and local weak convergence}\label{sec_prop_WP}

\noindent
In this section we prove \Prop s~\ref{prop_WP} and~\ref{prop_nlluzn}.
The proofs rely on the concept of local weak convergence.
Specifically, we are going to set up a Galton-Watson process that mimics the local topology of the graph $G(\FDC t)$ up to any fixed depth $\ell$.
Subsequently we will analyse WP on the Galton-Watson tree and argue that the result extends to $G(\FDC t)$.

\subsection{Local weak convergence}\label{sec_lwc}
The construction of the Galton-Watson process $\TT=\TT(d,k,t)$ is pretty straightforward.
The process has two types called {\em variable nodes} and {\em check nodes}.
The process starts with a single variable node $v_0$.
Furthermore, each variable node begets a $\Po(d)$ number of check nodes as offspring, while the offspring of a check node is a $\Bin(k-1,1-t/n)$ number of variable nodes.

Let $\TT$ be the Galton-Watson tree rooted at $v_0$ that this process generates; $\TT$ may be infinite.
Hence, for an integer $\ell$ obtain $\TT^{(\ell)}$ from $\TT$ by deleting all variable/check nodes at distance greater than $2\ell$ from $v_0$.
Thus, $\TT^{(\ell)}$ is a finite random tree rooted at $v_0$.
For any graphs $T,T'$ rooted at $v,v'$, respectively, we write $T\ism T'$ if there is a graph isomorphism $\iota:T\to T'$ such that $\iota(v)=v'$.
Furthermore, for a vertex $v$ of $G(\FDC{t})$ and an integer $\ell$ we let $\partial^{\leq\ell}_{\FDC t}v$ be the subgraph obtained from $G(\FDC t)$ by deleting all vertices at distance greater than $2\ell$ from $v$, rooted at $v$.
Finally, for a rooted graph $g$ and an integer $\ell$ we let $\vN_t^{(\ell)}(g)$ be the number of vertices $v$ of $G(\FDC{t})$ such that $\partial^{\leq\ell}_{\FDC t}v\ism g$.

\begin{lemma}\label{lem_lwc}
	For any rooted tree $g$ we have
	\begin{align}\label{eq_lem_lwc}
		\ex\abs{\vN_t^{(\ell)}(g)-(n-t)\pr\brk{\TT^{(\ell)}\ism g}}=o(n).
	\end{align}
\end{lemma}
\begin{proof}
	The proof is based on a routine second moment argument; that is, we claim that
	\begin{align}\label{eq_lem_lwc_1}
		\ex\brk{\vN_t^{(\ell)}(g)}&=(n-t)\pr\brk{\TT^{(\ell)}\ism g}+o(n),&
		\ex\brk{\vN_t^{(\ell)}(g)^2}&=(n-t)^2\pr\brk{\TT^{(\ell)}\ism g}^2+o(n^2).
	\end{align}
	Combining~\eqref{eq_lem_lwc_1} with the Markov and Chebyshev inequalities then yields the assertion.

	We prove~\eqref{eq_lem_lwc_1} and thereby \eqref{eq_lem_lwc} by induction on $\ell$.
	Recall that $\FDC t$ is a XORSAT instance with variables $x_{t+1},\ldots,x_n$.
	Let us begin with the estimate of the first moment. 
	Due to the linearity of expectation, it suffices to show that
	\begin{align}\label{eq_lem_lwc_2}
		\pr\brk{\partial^{\leq\ell}(\FDC t,x_{t+1})\ism g}=\pr\brk{\TT^{(\ell)}\ism g}+o(1).
	\end{align}

	For $\ell=0$ there is nothing to show.
	Hence, suppose that \eqref{eq_lem_lwc} is true with $\ell$ replaced by $\ell-1$.
	Furthermore, let $\Delta$ be the degree of the root $r$ of $g$ and let $1\leq\kappa_1\leq\ldots\leq\kappa_\Delta\leq k$ be the degrees of the children of the root; thus, we order the children of $r$ so that their degrees are increasing.
	For an integer $1\leq i\leq k$ let $K_i$ be the number $j\in[\Delta]$ such that $\kappa_j=i$.
	Further, let $(g_{i,j})_{1\leq i\leq \Delta,1\leq j\leq\kappa_i}$ be the trees pending on the grandchildren of the root.
	In addition, let $\vec\Delta$ be the degree of $x_{t+1}$ in $G(\FDC t)$ and let $1\leq\vec\kappa_1\leq\ldots\leq\vec\kappa_{\vec\Delta}\leq k$ be the degrees of the neighbours of $x_{t+1}$.
	Then $\partial^{\leq\ell}(\FDC t,x_{t+1})\ism g$ is possible only if $\vec\Delta = \Delta$ and $\vec\kappa_i=\kappa_i$ for all $1\leq i\leq\Delta$.
	Since the clauses of the random formula $\PHI$ are drawn uniformly and independently and $G(\FDC{t})$ is obtained from $G(\PHI)$ by deleting the variable nodes $x_1,\ldots,x_t$ along with any ensuing isolated check nodes, we conclude that the event $\cD=\{\vec\Delta=\Delta,\,\bigwedge_{1\leq i\leq\Delta}\vec\kappa_i=\kappa_i\}$ has probability
	\begin{align}\label{eq_lem_lwc_3}
		\pr\brk{\cD}&=\pr\brk{\Po(d)=\Delta} \binom{\Delta}{\kappa_1,\ldots,\kappa_k} \prod_{i=1}^k\pr\brk{\Bin(k-1,1-t/n)=i}^{\kappa_i}.
	\end{align}

	Further, let $\cG=\{g_{i,j}:1\leq i\leq\Delta,\,1\leq j\leq\kappa_i\}$ and let $\cE$ be the event that $\vN_t^{(\ell-1)}(\gamma)=(n-t)\pr\brk{\TT^{(\ell-1)}\ism \gamma}+o(n)$ for all $\gamma\in\cG$.
	Then by induction we have
	\begin{align}\label{eq_lem_lwc_4}
		\pr\brk{\cE\mid\cD}&=1-o(1).
	\end{align}
	Now, obtain $G^-(\FDC t)$ from $G(\FDC t)$ by deleting $x_{t+1}$ along with its adjacent check nodes.
	Let $\vN_t^{(\ell),-}(g_{i,j})$ be the number of vertices $v$ of $G^{-}(\FDC{t})$ such that $\partial^{\leq\ell}(G^-(\FDC t),v)\ism g_{i,j}$.
	Moreover, let $\cE^-$ be the event that $\vN_t^{(\ell-1),-}(g_{i,j})=(n-t)\pr\brk{\TT^{(\ell-1)}\ism g_{i,j}}+o(n)$ for all $i,j$.
	Since $x_{t+1}$ has degree $\Delta=O(1)$ given $\cD$ and all adjacent check nodes have degree at most $k$, \eqref{eq_lem_lwc_4} implies that
	\begin{align}\label{eq_lem_lwc_5}
		\pr\brk{\cE^-\mid\cD}&=1-o(1).
	\end{align}
	Finally, since $\FDC{t}$ is uniformly random, given $\cD$ the checks $a$ of $\FDC t$ adjacent to $x_{t+1}$ simply choose their other neighbours uniformly at random from the variable nodes $x_{t+2},\ldots,x_n$ of $G^-(\FDC t)$.
	Therefore, \eqref{eq_lem_lwc_3} implies that
	\begin{align*}
		\pr\brk{\partial^{\leq\ell}(\FDC t,x_{t+1})\ism g}&=
		\pr\brk{\TT^{(\ell)}\ism g}+o(1),
	\end{align*}
	thereby proving \eqref{eq_lem_lwc_2} and thus the first part of \eqref{eq_lem_lwc_1}.

	The proof of the second part of \eqref{eq_lem_lwc_1} (the estimate of the second moment) proceeds along similar lines, except that we need to explore the depth-${2}\ell$ neighbourhoods of two variable nodes of $\FDC t$ simultaneously.
	Specifically, the proof of the second moment bound comes down to showing that
\begin{align}\label{eq_lem_lwc_10}
		\pr\brk{\partial^{\leq\ell}(\FDC t,x_{t+1})\ism g,\, \partial^{\leq\ell}(\FDC t,x_{t+2})\ism g}=\pr\brk{\TT^{(\ell)}\ism g}^2+o(1).
	\end{align}
	Exploiting that the variable nodes $x_{t+1},x_{t+2}$ are at distance greater than $4\ell$ \whp, we conduct a similar induction as above to verify~\eqref{eq_lem_lwc_10} and thus \eqref{eq_lem_lwc_1}.
\end{proof}

\subsection{Proof of \Prop~\ref{prop_WP}}\label{sec_proof_prop_WP}
To prove \Prop~\ref{prop_WP} we estimate the sizes $|V_{\nll,\ell}(\FDC{t})|,|V_{\fzn,\ell}(\FDC{t})|$ separately.
Recall that $\theta \sim t/n$.

\begin{lemma}\label{lem_WP_n}
	Let $\eps>0$ and assume that one of the following conditions is satisfied:
	\begin{enumerate}[(i)]
		\item $d<\dmin$, or
		\item $d>\dmin$ and $|\theta_*-\theta|>\eps$. 
	\end{enumerate}
	Then there exists $\ell_0=\ell_0(d,\eps)>0$ such that for any fixed $\ell\geq\ell_0$ with $\lambda=-\log(1-\theta)$ \whp\ we have
	$$\abs{t+|V_{\nll,\ell}(\FDC{t})|-\alpha_*n}<\eps n.$$
\end{lemma}
\begin{proof}
	In light of \Lem~\ref{lem_lwc} it suffices to investigate WP on the random tree $\TT^{(\ell)}$ for large enough $\ell$.
	Specifically, let $p^{(\ell)}$ be the probability that WP marks the root of $\TT^{(\ell)}$ as $\nll$.
	In formulas, recalling~\eqref{eqmarks}, this means that
	\begin{align}\label{eqlem_WP_n0}
	p^{(\ell)}&=\pr\brk{\omega_{\TT^{(\ell)},r,\ell}=\nll}\mbox{ for $\ell\geq1$,}&\mbox{and }p^{(0)}&=0. 
	\end{align}
	Let $\vec\Delta$ be the degree of the root $r$ of $\TT^{(\ell)}$ and let $\vec\kappa_1,\ldots,\vec\kappa_{\vec\Delta}$ be the degrees of the children of $r$.
	Since the sub-trees of $\TT^{(\ell)}$ pending on the grandchildren of $r$ are independent copies of $\TT^{(\ell-1)}$, the WP update rules \eqref{eqWP1}--\eqref{eqWP2} yield the recurrence
	\begin{align}\label{eqlem_WP_n1}
		p^{(\ell)}&=1-\ex\brk{\prod_{i=1}^{\vec\Delta}\bc{1-\prod_{j=0}^{\vec\kappa_i-1}p^{(\ell-1)}}}&&(\ell>0).
	\end{align}
	By the construction of $\TT$ the degree $\vec\Delta$ of $r$ has distribution $\Po(d)$.
	Furthermore, each child of $r$ has $\Bin(k-1,1-t/n)$ children; thus, $\vec\kappa_i-1\disteq\Bin(k-1,1-t/n)$.
	Consequently, \eqref{eqlem_WP_n1} yields
	\begin{align}\nonumber
		p^{(\ell)}&=1-\exp(-d)
		\sum_{\Delta=0}^\infty\frac{d^\Delta}{\Delta!}\bc{1-\sum_{\kappa=0}^{k-1}\binom{k-1}\kappa \exp(-\lambda\kappa)(1-\exp(-\lambda))^{k-1-\kappa} p^{(\ell-1)\,\kappa}}^\Delta\\
				  &=1-\exp\bc{-d\bc{1-\exp(-\lambda)(1-p^{(\ell-1)})}^{k-1}}.\label{eqlem_WP_n2}
	\end{align}
	Letting $z^{(\ell)}=1-\exp(-\lambda)(1-p^{(\ell)})$ and recalling the definition~\eqref{eqphi} of $\ph$, we see that \eqref{eqlem_WP_n2} amounts to
	\begin{align}\label{eqlem_WP_n3}
		z^{(\ell)}&=\ph(z^{(\ell-1)}).
	\end{align}
	Moreover, \Lem~\ref{Claim13} (iii)--(iv), \Lem~\ref{lem_belowdmin} and \Lem~\ref{lem_zero} show that if (i) or (ii) above hold, then  $\ph$ is a contraction on $[0,\alpha_*]$.
	Therefore, \eqref{eqlem_WP_n3} shows that $\lim_{\ell\to\infty}p^{(\ell)}=\frac{\alpha_*-\theta}{1-\theta}$.
	Thus, the assertion follows from \Lem~\ref{lem_lwc}.
\end{proof}

\begin{lemma}\label{lem_WP_f}
	Let $\eps>0$ and assume that $d>0$, $t=t(n)$ are such that one of the following conditions is satisfied:
	\begin{enumerate}[(i)]
		\item $d<\dmin$, or
		\item $d>\dmin$, $|\theta_*-\theta|>\eps$ and $|\theta^*-\theta|>\eps$.
	\end{enumerate}
	Then there exists $\ell_0=\ell_0(d,\eps)>0$ such that for any fixed $\ell\geq\ell_0$ with $\lambda=-\log(1-\theta)$ \whp\ we have 
	$$
	\abs{|V_{\fzn,\ell}(\FDC{t})|-(\alpha^*-\alpha_*)n}<\eps n.
	$$
\end{lemma}
\begin{proof}
	Once again it suffices to trace WP on $\TT^{(\ell)}$ for large $\ell$.
	As in the proof of \Lem~\ref{lem_WP_n}, let
	\begin{align}\label{eqlem_WP_f00}
	p^{(\ell)}&=\pr\brk{\omega_{\TT^{(\ell)},r,\ell}\neq\uzn}\mbox{ for $\ell\geq1$,}&\mbox{and }p^{(0)}&=1. 
	\end{align}
	Then with $\vec\Delta$ the degree of $r$ and $\vec\kappa_1,\ldots,\vec\kappa_{\vec\Delta}$ the degrees of the children of $r$, the WP update rules \eqref{eqWP1}--\eqref{eqWP2} translate into
	\begin{align}\label{eqlem_WP_f0}
		p^{(\ell)}&=1-\ex\brk{\prod_{i=1}^{\vec\Delta}\bc{1-\prod_{j=0}^{\vec\kappa_i-1}p^{(\ell-1)}}}&&(\ell>0),
	\end{align}
	Thus, the recurrence is identical to~\eqref{eqlem_WP_n1}, but this time with the initial condition $p^{(0)}=1$.
	Hence, letting $z^{(\ell)}=1-\exp(-\lambda)(1-p^{(\ell)})$ and $z^{(0)}=1$ and retracing the steps towards \eqref{eqlem_WP_n3}, we obtain
	\begin{align}
		z^{(\ell)}=1-\exp(-\lambda)(1-p^{(\ell)}).\label{eqlem_WP_f3}
	\end{align}
	Invoking \Lem s~\ref{Claim13}, \ref{lem_belowdmin} and~\ref{lem_zero}, we conclude that (i) or (ii) ensure that $\ph$ contracts on $[0,\alpha^*]$.
	Consequently, \eqref{eqlem_WP_f3} implies that $\lim_{\ell\to\infty}p^{(\ell)}=\frac{\alpha^*-\theta}{1-\theta}$.
	Thus, the assertion follows from \Lem s~\ref{lem_lwc} and~\ref{lem_WP_n}.
\end{proof}

Finally, we compare the set $V_{\nll,\ell}(\FDC{t})$ obtained after a (large but) bounded number of iterations with the ultimate sets $V_{\nll}(\FDC{t})$ obtained upon convergence of WP.
The proof of the following lemma is an adaptation of the argument from~\cite{Molloy} for cores of random hypergraphs.

\begin{lemma}\label{lem_WP_n_infty}
	Assume that $\theta\in(0,1)\setminus\{\theta_*,\theta^*\}$.
	Then for any $\eps>0$ there exists $\ell_0=\ell_0(d,\theta, \eps)$ such that for all $\ell>\ell_0$ we have $|V_{\nll,\ell}(\FDC{t})\triangle V_{\nll}(\FDC{t})|<\eps n$ \whp
\end{lemma}
\begin{proof}
	In place of the WP message passing process from \Sec~\ref{sec_wp} we consider the following simpler peeling process, which reproduces the same set $V_{\nll}(\FDC t)$.
	Let $\vG_0=G(\PHI)$ be the bipartite graph induced by $\FDC t$.
	For $h\geq0$ obtain $\vG_{h+1}$ from $\vG_h$ by performing the following peeling operation.
	\begin{align}\label{eq_lem_WP_n_infty_1}
		\parbox{0.8\textwidth}{Remove all check nodes of degree one along with their variable node neighbours.}
	\end{align}
	Clearly, this process will reach a fixed point (i.e., $\vG_{h+1}=\vG_h$) after at most $\vm$ iterations.
	Moreover, a straightforward induction on $\ell$ shows that $V(\vG_0)\setminus V(\vG_\ell)=V_{\nll,\ell}(\FDC t)$ and thus $V(\vG_0)\setminus V(\vG_{\vm})=V_{\nll}(\FDC t)$.
	Hence, it suffices to prove that for large enough $\ell=\ell(d,\theta, \eps)$ we have
	\begin{align}\label{eq_lem_WP_n_infty_2}
		|V(\vG_{\ell})\triangle V(\vG_{\vm})|&<\eps n&&\mbox\whp
	\end{align}

	Towards the proof of~\eqref{eq_lem_WP_n_infty_2} let $\vd_h=(\vd_h(u))_{u\in V(\vG_h)\cup C(\vG_h)}$ be the degree sequence of $\vG_h$.
	By the principle of deferred decisions $\vG_h$ is uniformly random given $\vd_h$.
	Further, let
	\begin{align*}
		\vec\Delta_h(j)&=\abs{\cbc{x\in V(\vG_h):\vd_h(x)=j}},&
		\vec\Delta_h'(j)&=\abs{\cbc{a\in C(\vG_h):\vd_h(a)=j}}
	\end{align*}
	be the number of variable/check nodes of degree $j\geq0$.
	Pick $\delta=\delta(d,\theta,\eps)$, $\delta'=\delta'(d,\theta,\delta)$, $\delta'' = \delta''(d,\theta,\delta')$,  small enough and $\ell\geq\ell_0(d,\theta,\delta'')$ large enough.
	Then \Lem~\ref{lem_WP_n} implies that \whp 
	\begin{align}\label{eq_lem_WP_n_infty_3}
		|V(\vG_{\ell})\setminus V(\vG_{\ell+1})|<\delta'' n.
	\end{align}
	
	Furthermore, we claim that 
	\begin{align}\label{eq_lem_WP_n_infty_4}
		\sum_{j\geq0}\abs{\frac{\vec\Delta_\ell(j)}{|V(\vG_\ell)|}-\pr\brk{\Po(d(1-\alpha_*^{k-1}))=j}}&<\delta',&		\sum_{j\geq2}\abs{\frac{\vec\Delta'_\ell(j)}{|C(\vG_\ell)|}-\frac{\pr\brk{\Bin(k,1-\alpha_*)=j}}{\pr\brk{\Bin(k,1-\alpha_*)\geq2}}}&<\delta'.
	\end{align}
	Indeed, \Lem~\ref{lem_lwc} shows that we just need to study WP on the random tree $\TT^{(\ell)}$, as in the proof of \Lem~\ref{lem_WP_n}.
	Thus, let $\vec\Delta$ be the degree of the root variable and let $\vec\kappa_1,\ldots,\vec\kappa_{\vec\Delta}$ be the degrees of the children of the root.
	Since the sub-trees pending on the children of the root are independent copies of $\TT^{(\ell-1)}$, \Lem~\ref{lem_WP_n} shows that the probability that any one of the $\vec\Delta$ children sends a $\nll$-message to $r$ falls into the interval $(1-\alpha_*^{k-1}-\delta'',1-\alpha_*^{k-1}+\delta'')$, provided that $\ell$ is large enough.
	Since $\vec\Delta\disteq\Po(d)$, the first part of~\eqref{eq_lem_WP_n_infty_4} follows from Poisson thinning.

	Similarly, to obtain the second part of~\eqref{eq_lem_WP_n_infty_4} consider a clause $a$ that is a child of the root $r$ of $\TT^{(\ell)}$.
	Then by the same token as in the previous paragraph the number of children of $a$ that do not send a $\nll$-message after $\ell$ iterations of WP lies in the interval $(1-\alpha_*^{k-1}-\delta'',1-\alpha_*^{k-1}+\delta'')$.
	Furthermore, the number of children $a'\neq a$ of $r$ has distribution $\Po(\vec\Delta)$.
	Hence, the probability that the WP-message from $r$ to $a$ equals $\nll$ comes to $\alpha_*\pm\delta''$, and this event is independent of the messages that the children of $a$ send to $a$.
	Finally, the probability that one of the messages that $a$ receives after $\ell$ iterations of WP differs from the message received after $\ell-1$ iterations is smaller than $\delta''$ for large enough $\ell$.
	Since the peeling process removes any checks $a$ with at least $k-1$ incoming $\nll$-messages, we obtain~\eqref{eq_lem_WP_n_infty_4}.

	To complete the proof we are going to deduce from~\eqref{eq_lem_WP_n_infty_3}--\eqref{eq_lem_WP_n_infty_4} that the peeling process~\eqref{eq_lem_WP_n_infty_1} will remove no more than $\eps n/2$ further nodes from $\vG_\ell$ before it stops.
	Following~\cite{Molloy}, we consider a slowed-down version of the process where no longer all checks of degree one get removed simultaneously, but rather one-at-a-time.
	Let $(\vG_\ell[\nu])_{\nu\geq0}$ be the sequences of graphs produced by this modified process, with $\vG_\ell[0]=\vG_\ell$ and $\vG_\ell[\nu+1]=\vG_\ell[\nu]$ if all checks of $\vG_\ell[\nu]$ have degree at least two.
	Further, let $\vU_\ell[\nu]$ be the number of unary checks of $\vG_\ell[\nu]$.
	Let $\cD$ be the event that the bounds~\eqref{eq_lem_WP_n_infty_3}--\eqref{eq_lem_WP_n_infty_4} hold.
	Then it suffices to prove that on the event $\{\vU_\ell[\nu]>0\}\cap\cD$ we have
	\begin{align}\label{eq_lem_WP_n_infty_10}
		\ex\brk{\vU_\ell[\nu+1]-\vU_\ell[\nu]\mid\vU_\ell[\nu]}<0&&\mbox{for all }0\leq \nu\leq\eps n/2.
	\end{align}

	Invoking the principle of deferred decisions, in order to verify~\eqref{eq_lem_WP_n_infty_10} we compute the expected number of new degree one checks produced by the removal of a single random variable node $\vx$.
	Due to~\eqref{eq_lem_WP_n_infty_4}, for $\nu\leq\eps n/2$ the expected number of neighbours $a$ of $\vx$ of degree precisely two is bounded by
	\begin{align*}
		d\pr\brk{\Bin(k-1,1-\alpha_*)=1}+\delta&=d(k-1)(1-\alpha_*)\alpha_*^{k-2}+\delta=\ph'(\alpha_*)+\delta<1,
	\end{align*}
	provided that $\delta>0$ is chosen sufficiently small.
	Hence, we obtain \eqref{eq_lem_WP_n_infty_10}.
\end{proof}

%

\begin{proof}[Proof of \Prop~\ref{prop_WP}]
	The proposition is an immediate consequence of \Lem s~\ref{lem_WP_n}--\ref{lem_WP_n_infty}.
\end{proof}

\subsection{Proof of \Prop~\ref{prop_nlluzn}}\label{sec_prop_nlluzn}
We deal with the two claims separately.
Towards the first claim we establish the following stronger, deterministic statement.

\begin{lemma}\label{lem_nll}
	For any XORSAT instance $F$ with variables $V_n=\{x_1,\ldots,x_n\}$ and any integer $\ell\geq0$ we have $V_{\nll,\ell}(F)\subset\frz(F)$.
\end{lemma}
\begin{proof}
	We proceed by induction on $\ell$.
	For $\ell\leq1$ there is nothing to show because $V_{\nll,\ell}(F)=\emptyset$ by construction.
	Hence, assume that $\ell>1$ and that $V_{\nll,\ell'}(F)\subset\frz(F)$ for all $1\leq\ell'<\ell$.
	If $x\in V_{\nll,\ell-1}$, then \eqref{eqmarks} shows that there exists a check node $b\in\partial x$ such that $\omega_{F,b\to x,\ell}=\nll$.
	Furthermore, \eqref{eqWP1} shows that if $\omega_{F,b\to x,\ell}=\nll$, then for all $y\in\partial b\setminus\{x\}$ we have $\omega_{F,y\to b,\ell-1}=\nll$.
	Additionally, \eqref{eqWP2} shows that if $\omega_{F,y\to b,\ell-1}=\nll$, then there exists $a\in\partial y\setminus\{b\}$ such that $\omega_{F,a\to y,\ell-2}=\nll$.
	Hence, \eqref{eqmarks} ensures that $\omega_{F,y,\ell-2}=\nll$ and thus 
	\begin{align}\label{eq_lem_nll1}
		y&\in V_0(F)&&\mbox{for all }y\in\partial b\setminus\{x\}
	\end{align}
	by induction.
	Now suppose that $\partial b=\{x_{j_1},\ldots,x_{j_h}\}$ with pairwise distinct indices $1\leq j_1,\ldots,j_h\leq n$ such that $x=x_{j_1}$.
	Consider $\sigma\in\ker A(F)$.
	Then \eqref{eq_lem_nll1} implies that $\sigma_{j_2}=\cdots=\sigma_{j_h}=0$.
	Consequently, $\sigma_{j_1}=0$ and thus $x\in V_0(F)$.
\end{proof}

The following lemma deals with the variables that WP marks $\uzn$.

\begin{lemma}\label{lem_uzn}
	For any fixed $\ell\geq0$ we have $|V_{\uzn,\ell}(\FDC{t})\cap\frz(\FDC{t})|=o(n)$ \whp 
\end{lemma}

\newcommand{\nb}{\partial^{\leq\ell}}
\begin{proof}
	We are going to show by induction on $\ell$ that $\ex|V_{\uzn,\ell}(\FDC{t})\cap\frz(\FDC{t})|=o(n)$.
	To this end, because the distribution of $\FDC t$ is invariant under permutations of the variables $x_{t+1},\ldots,x_n$, it suffices to show that
	\begin{align}\label{eq_lem_uzn_1}
		\pr\brk{x_n\in V_{\uzn,\ell}(\FDC{t})\cap\frz(\FDC{t})}&=o(1).
	\end{align}
	Indeed, let $\cA$ be the event that the depth-$2\ell$ neighbourhood $\nb x_n$ of $x_n$ in $\FDC t$ is acyclic.
	Since \Lem~\ref{lem_lwc} shows that $\pr\brk{\cA}=1-o(1)$, towards \eqref{eq_lem_uzn_1} it suffices to prove that on the event $\cA$ we have
	\begin{align}\label{eq_lem_uzn_2}
		x_n\not\in V_{\uzn,\ell}(\FDC{t})\cap\frz(\FDC{t}).
	\end{align}
	But~\eqref{eq_lem_uzn_2} follows from the well-known fact that BP is exact on acyclic factor graphs (see Fact \ref{fact_BP_acyclic}). 
\end{proof}

\begin{proof}[Proof of \Prop~\ref{prop_nlluzn}]
	The proposition is an immediate consequence of \Lem s~\ref{lem_nll} and~\ref{lem_uzn}.
\end{proof}

\section{Analysis of the check matrix}\label{sec_nul}

\noindent
In this section we prove \Prop s~\ref{prop_nul} and~\ref{cor_frz}.
\Prop~\ref{prop_nul} is an easy consequence of~\cite[\Thm~1.1]{Maurice}.
Furthermore, \Prop~\ref{cor_frz} follows from \Prop~\ref{prop_nul} by interpolating on the parameter $\lambda$; a related argument was recently used in~\cite{fullrank} to show that certain random combinatorial matrices have full rank \whp\
In addition, we prove Corollaries~\ref{cor_nll} and~\ref{lem_fzn} and subsequently complete the proofs of \Thm s~\ref{thm_recnonrec}--\ref{thm_cond}.

\subsection{Proof of \Prop~\ref{prop_nul}}\label{sec_prop_nul}
We use a general result \cite[\Thm~1.1]{Maurice} about the rank of sparse random matrices from a fairly universal class of distributions.
The definition of this general random matrix goes as follows.
Let $\fd,\fk\geq0$ be integer-values random variables such that $0<\ex[\fd^3]+\ex[\fk^3]<\infty$.
Moreover, let $(\fd_i,\fk_i)_{i\geq0}$ be families of mutually independent random variables such that $\fd_i\disteq\fd$ and $\fk_i\disteq\fk$.
Let $\bar\fd=\ex[\fd]$ and $\bar\fk=\ex[\fk]$ and for an integer $n>0$ let $\fm=\fm_n\disteq\Po(\bar\fd n/\bar\fk)$.
The sequence $(\fm_n)_n$ is independent of $(\fd_i,\fk_i)_{i\geq0}$.
Further, let $\fS_n$ be the event that
\begin{align}\label{eqDegreeMatch}
	\sum_{i=1}^n\fd_i&=\sum_{i=1}^{\fm_n}\fk_i.
\end{align}
It is a known fact that $\pr\brk{\fS_n}=\Omega(n^{-1/2})$~\cite[\Prop~1.10]{Maurice}.
Given that $\fS_n$ occurs, create a simple random bipartite graph $\fG_n$ with a set $\fV_n=\{\fx_1,\ldots,\fx_n\}$ of {\em variable nodes} and a set $\fC_n=\{\fc_1,\ldots,\fc_{\fm_n}\}$ of {\em check nodes} uniformly at random subject to the condition that $\fx_j$ has degree $\fd_j$ and $\fc_i$ has degree $\fk_i$ for all $1\leq j\leq n$ and $1\leq i\leq\fm_n$.
Finally, let $\fA_n$ be the biadjacency matrix of $\fG_n$.
Thus, $\fA_n$ has size $\fm_n\times n$ and its $(i,j)$-entry equals $1$ iff $\fx_j$ and $\fc_i$ are adjacent in $\fG_n$.

\begin{theorem}[{special case of \cite[\Thm~1.1]{Maurice}}]\label{thm_maurice}
	Let $\fD(z)=\sum_{h=0}^\infty\pr\brk{\fd=h}z^h$ and $\fK(z)=\sum_{h=0}^\infty\pr\brk{\fk=h}z^h$ be the probability generating functions of $\fd,\fk$, respectively.
	Furthermore, let 
	\begin{align}\label{eq_thm_maurice}
		\fF&:[0,1]\to\RR,&z\mapsto\fD(1-\fK'(z)/\fK'(1))-\frac{\fD'(1)}{\fK'(1)}(1-\fK(z)-(1-z)\fK'(z)).
	\end{align}
	Then 
		\begin{align*}
			\lim_{n\to\infty}\frac1n\nul\fA_n&=\max_{z\in[0,1]}\fF(z)&&\mbox{in probability.}
		\end{align*}
\end{theorem}

We now derive \Prop~\ref{prop_nul} from \Thm~\ref{thm_maurice} by identifying suitable distributions $\fd,\fk$ such that $\fA_n$ resembles $\vA_t$.

\begin{proof}[Proof of \Prop~\ref{prop_nul}]
	Recall that $0\leq t=t(n)\leq n$ satisfies $t=\theta n+o(n)$ for a fixed $0\leq\theta\leq 1$.
	We continue to set $\lambda=-\log(1-\theta)$.
	We are going to construct several random matrices that can be coupled such that their nullities differ by no more than $o(n)$ \whp\
	The first of these random matrices is the matrix $\vA_t$ from \Prop~\ref{prop_nul}, and the last is the matrix $\fA_n$ from \Thm~\ref{thm_maurice}, with suitably chosen $\fd,\fk$.

	For a start, consider the check matrix $\vA'=\vA_0$ of the original, `undecimated' $k$-XORSAT formula $\PHI=\FDC 0$.
	Obtain $\vA_{t}'$ from $\vA'$ by adding $t$ new rows to $\vA'$.
	Each of these rows contains precisely a single non-zero entry.
	The positions of the non-zero entries are chosen uniformly without replacement.
	Thus, the extra $t$ rows have the effect of fixing $t$ uniformly random coordinates to zero.
	Since the distribution of the random matrix $\vA'$ is invariant under column permutations, we conclude that
	\begin{align}\label{eq_prop_nul_1}
		\nul\vA_t&\disteq\nul\vA_t'.
	\end{align}

	Further, let $\vA[\lambda]$ be the matrix obtained from $\vA'$ by adding a random number of $\vl=\Po(\lambda n)$ of rows.
	Each of these rows contains a single non-zero entry, which is placed in a uniformly random position.
	The extra rows are chosen mutually independently (thus, `with replacement') and independently of $\vA'$.
	By Poisson thinning, for any column index $j\in[n]$ the probability that one of the new $\vl$ rows has a non-zero entry in the $j$th column equals $1-\exp(-\lambda)=\theta$.
	Since $t\sim\theta n$, the total number of such indices $j$ has distribution $\Bin(n,\theta)$.
	Since $\pr\brk{|\Bin(n,\theta)-t|\leq\sqrt n\log n}\geq 1-1/n$ by the Chernoff bound, we can couple $\vA_t'$ and $\vA\brk{\lambda}$ such that 
	\begin{align}\label{eq_prop_nul_2}
		\nul\vA_t'&=\nul\vA\brk{\lambda}+o(n)\mbox{ \whp}	
	\end{align}

	Finally, let $\vA'[\lambda]$ be the matrix obtained as follows.
	Let $\fd,\fk$ have probability generating functions
	\begin{align}\label{eq_prop_nul_3}
		\fD(z)&=\exp((\lambda+d)(z-1)),&\fK(z)&=\frac{dz^k+k\lambda z}{d+k\lambda}.
	\end{align}
	In other words, $\fd$ has distribution $\Po(d+\lambda)$ while $\fk$ equals one with probability $k\lambda/(d+k\lambda)$ and equals $k$ with probability $d/(d+k\lambda)$.
	The definition \eqref{eq_prop_nul_3} readily yields
	\begin{align}\label{eq_prop_nul_4}
		\bar\fd&=\fD'(1)=\lambda+d,&\bar\fk&=\fK'(1)=\frac{k(d+\lambda)}{d+k\lambda}.
	\end{align}
	Hence, the number $\fm=\fm_n\disteq\Po(n\bar\fd/\bar\fk)$ of rows of $\fA=\fA_n$ can be written as a sum of independent random variables $\fm=\fm'+\fm''$ with distributions
	\begin{align}\label{eq_prop_nul_5}
		\fm'&=\Po(dn/k),&\fm''&=\Po(\lambda n).
	\end{align}
	The first summand $\fm'$ prescribes the number of rows of $\fA$ with $k$ non-zero entries, while $\fm''$ details the number of rows with a single non-zero entry.
	Consequently, \eqref{eq_prop_nul_5}  shows that the numbers of rows with $k$ or with just a single non-zero entry have the same distributions in both $\fA$ and $\vA[\lambda]$.
	
	We are left to argue that in $\fA$ the positions of the non-zero entries in the different rows are nearly independent and uniform.
	To see this, let $(\vh_{i,j})_{1\leq i\geq \fm,1\leq j\leq k}$ be a family of mutually independent and uniform random variables with values in $[n]=\{1,\ldots,n\}$.
	Moreover, let $\vX$ be the number of indices $1\leq i\leq\fm'$ such that there exist $1\leq j_1<j_2\leq k$ such that $\vh_{i,j_1}=\vh_{i,j_2}$; in other words, $\vh_{i,1},\ldots,\vh_{i,k}$ fail to be pairwise distinct.
	A routine calculation shows that
\begin{align}\label{eq_prop_nul_6}
		\ex[\vX]&=O(1).
	\end{align}
	Now, let us think of $(\vh_{i,j})_{1\leq i\leq\fm',1\leq j\leq k}$ and $(\vh_{i,1})_{\fm'<i\leq\fm}$ as the `bins' where $k\fm'+\fm''$ randomly tossed `balls' land. 
	Then the standard Poissonisation of the balls-into-bins experiment shows that given the event~\eqref{eqDegreeMatch} the loads of the bins are distributed precisely as the vector $(\fd_1,\ldots,\fd_n)$.
	Therefore, \eqref{eq_prop_nul_6} shows that $\vA[\lambda],\fA$ can be coupled such that
	\begin{align}\label{eq_prop_nul_7}
		\nul\vA[\lambda]&=\nul\fA+o(n)\mbox{ \whp}	
	\end{align}

	Combining \eqref{eq_prop_nul_1}, \eqref{eq_prop_nul_2} and~\eqref{eq_prop_nul_7}, we see that $\vA_t$ and $\fA$ can be coupled such that 
	\begin{align}\label{eq_prop_nul_8}
		\nul\vA_t&=\nul\fA+o(n)&&\mbox\whp
	\end{align}
	Hence, \Thm~\ref{thm_maurice} implies that
	\begin{align}\label{eq_prop_nul_9}
		\lim_{n\to\infty}\frac1n\nul\vA_t&=\max_{z\in[0,1]}\fF(z)&&\mbox{in probability}.
	\end{align}
	Further, recalling the definitions~\eqref{eq_thm_maurice}, \eqref{eq_prop_nul_3} of $\fF,\fD,\fK$ and performing a bit of calculus, we verify that $\fF(z)$ coincides with the function $\Ph(z)$ from~\eqref{eqPhi}.
	Finally, the assertion follows from~\eqref{eq_prop_nul_9} and the fact that $\Ph(\amax)=\max_{z\in[0,1]}\Ph(z)$.
\end{proof}

\subsection{Proof of \Prop~\ref{cor_frz}}\label{sec_cor_frz}
We continue to work with the random matrix $\vA[\lambda]$ from the above proof of \Prop~\ref{prop_nul}.
As we recall, this matrix is obtained by adding $\vl=\Po(\lambda n)$ stochastically independent new rows to the matrix $A(\PHI)$ that each contain a single non-zero entry in a uniformly random position.
Combining~\eqref{eq_prop_nul_1}--\eqref{eq_prop_nul_2}, we see that
\begin{align}\label{eqAA}
	\abs{\ex[\nul\vA[\lambda]]-\ex[\nul\vA_t]}&=o(n)&&\mbox{ for }\lambda=-\log(1-\theta).
\end{align}
Towards the proof of \Prop~\ref{cor_frz} we observe that $\nul\vA[\lambda],\nul\vA_t$ concentrate about their expectations.

\begin{lemma}\label{cor_nul}
	We have 
	\begin{align}\label{eq_cor_nul}
		\pr\brk{|\nul\vA_t-\ex[\nul\vA_t]|>\sqrt n\log n}&=o(n^{-10}),&
		\pr\brk{|\nul\vA\brk\lambda-\ex[\nul\vA\brk\lambda]|>\sqrt n\log n}&=o(n^{-10}).
	\end{align}
\end{lemma}
\begin{proof}
	We combine the Azuma--Hoeffding inequality with the simple observation that the nullity satisfies a Lipschitz condition.
	Specifically, adding or removing a single row to a matrix changes the nullity by at most one.
	We apply this observation to the matrix $\vA_t'$ from the proof of \Prop~\ref{prop_nul}, which consists of $\vm+t$ independent random rows.
	Indeed, Azuma-Hoeffding implies together with the Lipschitz property that
	\begin{align}\label{eq_cor_nul_1}
		\pr\brk{|\vA_t'-\ex[\vA_t'\mid\vm]|>u\mid\vm}&\leq2\exp\bc{-\frac{u^2}{2(\vm+t)}}&&\mbox{for any }u>0.
	\end{align}
	Furthermore, Bennett's concentration inequality for Poisson variables shows that
	\begin{align}\label{eq_cor_nul_2}
		\pr\brk{|\vm-dn/k|>\sqrt n\log^{2/3}n}&=o(n^{-10}).
	\end{align}
	Combining~\eqref{eq_cor_nul_1}--\eqref{eq_cor_nul_2} with the Lipschitz property and setting $u=\sqrt n\log^{2/3}n$, we obtain the first part of~\eqref{eq_cor_nul}.

	Similar reasoning applies to the second matrix $\vA[\lambda]$; for given $\vl$ and $\vm$ the Lipschitz property yields
	\begin{align}\label{eq_cor_nul_3}
		\pr\brk{|\vA_t'-\ex[\vA_t'\mid\vl,\vm]|>u\mid\vl,\vm}&\leq2\exp\bc{-\frac{u^2}{2(\vl+\vm)}}&&\mbox{for any }u>0.
	\end{align}
	Moreover, in analogy to~\eqref{eq_cor_nul_2} we have
\begin{align}\label{eq_cor_nul_4}
		\pr\brk{|\vl-\lambda n|>\sqrt n\log^{2/3}n}&=o(n^{-10}).
	\end{align}
	Thus, \eqref{eq_cor_nul_2}--\eqref{eq_cor_nul_4} and Azuma-Hoeffding imply the second part of \eqref{eq_cor_nul}.
\end{proof}


We are going to estimate $|V_0(\FDC t)|$ by way of estimating changes of $\nul\vA[\lambda]$ as $\lambda$ varies.
Since $\nul\vA[\lambda]/n$ converges to $\Ph(\amax)$ by \Prop~\ref{prop_nul}, we thus need to estimate the derivative $\frac\partial{\partial\lambda}\Ph(\amax)$.

\begin{lemma}\label{lem_nul}
	Let $d>0$ and assume that
	\begin{enumerate}[(i)]
		\item $d<\dmin$, or
		\item $d>\dmin$ and $\lambda\in(0,\infty)\setminus\{\lcond\}$.
	\end{enumerate}
Then
\begin{align}\label{eqlem_nul}
	\frac{\partial}{\partial\lambda}\Phi_{d,k,\lambda}(\amax)=\amax-1.
\end{align}
\end{lemma}
\begin{proof}
	The seeming difficulty is that $\amax=\amax(\lambda)$ varies with $\lambda$.
	Yet \Prop~\ref{prop_greg} (iii) ensures that the function $\lambda\mapsto\amax$ is continuously differentiable for $\lambda\neq\lcond$.
	Moreover, Fact~\ref{lem_max} shows that $\amax$ is a local maximum of $\Phi_{d,k,\lambda}$.
	Hence, applying the chain rule we obtain
	\begin{align}\label{eq_lem_nul_1}
		\frac\partial{\partial\lambda}\Phi_{d,k,\lambda}(\amax)&=\frac{\partial\Phi_{d,k,\lambda}}{\partial\lambda}\bigg|_{\lambda,\amax}+\frac{\partial\Phi_{d,k,\lambda}}{\partial\alpha}\bigg|_{\lambda,\amax}\frac{\partial\amax}{\partial\lambda}=\frac{\partial\Phi_{d,k,\lambda}}{\partial\lambda}\bigg|_{\lambda,\amax}=-\exp\bc{-\lambda-d\amax^{k-1}}.
	\end{align}
	In fact, since Fact~\ref{lem_max} shows that $\amax$ is a fixed point of $\phi_{d,k,\lambda}$, the r.h.s.\ of~\eqref{eq_lem_nul_1} simplifies to~\eqref{eqlem_nul}.
\end{proof}

Complementing the analytic formula~\eqref{eqlem_nul}, we now derive a combinatorial interpretation of the derivative of the nullity.
For a matrix $A$ of size $m\times n$ let $\frz(A)$ be the set of all indices $i\in[n]$ such that $\sigma_i=0$ for all $\sigma\in\ker A$.

\begin{lemma}\label{lem_diff_nul}
	For any $d,\lambda>0$ we have
	\begin{align*}
		\frac\partial{\partial\lambda}\ex[\nul \vA[\lambda]]&=\frac{\ex|\frz(\vA[\lambda])|}n-1.
	\end{align*}
\end{lemma}
\begin{proof}
	Recall that $\vA[\lambda]$ is obtained from $A(\PHI)$ by adding $\vm''\disteq\Po(\lambda n)$ stochastically independent rows with a single non-zero entry in a uniformly random position.
	Consequently,
	\begin{align}
	\frac\partial{\partial\lambda}\ex[\nul\vA[\lambda]]&=\frac\partial{\partial\lambda}\sum_{\ell=0}^\infty\pr\brk{\vm''=\ell}\ex[\nul\vA[\lambda]\mid\vm''=\ell]=\sum_{\ell=0}^\infty\ex[\nul\vA[\lambda]\mid\vm''=\ell]\frac\partial{\partial\lambda}\frac{(\lambda n)^\ell}{\ell!}\exp(-\lambda n)\nonumber\\
	&=\sum_{\ell=0}^\infty\ex[\nul\vA[\lambda]\mid\vm''=\ell]\bc{\vecone\{\ell\geq1\}\frac{(\lambda n)^{\ell-1}}{(\ell-1)!}-\frac{(\lambda n)^\ell}{\ell!}}\exp(-\lambda n)\nonumber\\
	&=\sum_{\ell=0}^\infty\ex[\nul\vA[\lambda]\mid\vm''=\ell]\bc{\pr\brk{\vm''=\ell}-\pr\brk{\vm''=\ell+1}}.
	\label{eq_lem_diff_nul_1}
\end{align}
	Hence, obtain $\vA[\lambda]^+$ from $\vA[\lambda]$ by adding one more row with a single non-zero entry in a uniformly random position $\vj^+\in[n]$.
	Then $\vA[\lambda]^+-\vA[\lambda]=-\vecone\{\vj^+\in V_0(\vA[\lambda])\}$.
	Hence, \eqref{eq_lem_diff_nul_1} yields
	\begin{align*}
		\frac\partial{\partial\lambda}\ex[\nul\vA[\lambda]]&=-\ex\brk{\nul(\vA[\lambda]^+)-\nul(\vA[\lambda])}=\pr\brk{\vec j^+\in V_0(\vA[\lambda])}-1=\frac{\ex|\frz(\vA[\lambda])|}n-1,
	\end{align*}
	as claimed.
\end{proof}

With these preparations in place we can now derive the desired formulas for $|V_0(\vA_t)|$.
We treat the cases $\amax=\alpha_*$ and $\amax=\alpha^*$ separately.

\begin{lemma}\label{lem_nul_*}
	Assume that $d,\lambda>0$ satisfy
	\begin{align}\label{eq_cor_frz_*}
		\Phi_{d,k,\lambda}(\alpha_*)>\Phi_{d,k,\lambda}(\alpha)\quad\mbox{for all $\alpha\in[0,1]\setminus\{\alpha_*\}$.}
	\end{align}
	Then $|\frz(\vA[\lambda])|=\alpha_*n+o(n)$ \whp
\end{lemma}
\begin{proof}
	\Prop~\ref{prop_WP} and \Lem~\ref{lem_nll} yield the lower bound
	\begin{align}\label{eq_lem_nul_*_0}
		|\frz(\vA[\lambda])|&\geq\alpha_*n+o(n)&&\mbox\whp
	\end{align}
	To derive the matching upper bound, fix $\eps>0$ and assume that the event $\cE=\{|\frz(\vA[\lambda])|>(\alpha_*+\eps)n\}$ has probability $\pr\brk\cE>\eps$.
	Then by \Prop~\ref{prop_greg} (iii) there exists $\lambda'>\lambda$ such that $\amax(\lambda'')=\alpha_*(\lambda'')$ and $\alpha_*(\lambda'')<\alpha_*(\lambda)+\eps^2/2$ for all $\lambda''\in[\lambda,\lambda']$.
	Hence, \Lem s~\ref{lem_nul} yields
	\begin{align}\label{eq_lem_nul_*_1}
		\Phi_{d,k,\lambda'}(\amax(\lambda'))-\Phi_{d,k,\lambda}(\amax(\lambda))&\leq\int_{\lambda}^{\lambda'}(\alpha_*(\lambda'')-1)\,\mathrm d\lambda''\leq(\lambda'-\lambda)(\alpha_*(\lambda)+\eps^2/2-1).
	\end{align}
	Combining \eqref{eq_lem_nul_*_1} with \Prop~\ref{prop_nul} and \Lem~\ref{cor_nul}, we obtain
	\begin{align}\label{eq_lem_nul_*_10}
		n^{-1}\brk{\ex\brk{\nul\vA[\lambda']}-\ex\brk{\nul\vA[\lambda]}}&\leq(\lambda'-\lambda)(\alpha_*(\lambda)+\eps^2/2-1+o(1)).
	\end{align}
	On the other hand, since adding checks can only increase the number of frozen variables, \Lem~\ref{lem_diff_nul} shows that
	\begin{align}\label{eq_lem_nul_*_2}
		n^{-1}\brk{\ex\brk{\nul\vA[\lambda']}-\ex\brk{\nul\vA[\lambda]}}&\geq(\lambda'-\lambda)(\alpha_*(\lambda)+\pr\brk{\cE}\eps-1+o(1))\geq(\lambda'-\lambda)(\alpha_*(\lambda)+\eps^2-1+o(1)).
	\end{align}
	Finally, since \eqref{eq_lem_nul_*_10} and \eqref{eq_lem_nul_*_2} contradict each other, we have refuted the assumption $\pr\brk{\cE}>\eps$.
\end{proof}

\begin{lemma}\label{lem_nul^*}
	Assume that $d,\lambda>0$ are such that
	\begin{align}\label{eq_cor_frz^*}
		\Phi_{d,k,\lambda}(\alpha^*)>\Phi_{d,k,\lambda}(\alpha)\quad\mbox{for all $\alpha\in[0,1]\setminus\{\alpha^*\}$.}
	\end{align}
	Then $|\frz(\vA[\lambda])|=\alpha^*n+o(n)$ \whp
\end{lemma}
\begin{proof}
	We use a similar strategy as in the proof of \Lem~\ref{lem_nul_*}.
	Hence, assume that $d,\lambda>0$ satisfy \eqref{eq_cor_frz^*}.
	Combining \Prop~\ref{prop_WP} and \Lem~\ref{lem_uzn}, we see that $|\frz(\vA[\lambda])|\leq\alpha^*n+o(n)$ \whp\
	Now choose a small enough $\eps>0$ and assume that $\cE=\{|\frz(\vA[\lambda])|<(\alpha^*-\eps)n\}$ occurs with probability $\pr\brk\cE>\eps$.
	Then \Prop~\ref{prop_greg} shows that there exists $\lambda'<\lambda$ such that $\amax(\lambda'')=\alpha^*(\lambda'')$ and $\alpha^*(\lambda'')>\alpha^*(\lambda)-\eps^2/2$ for all $\lambda''\in[\lambda,\lambda']$.
	Hence, \Lem s~\ref{lem_nul} yields
	\begin{align}\label{eq_lem_nul^*_1}
		\Phi_{d,k,\lambda}(\amax(\lambda))-\Phi_{d,k,\lambda'}(\amax(\lambda'))&=\int_{\lambda}^{\lambda'}(\alpha^*(\lambda'')-1)\,\mathrm d\lambda''\geq(\lambda'-\lambda)(\alpha^*(\lambda)-\eps^2/2-1).
	\end{align}
	But once again because adding checks can only increase the number of frozen variables, \Lem~\ref{lem_diff_nul} yields
	\begin{align}\label{eq_lem_nul^*_2}
		n^{-1}\brk{\ex\brk{\nul\vA\brk{\lambda}}-\ex\brk{\nul\vA\brk{\lambda'}}}&\leq(\lambda'-\lambda)(\alpha^*(\lambda)-\pr\brk{\cE}\eps-1+o(1))\leq(\lambda'-\lambda)(\alpha_*(\lambda)-\eps^2-1+o(1)).
	\end{align}
	However,  \Prop~\ref{prop_nul} and \Lem~\ref{lem_nul} show that \eqref{eq_lem_nul^*_1}--\eqref{eq_lem_nul^*_2} are in contradiction.
\end{proof}

\begin{proof}[Proof of \Prop~\ref{cor_frz}]
	Since $\amax\in\{\alpha_*,\alpha^*\}$, the assertion is an immediate consequence of \Lem s~\ref{lem_nul_*}--\ref{lem_nul^*}.
\end{proof}


\subsection{Proof of \Cor~\ref{cor_nll}}\label{sec_cor_nll}

There are four cases to consider separately.
Let $\eps>0$.
\begin{description}
	\item[Case 1: $d<\dmin$]
		As \Prop~\ref{prop_greg} (i) shows, in this case we have $\alpha_*=\alpha^*$ for all $\lambda>0$; thus, the function $\ph$ has only the single fixed point $\alpha_*$, which is stable.
		Furthermore, \Prop~\ref{prop_WP} shows that $||V_{\nll,\ell}(\FDC t)|-\alpha_*n|<\eps n/2$ for large enough $\ell$ \whp\
		Moreover, \Prop~\ref{cor_frz} yields $|\frz(\FDC t)|=\alpha_*n+o(n)$ \whp\
		Therefore, \Prop~\ref{prop_nlluzn} implies that $|\frz(\FDC t)\triangle V_{\nll,\ell}(\FDC t)|<\eps n$ \whp\ for large enough $\ell$.
		Since $|V_{\nll,\ell}(\FDC t)|\subset\frz(\FDC t)$ \whp\ and $|V_{\nll,\ell}(\FDC t)\triangle V_{\nll}(\FDC t)|<\eps n$ by~\eqref{eqprop_WP}, the assertion follows.
	\item[Case 2] $\dmin<d<\dsat$ and $\theta>\theta_*$: 
		A similar argument as under Case~1 applies.
		Indeed, \Prop~\ref{prop_greg} (ii) shows that $\alpha_*=\alpha^*$ is the unique and stable fixed point of $\ph$.
		Since $||V_{\nll,\ell}(\FDC t)|-\alpha_*n|<\eps n/2$ for large $\ell$ \whp\ by \Prop~\ref{prop_WP} and $|\frz(\FDC t)|=\alpha_*n+o(n)$ \whp\ by \Prop~\ref{cor_frz}, \Prop~\ref{prop_nlluzn} yields $|\frz(\FDC t)\triangle V_{\nll,\ell}(\FDC t)|<\eps n$ \whp\
		Therefore, \eqref{eqprop_WP} implies the assertion.
	\item[Case 3] $\dmin<d<\dsat$ and $\theta<\tcond$:
		\Prop~\ref{prop_greg} (ii) shows that $\alpha_*<\alpha^*$ in this case.
		Moreover, \Prop~\ref{prop_WP} yields $||V_{\nll,\ell}(\FDC t)|-\alpha_*n|<\eps n/2$ for large $\ell$ \whp, while \Prop~\ref{cor_frz} and \Prop~\ref{prop_greg} (iii) imply that $|\frz(\FDC t)|=\alpha_*n+o(n)$ \whp\
		Thus, the same steps as in Cases 1--2 complete the proof.
	\item[Case 4] $\dmin<d<\dsat$ and $\tcond<\theta<\theta_*$:
		Once again \Prop~\ref{prop_greg} (ii) shows that $\alpha_*<\alpha^*$, \Prop~\ref{prop_WP} yields $||V_{\nll,\ell}(\FDC t)|-\alpha_*n|<\eps n/2$ for large $\ell$ \whp, and \Prop~\ref{cor_frz} and \Prop~\ref{prop_greg} (iii) show that $|\frz(\FDC t)|=\alpha^*n+o(n)$ \whp\
		Since $|V_{\nll,\ell}(\FDC t)|\subset\frz(\FDC t)$ \whp, the assertion follows from \eqref{eqprop_WP} and the fact that $\alpha_*<\alpha^*$.
\end{description}

\subsection{Proof of \Cor~\ref{lem_fzn}}\label{sec_lem_fzn}

Assume first that $\theta<\tcond$.
Then \Cor~\ref{cor_nll} shows that $|\frz(\FDC t)\triangle V_{\nll}(\FDC t)|=o(n)$ for large enough $\ell$.
Since $V_{\nll}(\FDC t)\cap V_{\fzn}(\FDC t)=\emptyset$ by construction, the first assertion follows. 

Now suppose $\theta>\tcond$.
Then \Prop~\ref{prop_WP} yields $||V_{\fzn,\ell}(\FDC t)|-\alpha^*n|<\eps n/2$ for large $\ell$ \whp, while \Prop~\ref{cor_frz} and \Prop~\ref{prop_greg} (iii) show that $|\frz(\FDC t)|=\alpha^*n+o(n)$ \whp\
Additionally, \Prop~\ref{prop_WP} shows that $|V_{\uzn,\ell}(\FDC t)\cap \frz(\FDC t)|<\eps n$ for large $\ell$, which implies the assertion.
	
\subsection{Proofs of \Thm s~\ref{thm_recnonrec} and~\ref{thm_cond}}\label{sec_thm_recnonrec}
We begin with the following observation.

\begin{lemma}\label{prop_recon}
	Let $\SIGMA\in\ker(\FDC t)$ be uniformly random.
	For any $\ell>0$ \whp\ we have
	\begin{align}\label{eq_prop_recon_mu}
		\pr\brk{\SIGMA_{x_{t+1}}=0\mid\FDC t,\SIGMA_{\partial^{2\ell} x_{t+1}}}&=\frac12\bc{1+\vecone\{x_{t+1}\in V_{\fzn,\ell}(\FDC t)\cup V_{\nll,\ell}(\FDC t)\}},\\
		\pi_{\FDC{t}}=\pr\brk{\SIGMA_{x_{t+1}}=0\mid\FDC t}&=\frac12\bc{1+\vecone\{x_{t+1}\in \frz(\FDC t)\}}.\label{eq_prop_recon_pi}
	\end{align}
\end{lemma}
\begin{proof}
	Notice that for $d<\dsat$ the random XORSAT instance $\PHI$ is satisfiable \whp; therefore, so is $\FDC t$.

	We begin with the proof of \eqref{eq_prop_recon_pi}.
	The first equality $\pi_{\FDC{t}}=\pr\brk{\SIGMA_{x_{t+1}}=0\mid\FDC t}$ follows from the fact that the set of solutions of $\FDC t$ is an affine translation of $\ker(A(\FDC t))$.
	Moreover, the second equality sign follows from the well known fact that the marginal $\pr\brk{\SIGMA_{x_{t+1}}=0\mid\FDC t}$ is equal to $1/2$ or to $1$.

	Moving on to~\eqref{eq_prop_recon_mu}, we recall from \Lem~\ref{lem_lwc} that the depth-$2\ell$ neighbourhood $\nb x_{t+1}$ of $x_{t+1}$ in $\FDC t$ is acyclic \whp\
	Furthermore, we can think of $\pr\brk{\SIGMA_{x_{t+1}}=0\mid\FDC t,\SIGMA_{\nb x_{t+1}}}$ as the marginal probability that $x_{t+1}$ receives the value zero under a random vector from the kernel of the check matrix of $\nb x_{t+1}$, subject to imposing the values $\SIGMA_{\nb x_{t+1}}$ upon the variable at distance exactly $2\ell$ from $x_{t+1}$.
	Let $\FDC t^{(\ell)}$ signify the XORSAT instance thus obtained.
	Then we conclude that $\pr\brk{\SIGMA_{x_{t+1}}=0\mid\FDC t,\SIGMA_{\nb x_{t+1}}}=1$ iff $x_{t+1}\in V_0(\FDC t^{(\ell)})$.
	Furthermore, because BP is exact on acyclic factor graphs, we have $x_{t+1}\in V_0(\FDC t^{(\ell)})$ iff $x_{t+1}\in V_{\frz,\ell}(\FDC t)\cup V_{\nll,\ell}(\FDC t)$.
	Thus, we obtain \eqref{eq_prop_recon_mu}.
\end{proof}

\begin{proof}[Proof of \Thm~\ref{thm_recnonrec}]
	We begin with claim (i) concerning $d<\dmin$.
	As \Prop~\ref{prop_greg} (i) shows, in this case we have $\alpha_*=\alpha^*$.
	Furthermore, \Prop~\ref{prop_WP} shows that $||V_{\nll,\ell}(\FDC t)|-\alpha_*n|<\eps n$ and $|V_{\frz,\ell}(\FDC t)|<\eps n$ for large enough $\ell$ \whp\
	Moreover, \Prop~\ref{cor_frz} yields $|\frz(\FDC t)|=\alpha_*n+o(n)$ \whp\
	Therefore, \Prop~\ref{prop_nlluzn} implies that $|\frz(\FDC t)\triangle V_{\nll,\ell}|<\eps n$ \whp\ for large enough $\ell$.
	Hence, \Lem~\ref{prop_recon} shows that the non-reconstruction property~\eqref{eqnonrec} holds \whp\

	Similarly, towards the proof of (ii) assume that $\dmin<d<\dsat$ and $\theta<\theta^*$.
	Then \Prop~\ref{prop_greg}~(ii) shows that $\alpha_*=\alpha^*$ is the unique (stable) fixed point of $\ph$.
	Therefore, the argument from the previous paragraph shows that~\eqref{eqnonrec} holds \whp
	Further, suppose that $\dmin<d<\dsat$ and $\theta>\tcond$.
	Then \Cor~\ref{lem_fzn}~(ii) shows that $|(V_{\nll,\ell}(\FDC{t})\cup V_{\fzn,\ell}(\FDC{t}))\triangle V_{0,\ell}(\FDC{t})|<\eps n$ \whp\
	Therefore, \Lem~\ref{prop_recon} implies non-reconstruction property, and thus the proof of (ii) is complete.

	Finally, suppose that $\dmin<d<\dsat$ and $\theta^*<\theta<\tcond$.
	Then \Prop~\ref{prop_WP} shows that $||V_{\nll,\ell}(\FDC t)-\alpha_*n|<\eps n$ and $|V_{\frz,\nll}|-(\alpha^*-\alpha_*)n|<\eps n$ for large enough $\ell$ \whp\
	Moreover, \Cor~\ref{lem_fzn} shows that $|V_{\fzn,\nll}\cap\frz(\FDC t)|<\eps n$ \whp\
	Consequently, \Lem~\ref{prop_recon} demonstrates that the reconstruction condition~\eqref{eqrec} holds \whp
\end{proof}

\begin{proof}[Proof of \Thm~\ref{thm_cond}]
Part (i) regarding the case $d<\dmin$ is an immediate consequence of Fact~\ref{fact_wp} (the equivalence of WP and BP), \Cor~\ref{cor_nll}~(i) and \Lem~\ref{prop_recon}.
The same is true of part~(ii) concerning $\dmin<d<\dsat$ and $\theta<\tcond$ or $\theta>\theta_*$.
Furthermore, (iii) follows from \Cor~\ref{cor_nll}~(ii) and \Lem~\ref{prop_recon}.
\end{proof}

\section{Belief Propagation Guided Decimation}\label{sec_alg}

\noindent
In this section we prove \Thm~\ref{thm_bpgd}.
We begin by arguing that \BPGD\ is actually equivalent to the simple combinatorial Unit Clause Propagation algorithm.
Then we prove the `positive' part, i.e., the formula \eqref{eqthm_bpgd} for the success probability for $d<\dmin$.
Subsequently we prove the second part of the theorem concerning $\dmin<d<\dsat$.

\subsection{Unit Clause Propagation redux}\label{sec_UCP}

The simple-minded Unit Clause Propagation algorithm attempts to assign random values to as yet unassigned variables one after the other.
After each such random assignment the algorithm pursues the `obvious' implications of its decisions.
Specifically, the algorithm substitutes its chosen truth values for all occurrences of the already assigned variables.
If this leaves a clause with only a single unassigned variable, a so-called `unit clause', the algorithm assigns that variable so as to satisfy the unit clause.
If a conflict occurs because two unit clauses impose opposing values on a variable, the algorithm declares that a conflict has occurred, sets the variable to false and continues; of course, in the event of a conflict the algorithm will ultimately fail to produce a satisfying assignment.
The pseudocode for the algorithm is displayed in Algorithm~\ref{Alg_UCP}.

\begin{algorithm}[h!]
	Let $U=\emptyset$ and let $\SIGUC:U\to\{0,1\}$ be the empty assignment\;
 \For{$t=0,\ldots,n-1$}{
	 \If{$x_{t+1}\not\in U$}{%
		 add $x_{t+1}$ to $U$\;
		 choose $\SIGUC(x_{t+1})\in\{0,1\}$ uniformly at random\;
			\While{$\PHI[\SIGUC]$ contains a unit clause $a$}{%
				let $x$ be the variable in $a$\;
				let $s\in\{0,1\}$ be the truth value that $x$ needs to take to satisfy $a$\;
				\If{another unit clause $a'$ exists that requires $x$ be set to $1-s$}{output `conflict' and let $\SIGUC(x)=0$\;}\Else{add $x$ to $U$ and let $\SIGUC(x)=s$\;}
				}
			}
		}
	\Return $\SIGUC$\;
 \caption{The \UCP\ algorithm.}\label{alg_ucp}
 \label{Alg_UCP}
\end{algorithm}

Let $\FUC t$ denote the simplified formula obtained after the first $t$ iterations (in which the truth values chosen for $x_1,\ldots,x_t$ and any values implied by Unit Clauses have been substituted).
We notice that the values assigned during Steps~6--12 are deterministic consequences of the choices in Step~5.
In particular, the order in which unit clauses are processed Steps~6--12 does not affect the output of the algorithm.

\begin{proposition}\label{prop_UCP}
	We have
	\begin{align*}
		\pr\brk{\BPGD\mbox{ outputs a satisfying assignment of }\PHI}&=\pr\brk{\UCP\mbox{ outputs a satisfying assignment of }\PHI}.
	\end{align*}
\end{proposition}
\begin{proof}
	We  employ the following coupling.
	Let $\TAU\in\{0,1\}^n$ be a uniformly random vector.
	The \BPGD\ algorithm sets $\SIGBP(x_{t+1})=\TAU_{t+1}$ if $\mu_{\FBP{t}}=1/2$.
	Analogously, \UCP\ sets $\SIGUC(x_{t+1})=\TAU_{t+1}$ in Step~5 (if $x_{t+1}\not\in U$).
	Hence, because \eqref{eqHalfInt} guarantees that the BP marginals $\mu_{\FBP{t}}$ are half-integral, the coupling ensures that the ``free steps'' of the two algorithms pick the same truth values.

	We now proceed by induction on $0\leq t\leq n$ to prove the following two statements.
	\begin{description}
		\item[UCP1] unless \UCP\ encountered a conflict before time $t$ we have $\SIGBP(x_i)=\SIGUC(x_i)$ for $i=1,\ldots,t$.
		\item[UCP2] if $t<n$ and there has been no conflict before time $t$ we have we have $\mu_{\FBP{t+1}}=1/2$ iff $x_{t+1}\not\in U$.
	\end{description}

	For $t=0$ both of these statements are clearly correct because $\mu_{\FBP{0}}=1/2$ and $x_{1}\not\in U$.

	Now assume that {\bf UCP1}--{\bf UCP2} hold at time $t-1$ and that no conflict has occurred yet.
	Then we already know that $\SIGBP(x_i)=\SIGUC(x_i)$ for $i=1,\ldots,t-1$.
	Furthermore, since {\bf UCP2} is correct at time $t-1$ we have $\mu_{\FBP{t}}=1/2$ iff $x_{t}\not\in U$.
	Consequently, if $x_t\not\in U$ then $\SIGBP(x_t)=\SIGUC(x_t)$.
	Hence, suppose that $x_t\in U$ and thus $\mu_{\FBP{t}}\in\{0,1\}$.
	Then given $\SIGBP(x_1)=\SIGUC(x_1),\ldots,\SIGBP(x_{t-1})=\SIGUC(x_{t-1})$ the value $\SIGUC(x_t)$ is implied by unit clause propagation.
	But a glimpse at the BP update rules~\eqref{eqBPupdate1}--\eqref{eqBPupdate2} shows that these encompass the unit clause rule.
	Specifically, if $x$ is the only remaining variable in clause $a$, then \eqref{eqBPupdate1} ensures that the message from $a$ to $x$ gives probability one to the value that satisfies clause $a$.
	Therefore, the definition~\eqref{eqBPmarg} of the BP marginal demonstrates that $\mu_{\FBP{t}}=\SIGUC(x_1)$ and thus $\SIGBP(x_t)=\SIGUC(x_t)$.
	Thus, {\bf UCP1} continues to hold for $t$.

	Similar reasoning yields {\bf UCP2}.
	Indeed, revisiting \eqref{eqBPupdate1}, we see that the BP message that clause $a$ sends to variable $x$ equals $1/2$ unless $a$ is a unit clause.
	In effect, \eqref{eqBPmarg} shows that the BP marginal $\mu_{\FBP{t+1}}$ is equal to $1/2$ unless the value of $x_{t+1}$ is implied by the unit clause rule.
	This completes the induction.

	To complete the proof assume that \UCP\ manages to find a satisfying assignment.
	Then {\bf UCP1} applied to $t=n$ demonstrates that \BPGD\ outputs the very same satisfying assignment.
	Conversely, if \UCP\ encounters a conflict at some time $t$, then {\bf UCP1} shows that \BPGD\ chose the same assignment up to time $t$.
	Therefore, it is not possible to extend the partial assignment $\SIGBP(x_1),\ldots,\SIGBP(x_t)$ to a satisfying assignment of $\PHI$ and thus \BPGD\ will ultimately fail to output a satisfying assignment.
\end{proof}

In light of \Prop~\ref{prop_UCP} we are left to study the success probability of \UCP.
The following two subsections deal with this task for $d<\dmin$ and $d>\dmin$, respectively.

\subsection{The success probability of \UCP\ for $d<\dmin$}\label{sec_alg_pos}
We continue to denote by $\FUC{t}$ the sub-formula obtained after the first $t$ iterations of \UCP.
Let $V_n=\{x_1,\ldots,x_n\}$ be the set of variables of the \emph{XORSAT} instance \emph{F}. 
Also, let $\vV(t)\subset\{x_{t+1},\ldots,x_n\}$ be the set of variables of $\FUC{t}$.
Thus, $\vV(t)$ contains those variables among $x_{t+1},\ldots,x_n$ whose values are not implied by the assignment of $x_1,\ldots,x_t$ via unit clauses.
Also let $\vC(t)$ be the set of clauses of $\FUC{t}$; these clauses contain variables from $\vV(t)$ only, and each clause contains at least two variables.
Let $\bar\vV(t)=V_n\setminus\vV(t)$ be the set of assigned variables.
Thus, after its first $t$ iterations \UCP\ has constructed an assignment $\SIGUC:\bar\vV(t)\to\{0,1\}$.
Moreover, let $\vV'(t+1)=\vV(t)\setminus\vV(t+1)$ be the set of variables that receive values in the course of the iteration $t+1$ for $0\leq t<n$.
Additionally, let $\vC'(t+1)$ be the set of clauses of $\FUC t$ that consists of variables from $\vV'(t+1)$ only.
Finally, let $\FUC{t+1}'$ be the formula comprising the variables $\vV'(t+1)$ and the clauses $\vC'(t+1)$.

To characterise the distribution of $\FUC t$ let $\vn(t)=|\vV(t)|$ and let $\vm_\ell(t)$ be the number of clauses of length $\ell$, i.e., clauses that contain precisely $\ell$ variables from $\vV(t)$.
Observe that $\vm_1(t)=0$ because unit clauses get eliminated.
Let $\fF_t$ be the $\sigma$-algebra generated by $\vn(t)$ and $(\vm_\ell(t))_{2\leq\ell\leq k}$.

\begin{fact}\label{fact_deferred}
	The XORSAT formula $\FUC t$ is uniformly random given $\fF_t$.
	In other words, the variables that appear in each clause are uniformly random and independent, as are their signs.
\end{fact}
\begin{proof}
	This follows from the principle of deferred decisions.
\end{proof}

We proceed to estimate the random variables $\vn(t),\vm_\ell(t)$.
Let $\vec\alpha(t)=|\bar\vV(t)|/n$ so that $\vn(t)=n(1-\vec\alpha(t))$.
Recall, that $\bar\vV(t) =V_n\setminus\vV(t)$.
Let $\lambda=\lambda(\theta)=-\log(1-\theta)$ \lk{with $\theta \sim t/n$} and recall that $\alpha_*=\alpha_*(d,k,\lambda)$ denotes the smallest fixed point of $\phi_{d,k,\lambda}$.
The proof of the following proposition proof can be found in \Sec~\ref{sec_prop_uc_alpha}.

\begin{proposition}\label{prop_uc_alpha}
	Suppose that $d<\dmin(k)$.
	There exists a function $\delta=\delta(n)=o(1)$ such that for all $0\leq t<n$ and all $2\leq\ell\leq k$ we have
	\begin{align}\label{eq_prop_uc_alpha}
		\pr\brk{|\vec\alpha(t)-\alpha_*|>\delta}&=O(n^{-2}),&
		\pr\brk{\abs{\vm_\ell(t)-\frac{dn}k\binom k\ell(1-\alpha_*)^\ell\alpha_*^{k-\ell}}>\delta n}&=O(n^{-2}).
	\end{align}
\end{proposition}

\Prop~\ref{prop_uc_alpha} paves the way for the actual computation of the success probability of \UCP.
Let $\cR_t$ be the event that a conflict occurs in iteration $t$.
The following proposition gives us the correct value of $\pr\brk{\cR_t \mid \fF_t} $ \whp\, 
Since $\fF_t$ is a random variable the value for the probability $\pr\brk{\cR_t \mid \fF_t} $ is random as well.

\begin{proposition}\label{prop_uc_error}
	Fix $\eps>0$, let $0\leq t<(1-\eps)n$ and define
	\begin{align}\label{eq_prop_uc_error}
		f_n(t)= d(k-1)(1-\alpha_*)  \alpha_*^{k-2}.
	\end{align}
	Then with probability $1-o(1/n)$ we have 
	\begin{align*}
		\pr\brk{\cR_t \mid \fF_t} = \frac{f_n(t)^2}{4(n-t)(1-f_n(t))^2}+o(1/n).
	\end{align*}
\end{proposition}

The proof of \Prop~\ref{prop_uc_error} can be found in \Sec~\ref{sec_prop_uc_error}.
Moreover, in \Sec~\ref{sec_prop_uc_pois} we prove the following.

\begin{proposition}\label{prop_uc_pois}
	Fix $\eps>0$ and $\ell\geq1$.
	For any $0\leq t_1<\cdots<t_\ell<(1-\eps)n$ we have
	\begin{align}\label{eq_prop_uc_pois}
		\pr\brk{\bigcap_{i=1}^\ell\cR_{t_i}}&\sim\prod_{i=1}^\ell \frac{f_n(t_i)^2}{4(n-t_i)(1-f_n(t_i))^2}.
	\end{align}
\end{proposition}

Finally, the following statement deals with the $\eps n$ final steps of the algorithm.

\begin{proposition}\label{prop_uc_endgame}
	For any $\delta>0$ there exists $\eps>0$ such that $\pr\brk{\bigcup_{(1-\eps)n<t<n}\cR_{t}}<\delta.$ 
\end{proposition}

Before we proceed we notice that \Prop s~\ref{prop_uc_error}--\ref{prop_uc_endgame} imply the first part of \Thm~\ref{thm_bpgd}.

\begin{proof}[Proof of \Thm~\ref{thm_bpgd} (i)]
	Pick $\delta>0$, fix a small enough $\eps=\eps(\delta)>0$ and let $\vR=\sum_{t=0}^{n-1}\vecone\{\cR_t\}$ be the total number of times at which conflicts occur.
	\Prop~\ref{prop_UCP} shows that the probability that \BPGD\ succeeds equals $\pr\brk{\vR=0}$.
	In order to calculate $\pr\brk{\vR=0}$, let $\vR_\eps=\sum_{0\leq t\leq(1-\eps)n}\vecone\{\cR_t\}$ be the number of failures before time $(1-\eps)n$.
	\Prop~\ref{prop_uc_pois} shows that for any fixed $\ell\geq1$ we have
	\begin{align}\nonumber
		\ex\brk{\prod_{i=1}^\ell(\vR_\eps-i+1)}&=\ell!\sum_{0\leq t_1<\cdots<t_\ell\leq (1-\eps)n}\pr\brk{\bigcap_{i=1}^\ell\cR_{t_i}}\sim\ell!\sum_{0\leq t_1<\cdots<t_\ell\leq (1-\eps)n} \hspace{1mm}
		\prod_{i=1}^\ell \frac{f_n(t_i)^2}{4(n-t_i)(1-f_n(t_i))^2}\\
										  &=(1+o(1))\sum_{0\leq t_1,\ldots,t_\ell\leq (1-\eps)n}
										  \hspace{1mm}\prod_{i=1}^\ell \frac{f_n(t_i)^2}{4(n-t_i)(1-f_n(t_i))^2}\sim\ex[\vR_\eps]^\ell.\label{eq_thm_bpgd_i_1}
	\end{align}
	Hence, the inclusion/exclusion principle (e.g., \cite[\Thm~1.21]{BB}) implies that 
	\begin{align}\label{eq_thm_bpgd_i_2}
		\pr\brk{\vR_\eps=0}&\sim\exp(-\ex[\vR_\eps]).
	\end{align}
	Further, using \Prop~\ref{prop_uc_error} and the linearity of expectation, we obtain with $\lambda(\theta) = -\log(1-\theta)$
\begin{align}\nonumber
	\ex[\vR_\eps]&=\sum_{0\leq t\leq(1-\eps)n }\pr\brk{\cR_t }\sim\sum_{0\leq t\leq(1-\eps)n}\frac{f_n(t)^2}{4(n-t)(1-f_n(t))^2}\sim\frac{1}{4n} \int_0^{1-\eps} \frac{f_n(\theta n)^2}{ (1-\theta)(1-f_n(\theta n))^2}  \mathrm d \theta \\
				 &= \frac{1}{4n} \int_0^{1-\eps} \frac{f_n(\theta n)^2}{  (1-\alpha_*)(1-f_n(\theta n)) } \frac{\partial \alpha_*}{\partial \lambda}\frac{\partial\lambda(\theta)}{\partial \theta} \mathrm d \theta   &&\mbox{[by \eqref{eq_alphas}]}\nonumber\\
				 &= \frac{d^2(k-1)^2}4\int_0^{1-\eps}\frac{z^{2k-4}(1-z)}{1-d(k-1)z^{k-2}(1-z)}\,\mathrm d z&&\mbox{[by \eqref{eq_prop_uc_error}]}.\label{eq_thm_bpgd_i_3}
\end{align}
Finally, \Prop~\ref{prop_uc_endgame} implies that
\begin{align}\label{eq_thm_bpgd_i_4}
	\pr\brk{\vR>\vR_\eps}<\delta.
\end{align}
Thus, the assertion follows from \eqref{eq_thm_bpgd_i_2}--\eqref{eq_thm_bpgd_i_4} upon taking the limit $\delta\to0$.
\end{proof}

\subsubsection{Proof of \Prop~\ref{prop_uc_alpha}}\label{sec_prop_uc_alpha}

The proof of \Prop~\ref{prop_uc_alpha} is based on the method of differential equations.
Specifically, based on Fact~\ref{fact_deferred} we derive a system of ODEs that track the random variables $\vec\alpha(t),\vm_2(t),\ldots,\vm_k(t)$.
We will then identify the unique solution to this system.
As a first step we work out the conditional expectations of $\vec\alpha(t+1),\vm_2(t+1),\ldots,\vm_k(t+1)$ given $\fF_t$.

\begin{lemma}\label{lem_condex}
	If $2\vm_2(t)/\vn(t)<1-\Omega(1)$ and $\vn(t)=\Omega(n)$, then 
	\begin{align}\label{eq_lem_condex_1}
		\ex\brk{\vn(t)-\vn(t+1)\mid\fF_t}&=\frac{\vn(t)^2}{(n-t)(\vn(t)-2\vm_2(t))}+o(1),\\
		\ex\brk{\vec m_\ell(t+1)-\vec m_\ell(t)\mid\fF_t}&=\frac{\vn(t)^2}{(n-t)(\vn(t)-2\vm_2(t))}\cdot\frac{(\ell+1)\vm_{\ell+1}(t)-\ell\vm_\ell(t)}{\vn(t)}+o(1)&&(2\leq\ell< k),\label{eq_lem_condex_2}\\
		\ex\brk{\vec m_k(t+1)\mid\fF_t}&=-\frac{\vn(t)^2}{(n-t)(\vn(t)-2\vm_2(t))}\cdot\frac{k\vm_k(t)}{\vn(t)}+o(1).\label{eq_lem_condex_3}
	\end{align}
\end{lemma}
\begin{proof}
	Going from time $t$ to time $t+1$ involves the express assignment of variable $x_{t+1}$, unless it had already been assigned a value due to previous decisions, and the subsequent pursuit of unit clause implications.
	The probability given $\fF_t$ that $x_{t+1}$ was set in a previous iteration equals
	\begin{align}\label{eqqt+1}
		q_{t+1}&=1-\frac{\vn(t)}{n-t}.
	\end{align}
	Indeed, the first $t$ iterations assigned values to a total of $n-\vn(t)$ variables,   including $x_1,\ldots,x_{t}$, and Fact~\ref{fact_deferred} shows that the identities of the assigned variables among $x_{t+1},\ldots,x_{n}$ are random.

	Let $\cQ_{t+1}$ be the event that $x_{t+1}$ was not assigned previously.
	Given $\cQ_{t+1}$ we need to pursue unit clause implications.
	To this end, recall the bipartite graph representation $G(\FUC t)$ of the formula $\FUC t$.
	Let $G_2(\FUC t)$ be the subgraph of $G(\FUC t)$ obtained by removing all clauses of length greater than two.
	Then Fact~\ref{fact_deferred} shows that $G_2(\FUC t)$ is a uniformly random bipartite graph with $\vn(t)$ nodes on one side and $\vm_2(t)$ nodes of degree two on the other side.
	Furthermore, the number of variables whose values are implied by unit clause propagation is lower bounded by the number of variable nodes in the component of $x_{t+1}$ in $G_2(\FUC t)$.
	The expected size of this component can be computed as the expected progeny of a branching process with offspring $\Po(2\vm_2(\ell)/\vn(t))$.
	As is well known, under the assumption $2\vm_2(t)/\vn(t)<1-\Omega(1)$ that the branching process is sub-critical, the expected progeny comes to $(1-2\vm_2(t)/\vn(t))^{-1}$.
	Hence, we obtain
	\begin{align}\label{eq_lem_condex_10}
		\ex\brk{n(\vec\alpha(t+1)-\vec\alpha(t))\mid\fF_t}&\geq\frac{1-q_{t+1}}{1-2\vm_2(t)/\vn(t)}.
	\end{align}

	Strictly speaking, \eqref{eq_lem_condex_10} only gives a lower bound on $\ex\brk{n(\vec\alpha(t+1)-\vec\alpha(t))\mid\fF_t}$ because additional unit clause implications could arise from clauses of length greater than two.
	However, for this to happen a clause would have to contain at least two variables that are set in iteration $t+1$ (i.e., either $x_{t+1}$ itself or a variable whose value is implied due to unit clause propagation).
	But since $2\vm_2(t)/\vn(t)<1-\Omega(1)$, the expected number of such implications is bounded, and thus the expected number of longer clauses that turn into unit clauses is of order $O(1/n)$.
	Consequently, the lower bound~\eqref{eq_lem_condex_10} is tight up to an $O(1/n)$ error term, whence we obtain~\eqref{eq_lem_condex_1}.

	Moving on to \eqref{eq_lem_condex_2}--\eqref{eq_lem_condex_3} we notice that for $2\leq\ell<k$ there are two ways in which the number of clauses of length $\ell$ can change from iteration $t$ to iteration $t+1$.
	First, it could be that clauses of length $\ell$ contain one variable that gets a value assigned.
	Any such clauses shorten to length $\ell-1$ (if $\ell>2$) or become unit clauses and subsequently disappear ($\ell=2$).
	In light of Fact~\ref{fact_deferred}, the probability that a given clause of length $\ell$ suffers this fate comes to $\ell(\vn(t)-\vn(t+1))/\vn(t)+o(1)$.
	Conversely, if $\ell<k$ additional clauses of length $\ell$ may result from the shortening of clauses of length $\ell+1$.
	Analogously to the previous computation, the probability that a given clause of length $\ell+1$ shortens to length $\ell$ comes to $(\ell+1)(\vn(t)-\vn(t+1))/\vn(t)+o(1)$.
	Of course, there could also be clauses that contain more than one variable that receives a value during iteration $t+1$.
	However, the probability of this event is of order $O(1/n^2)$.
	Hence, \eqref{eq_lem_condex_1} implies \eqref{eq_lem_condex_2} and~\eqref{eq_lem_condex_3}.
\end{proof}

\Lem~\ref{lem_condex} puts us in a position to derive a system of ODEs to track the random variables $\vn(t),\vm_2(t),\ldots,\vm_k(t)$.
Specifically, we obtain the following.

\begin{corollary}\label{cor_ODE}
	Let $\fn,\fm_2,\ldots,\fm_k:[0,1]\to\RR$ be continuously differentiable functions such that
	\begin{align}\label{eqODE1}
		\fn(0)&=1,&\fm_k(0)&=\frac dk,\\
		\frac{\partial\fn}{\partial\theta}&=-\frac{\fn^2}{(1-\theta)(\fn-2\fm_2)},\label{eqODE2}\\
		\frac{\partial\fm_\ell}{\partial\theta}&=\frac{\fn((\ell+1)\fm_{\ell+1}-\ell\fm_\ell)}{(1-\theta)(\fn-2\fm_2)}\quad(2\leq\ell<k),&
		\frac{\partial\fm_k}{\partial\theta}&=-\frac{k\fn\fm_k}{(1-\theta)(\fn-2\fm_2)}.\label{eqODE3}
	\end{align}
	Assume, furthermore, that
	\begin{align}\label{eqODE4}
		\sup_{\theta\in[0,1]}2\fm_2(\theta)/\fn(\theta)&<1.
	\end{align}
	Then with probability $1-o(n^{-2})$ for all $0\leq t\leq n$ we have
	\begin{align*}
		\vn(t)/n&=\fn(t/n)+o(1),& \vm_\ell(t)/n&=\fm_\ell(t/n)+o(1)\quad(2\leq\ell\leq k).
	\end{align*}
\end{corollary}
\begin{proof}
	This follows from \Lem~\ref{lem_condex} in combination with~\cite[\Thm~2]{Wormald}.
\end{proof}

As a next step we construct an explicit solution to the system~\eqref{eqODE1}--\eqref{eqODE3}.

\begin{lemma}\label{lem_ODE}
	If $d<\dmin$, then the functions
	\begin{align}\label{eq_lem_ODE}
		\fn^*(\theta)&=1-\alpha_*(\lambda(\theta)),&
		\fm^*_\ell(\theta)&=\frac dk\binom{k}\ell(1-\alpha_*(\lambda(\theta)))^\ell\alpha_*(\lambda(\theta))^{k-\ell}
	\end{align}
	satisfy \eqref{eqODE1}--\eqref{eqODE4}.
\end{lemma}
\begin{proof}
	The initial condition \eqref{eqODE1} is satisfied because $\alpha_*(\lambda(0))=0$.
	Furthermore, \eqref{eq_alphas} shows that
	\begin{align}\label{eq_lem_ODE_1}
		\frac{\partial\fn^*}{\partial\theta}&=
		-\frac{\partial\alpha_*}{\partial\lambda}\cdot\frac{\partial\lambda}{\partial\theta}=-\frac{1- \alpha_*}{1-d(k-1)\alpha_*^{k-2}(1-\alpha_*)}\cdot\frac1{1-\theta}=-\frac{\fn^*}{(1-\theta)(1-2\fm_2^*/\fn^*)}.
	\end{align}
	Hence, \eqref{eqODE2} is satisfied.
	Furthermore, \eqref{eq_lem_ODE_1} implies that for $2\leq\ell<k$ we have
	\begin{align*}
		\frac{\partial\fm_\ell^*}{\partial\theta}&=\frac dk\cdot\frac{\partial\lambda}{\partial\theta}\cdot\frac{\partial\alpha_*}{\partial\lambda}\cdot\binom{k}{\ell}\brk{(k-\ell)\alpha_*^{k-\ell-1}(1-\alpha_*)^{\ell}-\ell\alpha_*^{k-\ell}(1-\alpha_*)^{\ell-1}}\\
												 &=\frac{\fn^*}{(1-\theta)(1-2\fm_2^*/\fn^*)}\cdot\frac d{k(1-\alpha_*)}\cdot\binom k\ell\brk{(\ell+1)(1-\alpha_*)^{\ell+1}\alpha_*^{k-\ell-1}-\ell \alpha_*^{k-\ell}(1-\alpha_*)^\ell}\\
												 &=\frac{\fn^*}{(1-\theta)(\fn^*-2\fm_2^*)}\cdot\brk{(\ell+1)\fm_{\ell+1}^*-\ell\fm_\ell^*},
	\end{align*}
	which is the first part of~\eqref{eqODE3}.
	An analogous computation yields the second part of~\eqref{eqODE3}.
	Finally, \eqref{eqODE4} follows from~\eqref{eq_alphas}.
\end{proof}

\begin{proof}[Proof of \Prop~\ref{prop_uc_alpha}]
	The proposition	is an immediate consequence of \Cor~\ref{cor_ODE} and \Lem~\ref{lem_ODE}.
\end{proof}

\subsubsection{Proof of \Prop~\ref{prop_uc_error}}\label{sec_prop_uc_error}
$\FUC{t+1}'$ is the XORSAT formula that contains the variables $\vV'(t+1)$ that get assigned during iteration $t+1$ and the clauses $\vC'(t+1)$ of $\FUC{t}$ that contain variables from $\vV'(t+1)$ only.
Also recall that $G(\FUC{t+1}')$ signifies the graph representation of this XORSAT formula.
Unless $\vV'(t+1)=\emptyset$, the graph $G(\FUC{t+1}')$ is connected.

\begin{lemma}\label{lemma_arnab}
	Fix $\eps>0$ and let $0\leq t\leq(1-\eps)n$.
	With probability $1-o(1/n)$ the graph $G(\FUC{t+1}')$ satisfies 
			$$|E(G(\FUC{t+1}'))|\leq|V(G(\FUC{t+1}'))|.$$
\end{lemma}
\begin{proof}
	We recall from the proof of \Lem~\ref{lem_condex} that iteration $t+1$ of \UCP\ can be described by a branching process on the random graph $G(\FUC t)$.
	Given that $x_{t+1}$ is still unassigned, the offspring distribution of the branching process has mean $2\vm_2(t)/\vn(t)$.
	Moreover, \Prop~\ref{prop_uc_alpha} shows that with probability $1-O(n^{-2})$ we have $2\vm_2(t)/\vn(t)\sim d(k-1)(1-\alpha_*)  \alpha_*^{k-2}<1$ (as $d<\dmin$).
	Hence, the branching process is sub-critical.
	As a consequence, with probability $1-O(n^{-2})$ we have
	\begin{align}\label{eq_lemma_arnab_1}
		\pr\brk{|V(G(\FUC{t+1}'))|\geq\log^2n}=O(n^{-2}).
	\end{align}

	Each step of the branching process corresponds to pursuing the unit clause implications of assigning a truth value to a single variable $x$.
	A cycle in $G(\FUC{t+1}')$ can only ensue if a clause that contains $x$ also contains a variable that has already been set previously during iteration $t+1$.
	In light of \eqref{eq_lemma_arnab_1}, with probability $1-O(n^{-2})$ there are no more than $\log^2n$ such variables.
	Hence, the probability that the assignment of $x$ closes a cycle is of order $O(\log^2n/n)$.
	Additionally, by the principle of deferred decisions the events that two different clauses processed by unit clause propagation close cycles is of order $O(\log^4n/n^2)$.
	Finally, since by~\eqref{eq_lemma_arnab_1} we may assume that the total number of clauses does not exceed $O(\log^2n)$, we conclude that $$\pr\brk{|E(G(\FUC{t+1}'))|>|V(G(\FUC{t+1}'))|}=O(\log^6n/n^2)=o(1/n),$$ as desired.
\end{proof}

Thus, with probability $1-o(1/n)$ the graph $G(\FUC{t+1}')$ contains at most one cycle.
While it is easy to check that no conflict occurs in iteration $t+1$ if $G(\FUC{t+1}')$ is acyclic, in the case that $G(\FUC{t+1}')$ contains a single cycle there is a chance of a conflict.
The following definition describes the type of cycle that poses an obstacle.

\begin{definition}
	For a XORSAT formula $F$ we call a sequence of variables and clauses $\toxl=(v_1, c_1, \dots, v_\ell, c_\ell, v_\ell+1=v_1)$ a \emph{toxic cycle} of length $\ell$ if 
	\begin{description}
		\item[TOX1] $c_i$ contains the variables $x_i, x_{i+1}$ only, and
		\item[TOX2] the total number of negations in $c_1, \dots c_\ell$ is odd iff $\ell$ is even.
	\end{description}
\end{definition}

\begin{lemma}\label{lem_tox}
	\begin{enumerate}[(i)]
		\item If $\FUC{t+1}'$ contains a toxic cycle, then a conflict occurs in iteration $t+1$.
		\item If $\FUC{t+1}'$ contains no toxic cycle and $|E(G(\FUC{t+1}'))|\leq|V(G(\FUC{t+1}'))|$, then no conflict occurs in iteration $t+1$.
	\end{enumerate}
\end{lemma}
\begin{proof}
	Towards (i) we show that $\FUC{t+1}'$ is not satisfiable if there is a toxic cycle $\cC=(v_1,c_1,\ldots,c_\ell,v_{\ell+1}=v_1)$; then \UCP\ will, of course, run into a contradiction.
	To see that $\FUC{t+1}'$ is unsatisfiable, we transform each of the clauses $c_1,\ldots,c_\ell$ into a linear equation $c_i\equiv(v_i+v_{i+1}=y_i)$ over $\FF_2$.
	Here $y_i\in\FF_2$ equals $1$ iff $c_i$ contains an even number of negations.
	Adding these equations up yields $\sum_{i=1}^\ell y_i=0$ in $\FF_2$.
	This condition is violated if $\cC$ is toxic. 

	Let us move on to (ii).
	Assume for contradiction that there exists a formula $F$ without a toxic cycle such that $|V(G(F))|\leq|E(G(F))|$ and such that given $\FUC{t+1}'=F$, \UCP\ may run into a conflict.
	Consider such a formula $F$ that minimises $|V(F)|+|C(F)|$.
	Since \UCP\ succeeds on acyclic $F$, we have $|V(G(F))|=|E(G(F))|$. 
	Thus, $G(F)$ contains a single cycle $\cC=(v_1,c_1,\ldots,v_\ell,c_\ell,v_{\ell+1}=v_1)$.
	Apart from the cycle, $F$ contains (possibly empty) acyclic formulas $F_1',\ldots,F_\ell'$ attached to $v_1,\ldots,v_\ell$ and $F_1'',\ldots,F_\ell''$ attached to $c_1,\ldots,c_\ell$.
	The formulas $F_1',F_1'',\ldots,F_\ell',F_\ell''$ are mutually disjoint and do not contain unit clauses. 	
	
	We claim that $F_1',\ldots,F_\ell'$ are empty because $|V(F)|+|C(F)|$ is minimum.
	This is because given any truth assignment of $v_1,\ldots,v_\ell$, \UCP\ will find a satisfying assignment of the acyclic formulas $F_1',\ldots,F_\ell'$.	

	Further, assume that one of the formulas $F_1'',\ldots,F_\ell''$ is non-empty; say, $F_1''$ is non-empty.
	If the start variable that \UCP\ assigns were to belong to $F_1''$, then $c_1$, containing $x_1$ and $x_2$, would not shrink to a unit clause, and thus \UCP\ would not assign values to these variables.
	Hence, \UCP\ starts by assigning a truth value to one of the variables $v_1,\ldots,v_\ell$; say, \UCP\ starts with $v_1$.
	We claim that then \UCP\ does not run into a conflict.
	Indeed, the clauses $c_2,\ldots,c_\ell$ may force \UCP\ to assign truth values to $x_2,\ldots,x_\ell$, but no conflict can ensue because \UCP\ will ultimately satisfy $c_1$ by assigning appropriate truth values to the variables of $F_1''$.

	Thus, we may finally assume that all of $F_1',F_1'',\ldots,F_\ell',F_\ell''$ are empty.
	In other words, $F$ consists of the cycle $\cC$ only.
	Since $\cC$ is not toxic, {\bf TOX2} does not occur.
	Consequently, \UCP\ will construct an assignment that satisfies all clauses $c_1,\ldots,c_\ell$.
	This final contradiction implies (ii).
\end{proof}
	
\begin{corollary}\label{lem_tox_error_t}
	Fix $\eps>0$ and let $0\leq t\leq(1-\eps)n$.
	Then
	\begin{align*}
		\pr\brk{\cR_{t+1}}&=\pr\brk{\mbox{$\FUC{t+1}'$ contains a toxic cycle}}+o(1/n).
	\end{align*}
\end{corollary}
\begin{proof}
	This is an immediate consequence of \Lem~\ref{lemma_arnab} and \Lem~\ref{lem_tox}.
\end{proof}

	Thus, we are left to calculate the probability that $\FUC{t+1}'$ contains a toxic cycle.
	To this end, we estimate the number of toxic cycles in the `big' formula $\FUC t$.
	Let $\vT_{t}(\ell)$ be the number of toxic cycles of length $\ell$ in $\FUC t$.

	\begin{lemma}\label{lem_toxlt}
	Fix $\eps>0$ and let $1\leq t\leq(1-\eps)n$.
	\begin{enumerate}[(i)]
		\item For any fixed $\ell$, with probability $1-O(n^{-2})$ we have
		\begin{align*}
			\ex\brk{\vT_t\bc\ell\mid\fF_t}&=\beta_\ell+o(1),&&\mbox{where }\beta_\ell=\frac{1}{4\ell} \bc{ d(k-1) (1-\alpha_*)\alpha_*^{k-2} }^\ell = \frac 1 {4\ell} \bc{f_n(t)}^\ell.
		\end{align*}
	\item For any $1\leq\ell\leq n$, with probability $1-O(n^{-2})$ we have $\ex\brk{\vT_t\bc{\ell}\mid\fF_t}\leq\beta_\ell\exp(\eps\ell).$ \end{enumerate}
	\end{lemma}
	\begin{proof}
		In light of Fact~\ref{fact_deferred}, the calculation of the expected number of toxic cycles is straightforward.
		Indeed, we just need to pick sequences of $\ell$ distinct variables and clauses, place the variables into the clauses in a cyclic fashion, and multiply by the probability that the clauses contain no other variables and that the parity of the signs of the clauses works out as per {\bf TOX2}.
		Of course, in this way we over count toxic cycles $2\ell$ times (due to the choice of the starting point and the orientation).
		Hence, we obtain
		\begin{align}\label{eq_lem_toxlt_1}
			\Erw\brk{\vT_t(\ell) \mid \fF_t}  &= \frac{(n)_\ell (m)_{\ell}}{4\ell n^{2\ell}}  \bc{k(k-1)}^\ell \bc{1-\vec\alpha(t)}^{\ell} \vec\alpha(t)^{\ell(k-2)}.
		\end{align}
		Thus, (i) follows from \eqref{eq_lem_toxlt_1} and \Prop~\ref{prop_uc_alpha}.
		Further, \eqref{eq_lem_toxlt_1} demonstrates that
		\begin{align}\label{eq_lem_toxlt_2}
			\Erw\brk{\vT_t(\ell) \mid \fF_t}  &\leq \frac{1}{4\ell} \bc{ d(k-1) (1-\vec\alpha(t))\vec\alpha(t)^{k-2} }^\ell.
		\end{align}
		Finally, combining~\eqref{eq_lem_toxlt_2} with  \Prop~\ref{prop_uc_alpha}, we obtain (ii).
	\end{proof}
	
\begin{proof}[Proof of \Prop~\ref{prop_uc_error}]
	In light of \Cor~\ref{lem_tox_error_t} we just need to calculate the probability that $\FUC{t+1}'$ contains a toxic cycle.
	Clearly, if during iteration $t+1$ \UCP\ encounters a variable of $\FUC t$ that lies on a toxic cycle, \UCP\ will proceed to add the entire toxic cycle to $\FUC{t+1}'$ (and run into a contradiction).
	Furthermore, \Lem~\ref{lem_toxlt} shows that with probability $1-O(n^{-2})$ given $\fF_t$ the probability that a random variable of $\FUC t$ belongs to a toxic cycle comes to
	\begin{align}\label{eq_prop_uc_error_100}
		\bar\beta&=\sum_{\ell\geq2}\ell\beta_\ell+o(1)\mk{= \sum_{\ell \geq 2}  \frac 1 {4} \bc{f_n(t)}^\ell }= \frac{f_n(t)^2}{4(1-f_n(t))} +o(1)=O(1).
	\end{align}

	We now use \eqref{eq_prop_uc_error_100} to calculate the desired probability of encountering a toxic cycle.
	To this end we recall from the proof of \Lem~\ref{lem_condex} that the $(t+1)$-st iteration of \UCP\ corresponds to a branching process with expected offspring $f_n(t)$, unless the root variable $x_{t+1}$ has already been assigned.
	Due to~\eqref{eqqt+1} and \Prop~\ref{prop_uc_alpha}, with probability $1-O(n^{-2})$ the conditional probability of this latter event equals $(n\alpha_*-t)/(n-t)+o(1)$.
	Further, given that the root variable has not been assigned previously, the expected progeny of the branching process, i.e., the expected number of variables in $\FUC{t+1}'$, equals $1/(1-f_n(t))+o(1)$.
	Since with probability $1-O(n^{-2})$ given $\fF_t$ there remain $\vn(t)=(1-\alpha_*+o(1))n$ unassigned variables in total, \eqref{eq_prop_uc_error_100} implies that with probability $1-o(1/n)$,
	\begin{align*}
		\pr\brk{\cR_{t+1}\mid\fF_t}&\sim\frac{\bar\beta}{(1-\alpha_*)n}\cdot\frac{1-\alpha_*}{1-t/n}\cdot\frac1{1-f_n(t)}=\frac{f_n(t)^2}{4(1-f_n(t))^2(n-t)}+o(1/n),
	\end{align*}
	as claimed.
\end{proof}

\subsubsection{Proof of \Prop~\ref{prop_uc_pois}}\label{sec_prop_uc_pois}
	We combine Fact~\ref{fact_deferred} with the tower rule.
	Specifically, let $0\leq t_1<\cdots<t_h<(1-\eps)n$ be distinct time indices.
	Then repeated application of the tower rule gives
	\begin{align}\nonumber
		\Pr\brk{\bigcap_{i=1}^h \everrti}&= \Erw\brk{\prod_{i=1}^h \vecone\cbc{\everrti}} =  \Erw\brk{ \Erw\brk{\prod_{i=1}^h \vecone \cbc{\everrti} \mid \pastim } }\\
										 &= \Erw\brk{ \bc{\prod_{i=1}^{h-1} \vecone\cbc{\everrti}}  \pr\brk{\everrth \mid \pasthm } } = \cdots=\Erw\brk{\prod_{i=1}^{h} \pr\brk{\everrti \mid \pastim }  }.\label{eq_prop_uc_pois_1}
	\end{align}
	Furthermore, \Prop~\ref{prop_uc_error} shows that with probability $1-o(1/n)$,
	\begin{align}\label{eq_prop_uc_pois_2}
		\pr\brk{\everrti \mid \pastim }&= \frac{f_n(t_i)^2}{4(n-t_i)(1-f_n(t_i))^2}+o(1/n)&&\mbox{for all }1\leq i\leq h.
	\end{align}
Combining~\eqref{eq_prop_uc_pois_1}--\eqref{eq_prop_uc_pois_2} completes the proof.

\subsubsection{Proof of \Prop~\ref{prop_uc_endgame}}\label{sec_prop_uc_endgame}
Given $\delta>0$ pick $\eps>0$ small enough and let $t=\lceil(1-\eps)n\rceil$.
We are going to show that the graph $G(\FUC t)$ is acyclic with probability at least $1-\delta$.
Since all clauses of $\FUC t$ contain at least two variables, \UCP\ will find a satisfying assignment if $G(\FUC t)$ is acyclic.

To show that $G(\FUC t)$ is acyclic, we observe that $\alpha_*\geq t/n$.
Hence, $\alpha_*$ approaches one as $t/n\to1$.
Further, Fact~\ref{fact_deferred} shows that $G(\FUC t)$ is uniformly random given the degree distribution \eqref{eq_prop_uc_alpha} of the clause nodes.
Indeed, the expression \eqref{eq_prop_uc_alpha} shows that with probability $1-O(n^{-2})$ the expected size of the second neighbourhood of a given variable node is asymptotically equal to
	\begin{align*}
		\gamma=\gamma(\eps)&=\frac1{(1-\alpha_*)n}\cdot\frac{dn}{k}\sum_{\ell=2}^k\ell\binom k\ell(1-\alpha_*)^\ell\alpha_*^{k-\ell}=d(1-\alpha_*^{k-1}).
	\end{align*}
Hence, as $\lim_{\eps\to 0}\gamma=0$, the average degree of the random graph $G(\FUC t)$ tends to zero as $\eps\to0$.
Therefore, for small enough $\eps>0$ the random graph $\G(\FUC t)$ is acyclic with probability greater than $1-\delta$.

\subsection{Failure of \UCP\ for $\dmin<d<\dsat$}\label{sec_failure}
In this section we assume that $\dmin<d<\dsat$.
As in \Sec~\ref{sec_alg_pos} we are going to trace \UCP\ via the method of differential equations.
In particular, we keep the notation from \Sec~\ref{sec_alg_pos}.
Thus, $\vn(t)$ signifies the number of unassigned variables after $t$ iterations, and $\vm_\ell(t)$ denotes the number of clauses that contain precisely $2\leq\ell\leq k$ unassigned variables.
Moreover, $\FUC t$ is the formula comprising these variables and clauses.
The following statement is the analogue of \Prop~\ref{prop_uc_alpha} for $\dmin<d<\dsat$.
Its proof relies on similar arguments as the proof of \Prop~\ref{prop_uc_alpha}.

\begin{proposition}\label{prop_bpgd_cond}
	Suppose that $\dmin(k)<d<\dsat(k)$, fix $\eps,\delta>0$ and let $0<t<(1-\eps)\theta_*n$.
	Then \eqref{eq_prop_uc_alpha} holds with probability $1-O(n^{-2})$.
\end{proposition}
\begin{proof}
The formulas \eqref{eq_lem_condex_1}--\eqref{eq_lem_condex_3} for the conditional expected changes $\vn(t+1)-\vn(t),\vm_\ell(t+1)-\vm(t)$ continue to hold for $\dmin<d<\dsat$, so long as we assume that $2\vm_2(t)/\vn(t)<1-\Omega(1)$ and $\vn(t)=\Omega(n)$.
Indeed, the proof of \Lem~\ref{lem_condex} only hinges on these assumptions on $\vn(t),\vm_2(t)$, irrespective of $d$.
Hence, if $\fn,\fm_2,\ldots,\fm_k:[0,\theta_*-\delta] \mk{\to \mathbb{R}}$ are functions that satisfy the conditions \eqref{eqODE1}--\eqref{eqODE3} and that satisfy
\begin{align}\label{eqODE4'}
	\sup_{\theta\in[0,\theta_*-\delta]}2\fm_2(\theta)/\fn(\theta)&<1,
\end{align}
then~\cite[\Thm~2]{Wormald} implies that for all $0\leq t<(1-\delta)\theta_*n$ we have
\begin{align*}
	\vn(t)/n&=\fn(t/n)+o(1),& \vm_\ell(t)/n&=\fm_\ell(t/n)+o(1)\quad(2\leq\ell\leq k).
\end{align*}

Finally, we claim that the functions $\fn^*:[0,\theta_*-\delta]\to\RR$, $\fm_\ell^*:[0,\theta_*-\delta]\to\RR$ defined by \eqref{eq_lem_ODE} satisfy \eqref{eqODE1}--\eqref{eqODE3}  and \eqref{eqODE4'}.
In fact, the same manipulations as in the proof of \Lem~\ref{lem_ODE} yield \eqref{eqODE1}--\eqref{eqODE3}.
Additionally, \eqref{eqODE4'} follows from \Lem~\ref{lem_alphas}~(ii) and \Prop~\ref{prop_greg}~(ii), which shows that $\alpha_*$ is a stable fixed point and therefore 
\begin{align*}
	2\fm_2(\theta)/\fn(\theta)&=d(k-1)(1-\alpha_*)  \alpha_*^{k-2}<1 &&\mbox{for }0\leq\theta\leq\theta_*-\delta.
\end{align*}
Thus, we obtain \eqref{eq_prop_uc_alpha} for $0\leq \theta<\theta_*$.
\end{proof}

\begin{proof}[Proof of \Thm~\ref{thm_bpgd} (ii)]
	Let $\vu_1,\ldots,\vu_n\in\{0,1\}$ be uniformly distributed, mutually independent and independent of all other randomness.
	We couple the execution of the decimation process and of the \UCP\ algorithm on a random formula $\PHI$ as follows.
	At every time $t$ where $\pi_{\FDC t}=1/2$, the decimation process sets $\SIGDC(x_{t+1})=\vu_{t+1}$. 
	Similarly, whenever \UCP\ executes Step~5 we set $\SIGUC(x_{t+1})=\vu_{t+1}$.
	Let $\vec\Delta$ be the first time $0\leq t<n$ such that $\SIGDC(x_{t+1})\neq\SIGUC(x_{t+1})$; if $\SIGDC(x_{t+1})=\SIGUC(x_{t+1})$ for all $t$, we set $\vec\Delta=n$.

	We claim that $\UCP$ encounters a conflict if $\vec\Delta<n$.
	To see this, assume that $0\leq t<n$ satisfies $\SIGDC(x_{t+1})\neq\SIGUC(x_{t+1})$ but $\SIGDC(x_{s+1})\neq\SIGUC(x_{s+1})$ for all $0\leq s<t$ and that \UCP\ did not encounter a conflict at any time $s\leq t$.
	Then $\pi_{\FDC t}\in\{0,1\}$ but Step~5 of \UCP\ sets $\SIGUC(x_{t+1})=\vu_{t+1}\neq\SIGDC(x_{t+1})$.
	Consequently, $\PHI$ possesses no satisfying assignment $\sigma$ such that $\SIGUC(x_{i})=\sigma(x_i)$ for $1\leq i\leq t+1$, and thus \UCP\ will ultimately encounter a conflict.

	To complete the proof we claim that $\pr\brk{\vec\Delta<n}=1-o(1)$.
	To verify this consider a time $(1+\eps)\tcond<t/n<(1-\eps)\theta_*n$.
	Then \Prop~\ref{prop_greg} and \Prop~\ref{cor_frz} show that $|\frz(\FDC{t})|=\alpha^* n+o(n)$ \whp, while \Prop~\ref{prop_bpgd_cond} shows that $\vec\alpha(t)=\alpha_*+o(1)$ \whp\
	In particular, even if $\vec\Delta\geq(1+\eps)\tcond$, the probability that $\pi_{\FDC t}\in\{0,1\}$ while $\UCP$ assigns $x_{t+1}$ randomly is $\Omega(1)$.
	Therefore, $\vec\Delta<\theta_*n$ \whp
\end{proof}

\section*{Acknowledgement}

\thanks{Amin Coja-Oghlan is supported by DFG CO 646/3, DFG CO 646/5 and DFG CO 646/6.}
\thanks{Lena Krieg is supported by DFG CO 646/3.}
\thanks{Maurice Rolvien is supported by DFG Research Group ADYN (FOR 2975) under grant DFG 411362735.}
\thanks{This research was funded in part by the Austrian Science Fund (FWF) [10.55776/I6502]. For open access purposes, the authors have applied a CC BY public copyright license to any author accepted manuscript version arising from this submission.}

\bibliography{jdec}

\end{document}